\documentclass[10pt,a4paper]{amsart}
\usepackage[foot]{amsaddr}
\usepackage[lmargin=0.9in,rmargin=0.9in]{geometry}
\usepackage[T1]{fontenc}
\usepackage[utf8]{inputenc}
\usepackage[british]{babel}
\usepackage{mathtools}
\usepackage{amsthm}
\usepackage{lmodern}
\usepackage{calrsfs}
\usepackage{amssymb}
\usepackage{amsmath, calligra, mathrsfs}
\usepackage[mathscr]{euscript}
\usepackage{enumitem}
\usepackage{tikz-cd}
\usetikzlibrary{decorations.markings}
\usepackage{hyperref}
\usepackage[noabbrev,nameinlink]{cleveref}
\theoremstyle{plain}
\newtheorem{thm}{Theorem}[section]
\newtheorem*{thm*}{Theorem}
\newtheorem{lem}[thm]{Lemma}
\newtheorem{prop}[thm]{Proposition}
\newtheorem{cor}[thm]{Corollary}
\theoremstyle{definition}

\newtheorem{nota}[thm]{Notation}
\newtheorem{ex}[thm]{Example}

\newtheorem{hyp}[thm]{Hypothesis}

\theoremstyle{remark}
\newtheorem{rem}[thm]{Remark}
\Crefname{thm}{Theorem}{Theorems}
\Crefname{lm}{Lemma}{Lemmata}
\Crefname{prop}{Proposition}{Propositions}
\Crefname{cor}{Corollary}{Corollaries}
\Crefname{hyp}{Hypothesis}{Hypotheses}
\Crefname{q}{Question}{Questions}
\Crefname{defn}{Definition}{Definitions}
\Crefname{nota}{Notation}{Notations}
\Crefname{ex}{Example}{Examples}
\Crefname{xca}{Exercise}{Exercises}
\Crefname{rem}{Remark}{Remarks}
\Crefname{constr}{Construction}{Constructions}

\newcommand{\N}{\mathbb{N}}

\newcommand{\Q}{\mathbb{Q}}

\newcommand{\Ccal}{\mathcal{C}}
\newcommand{\Dcal}{\mathcal{D}}

\newcommand{\Gcal}{\mathcal{G}}

\newcommand{\Rcal}{\mathcal{R}}
\newcommand{\Scal}{\mathcal{S}}
\newcommand{\Ucal}{\mathcal{U}}

\newcommand{\id}{\textup{id}}
\newcommand{\Hom}{\textup{Hom}}
\newcommand{\unit}{\mathbf{1}}
\newcommand{\Hbb}{\mathbb{H}}

\newcommand{\conn}{\mathsf{conn}}
\newcommand{\Fact}{\mathsf{Fact}}
\newcommand{\Emb}{\mathsf{Emb}}
\newcommand*{\isoarrow}[1]{\arrow[#1,"\rotatebox{90}{\(\sim\)}"]}
\newcommand{\Var}{\mathsf{Var}}
\newcommand{\Sm}{\mathsf{Sm}}
\newcommand{\Perv}{\mathsf{Perv}}

\newcommand\blfootnote[1]{%
	\begingroup
	\renewcommand\thefootnote{}\footnote{#1}%
	\addtocounter{footnote}{-1}%
	\endgroup
}

\usepackage{changepage} 

\title{Extending monoidal structures on fibered categories via embeddings}
\author[Luca Terenzi]{Luca Terenzi}
\address{Luca Terenzi \newline
	\indent UMPA, ENS de Lyon \newline
	\indent 46 Allée d'Italie, 69007 Lyon (France)}
\email{\normalfont\href{mailto:luca.terenzi@ens-lyon.fr}{luca.terenzi@ens-lyon.fr}}

\makeatletter
\hypersetup{
	pdfauthor={\author},
	pdftitle={\@title},
	colorlinks,
	linkcolor=[rgb]{0.2,0.2,0.6},
	citecolor=[rgb]{0.2,0.6,0.2},
	urlcolor=[rgb]{0.6,0.2,0.2}}
\makeatother

\calclayout
\begin{document}

\maketitle

\begin{abstract}
	Let $\Scal$ be a small category, and suppose that we are given a full subcategory $\Ucal$ such that every object of $\Scal$ can be embedded into some object of $\Ucal$ in the same way as every quasi-projective algebraic variety admits a closed embedding into a smooth one. 
	We show that every monoidal structure on a given $\Scal$-fibered category satisfying certain natural conditions is completely determined by its restriction to $\Ucal$; in fact, any monoidal structure over $\Ucal$ satisfying similar natural conditions admits an essentially unique extension to the whole of $\Scal$. In this way, one can recover the unit constraint on the classical constructible derived categories over quasi-projective algebraic varieties from the abelian categories of perverse sheaves.
	The same principle applies to morphisms of $\Scal$-fibered categories and monoidality thereof. This allows one to fully understand the unitarity properties of the natural realization of M. Saito's categories of mixed Hodge modules into the underlying categories of perverse sheaves.
\end{abstract}

\tableofcontents

\blfootnote{\textit{\subjclassname}. 18D30, 18M05.}
\blfootnote{\textit{\keywordsname}. Monoidal fibered categories, embeddings, perverse sheaves.}

\blfootnote{The author acknowledges support by the GK 1821 "Cohomological Methods in Geometry" at the University of Freiburg and by the Labex Milyon at the ENS Lyon.}

\section*{Introduction}

\subsection*{Motivation and goal of the paper}

In Algebraic Geometry, studying smooth varieties often proves much easier than studying singular varieties: this basically happens every time one tries to study singular varieties via algebraic differential forms, which are well-behaved only in the smooth case. For instance, it is unreasonable to define algebraic de Rham cohomology of singular varieties via the usual algebraic de Rham complex, since this would not produce cohomology groups of the expected dimension; see \cite[Ex.~4.4]{AraKang} for an explicit counterexample. More generally, the behavior of algebraic differential equations over a singular complex variety $S$ is quite subtle both from the perspective of $\mathcal{D}$-modules and from that of local systems: on the side of $\Dcal$-modules, a typical pathology is that the ring $\mathcal{D}(S)$ of algebraic differential operators over $S$ is not necessarily Noetherian, so it is unreasonable to define $\mathcal{D}$-modules over $S$ naively as modules over $\mathcal{D}(S)$; a typical issue on the side of local systems is that an ordinary local system over $S$ is not a perverse sheaf (up to shiftings), so it does not arise from a regular holonomic $\Dcal$-module over $S$. 

When the singular variety $S$ is quasi-projective, a useful workaround is to choose a closed embedding $s: S \hookrightarrow Y$ into a smooth (still quasi-projective) variety $Y$: in the $\Dcal$-module setting, one is led to define the category of $\Dcal$-modules over $S$ as the subcategory of those $\Dcal$-modules over $Y$ which vanish on the open complement $Y \setminus S$; in the local system setting, one can study those local systems over $S$ that extend to local systems over $Y$ in this way. In both cases, the main point is to check that the sought-after categories or objects are essentially independent of the choice of the specific closed embedding $s: S \hookrightarrow Y$, and that they enjoy the expected functoriality with respect to morphisms of possibly singular varieties $f: T \rightarrow S$.

The success of this approach is related to - and, a posteriori, explained by - the existence of M. Saito's six functor formalism of \textit{mixed Hodge modules} over quasi-projective complex varieties, developed in \cite{Sai90}. In fact, for every coefficient system $\Hbb$ over quasi-projective varieties satisfying the six functor formalism (as formalized, for example, in J. Ayoub's thesis \cite{Ayo07a,Ayo07b}), a closed embedding $s: S \hookrightarrow Y$ determines an adjoint pair of the form
\begin{equation*}
	s^*: \Hbb(Y) \rightleftarrows \Hbb(S) :s_*
\end{equation*}
in which the direct image functor $s_*$ is fully faithful while the inverse image functor $s^*$ is monoidal. This allows one to identify objects of $\Hbb(S)$ with the objects of $\Hbb(Y)$ having support in $S$, and also to identify the unit object $\unit_S \in \Hbb(S)$ with the pull-back $s^* \unit_Y$. Combining these two facts, one is able to reconstruct the unit constraint over $S$
\begin{equation*}
	u_S: \unit_S \otimes A \xrightarrow{\sim} A
\end{equation*}
from the unit constraint over $Y$, by the formula
\begin{equation*}
	s^* \unit_Y \otimes A = s^* \unit_Y \otimes s^* s_* A \xrightarrow{m_s} s^* (\unit_Y \otimes s_* A) \xrightarrow{u_Y} s^* s_* A = A.
\end{equation*}
A similar strategy works in presence of a morphism of coefficient systems $R: \Hbb_1 \rightarrow \Hbb_2$. In particular, in order to make sure that a given morphism is compatible with unit constraints over all quasi-projective varieties, it suffices to check its unitarity over smooth varieties.

The main goal of the present paper is to turn the embedding principle just sketched into a rigorous general method applicable to the whole theory of monoidal categories over quasi-projective algebraic varieties. In particular, we want to exploit the embedding method to address the problem of extending monoidal structures, morphisms of fibered categories, and monoidality thereof.

This work is part of a series of articles whose general motivation is to study the abstract monoidality properties of coefficient systems modeled on perverse sheaves. To fix notation, let $k$ be a subfield of the complex numbers and consider the category $\Var_k$ of quasi-projective algebraic $k$-varieties. To every $S \in \Var_k$ one can attach its algebraically constructible derived category, denoted $D^b_c(S,\Q)$. As $S$ varies in $\Var_k$, these categories assemble into a monoidal $\Var_k$-fibered category, and this structure actually extends to a complete six functor formalism. The classical theory of \cite{BBD82} yields a perverse $t$-structure on $D^b_c(S,\Q)$, whose heart is the abelian category $\Perv(S)$ of \textit{perverse sheaves}. As shown by Beilinson in \cite{Bei87}, the triangulated category $D^b_c(S,\Q)$ can be recovered as the bounded derived category $D^b(\Perv(S))$. In light of Beilinson's result, it is natural to ask whether the six functor formalism can be reconstructed on the level of perverse sheaves. In our articles \cite{Ter23Fib} and \cite{Ter23Fact} we address this problem systematically, especially for what concerns the monoidal structure. Our techniques yield a satisfying description of the constructible derived categories as a monoidal $\Var_k$-fibered category in terms of functors respecting the perverse $t$-structures. The only question that is left open is how to describe the unit constraint in terms of perverse sheaves: the fundamental issue, as alluded to above, is that the unit object $\unit_S \in D^b_c(S,\Q)$ is not a perverse sheaf (up to shiftings) unless $S$ is smooth. The techniques introduced in the present paper solve this problem.

The issues just mentioned, although rather formal-looking in the setting of usual perverse sheaves, become quite relevant for more refined coefficient systems such as mixed Hodge modules. Indeed, Saito's triangulated categories are defined directly as the bounded derived categories of their perverse hearts, and the six functors on mixed Hodge modules are obtained by cleverly combining various exact functors defined on the level of abelian categories. The same issues become even more serious in the emerging theory of \textit{perverse Nori motives}, introduced by F. Ivorra and S. Morel in \cite{IM19}. This consists of a system of abelian categories of motivic perverse sheaves defined in terms of a purely category-theoretic universal property. One cannot describe objects and morphisms in these categories in a concrete way; the same applies, a fortiori, to functors and natural transformations relating their derived categories. We ultimately want to apply our abstract results to the construction of a monoidal structure on perverse Nori motives, which we achieve in \cite{Ter23Nori} by working directly on the perverse hearts. 
The present work is exploited mostly in the construction of the motivic unit constraint, which represents a fundamental part of the monoidal structure. We plan to apply the present results to a larger extent in order to study Ivorra's perverse motivic realization functors, defined in \cite{IvReal}. Without going into the technical details, let us just mention that these realization functors are only available over smooth varieties; again, the obstruction to defining them in the singular case lies in the lack of perversity of the unit object. In principle, our results would allow us to extend the realization to possibly singular varieties automatically, on the condition of establishing the compatibility of Ivorra's construction with the six functors. We hope to be able to say more about this problem in future work.

\subsection*{Main results}

The basic principle behind our constructions is that a possibly singular quasi-projective variety $S$ admits a closed embedding $s: S \hookrightarrow Y$ into a smooth quasi-projective variety $Y$, which allows one to define structures over $S$ from similar structures over $Y$.
We work in a general axiomatic framework where we can speak about closed immersions and smooth morphisms in a purely categorical setting. Such a framework was introduced in the companion paper \cite{Ter23Fact}; here, we need to strengthen its assumptions slightly in order to make our abstract embedding method work. We also introduce suitable axioms on a (possibly non-full) subcategory $\Ucal$ of $\Scal$ mimicking some basic properties of smooth varieties: they guarantee the existence of a closed embedding from every object of $\Scal$ into some object of $\Ucal$. 

Our main results only apply to a certain class of triangulated $\Scal$-fibered categories that partly satisfy the axiomatics of stable homotopy $2$-functors studied by Ayoub in \cite{Ayo07a,Ayo07b} and in \cite{Ayo10}: following \cite[\S~3]{Ter23Fact}, we call them \textit{localic $\Scal$-fibered categories}, and the only morphisms of $\Scal$-fibered categories that we allow between any two such are called \textit{localic morphisms}.

Our first main result gives a way to recover localic morphisms of localic $\Scal$-fibered categories from their restrictions to the underlying $\Ucal$-fibered categories. It can be stated as follows:

\begin{thm*}[\Cref{thm_ext}]
	Let $R: \Hbb_1 \rightarrow \Hbb_2$ be a localic morphism between localic $\Scal$-fibered categories. Then the morphism $R$ is completely determined by its restriction to $\Ucal$. Conversely, every localic morphism between the underlying localic $\Ucal$-fibered categories extends uniquely to the whole of $\Scal$.
\end{thm*}
\noindent
The proof is based on some auxiliary constructions that exploit the functoriality properties of closed embeddings into $\Ucal$ in a canonical and explicit way. 

Following the same route, we obtain analogous results in the setting of monoidal structures. Again, our constructions only apply to those monoidal structures on localic $\Scal$-fibered categories which satisfy the so-called \textit{projection formulae} with respect to smooth morphisms as formulated in \cite{Ayo07a}: following \cite[\S~5]{Ter23Fact}, we call them \textit{localic tensor structures}. Our second main result can be stated as follows:

\begin{thm*}[\Cref{thm:ext-otimes}, \Cref{prop:asso_int-ext}, \Cref{prop:comm_int-ext}, \Cref{prop:ext-unit}, \Cref{lem:ac-comp_int_ext}]
	Let $\Hbb$ be a localic $\Scal$-fibered category equipped with a localic (unitary, symmetric, associative) monoidal structure. Then the monoidal structure on $\Hbb$ is completely determined by its restriction to $\Ucal$. Conversely, any localic (unitary, symmetric, associative) monoidal structure on the underlying $\Ucal$-fibered category uniquely extends to the whole of $\Scal$.
\end{thm*}

Finally, we combine the two extension methods just discussed to study monoidality of localic morphisms between localic $\Scal$-fibered categories in a similar way. Our third main result can be stated as follows:

\begin{thm*}[\Cref{thm_ext-rho}, \Cref{lem:ext-acu-rho}]
	Let $\Hbb_1$ and $\Hbb_2$ be two localic $\Scal$-fibered categories equipped with localic monoidal structures, and let $R: \Hbb_1 \rightarrow \Hbb_2$ be a localic morphism equipped with a monoidal structure (with respect to the given monoidal structures on $\Hbb_1$ and $\Hbb_2$). Then the monoidal structure on $R$ is completely determined by its restriction to $\Ucal$. Conversely, any monoidal structure on the underlying morphism of $\Ucal$-fibered categories extends uniquely to the whole of $\Scal$.
\end{thm*}

We give complete proofs in the case of morphisms of $\Scal$-fibered categories, which already contains all the essential ingredients. On the other hand, we prefer to omit the proofs in the case of monoidal structures on single fibered categories, since they are just notationally more intricate variants of the previous ones; we give more details in the case of associativity, commutativity and unit constraints, and also in the case of monoidal structures on morphisms.

\subsection*{Structure of the paper}

Throughout the paper, we work over a fixed small category $\Scal$; at the beginning of each section, we specify the natural conditions that $\Scal$ has to satisfy in order to make our constructions possible.

In \Cref{sect:ext-sect}, after recalling our basic conventions and notation about $\Scal$-fibered categories from \cite{Ter23Fib}, we introduce the main working hypotheses on $\Scal$ (\Cref{hyp:Fact}) following the approach of \cite[\S\S~2, 3]{Ter23Fact}, and we define the full subcategory $\Ucal$ accordingly (\Cref{nota:Ucal}); we also spell out the minimal set of properties of $\Ucal$ needed in order for our embedding method to be applied successfully (\Cref{lem:Ucal}). We then establish a first extension result concerning sections of $\Scal$-fibered categories (\Cref{prop:ext-sect}); the overall structure of our extension method is already visible in the proof of this result, which however is technically simpler compared to those for morphisms of $\Scal$-fibered categories and for monoidal structures presented in the subsequent sections.

Starting from \Cref{sect:ext-mor}, we need to introduce more assumptions on the base category $\Scal$ (\Cref{hyp:open}) as well as more restrictive hypotheses on the $\Scal$-fibered categories considered: the natural setting for our results is that of \textit{localic $\Scal$-fibered categories} and \textit{localic morphisms} of such, as axiomatized in \cite[\S~3]{Ter23Fact}. Our first main result (\Cref{thm_ext}) asserts that every localic morphisms between localic $\Scal$-fibered categories is canonically determined by its restriction to $\Ucal$ and that, conversely, any such morphism over $\Ucal$ uniquely extends to the whole of $\Scal$. We also discuss how this extension result applies to images of sections along morphisms of $\Scal$-fibered categories (\Cref{lem:ext-sect-comp}).

In \Cref{sect_ext-boxtimes} we explain how to extend monoidal structures on $\Scal$-fibered categories in a similar fashion, using the language of \textit{internal tensor structures} introduced in \cite{Ter23Fib}. Again, we need to restrict ourselves to the setting of localic $\Scal$-fibered categories and only consider \textit{localic internal tensor structures} as done in \cite[\S~5]{Ter23Fact}. We obtain the expected extension result (\Cref{thm:ext-otimes}) by suitably adapting our construction for morphisms of $\Scal$-fibered categories to the new setting. We then complete the discussion by explaining how to extend associativity constraints (\Cref{prop:asso_int-ext}), commutativity constraints (\Cref{prop:comm_int-ext}) and unit constraints (\Cref{prop:ext-unit}) in such a way that their natural mutual compatibility conditions are respected (\Cref{lem:ac-comp_int_ext}).

In the final \Cref{sect:ext-rho} we study monoidality of morphisms between monoidal $\Scal$-fibered categories in a similar way; as usual, we consider the setting of localic morphisms between localic $\Scal$-fibered categories. In the first place, we prove a general extension result for internal tensor structure on morphisms (\Cref{thm_ext-rho}). Then, we check that this operation respects the natural notions of associativity, symmetry and unitarity for such monoidal structures (\Cref{lem:ext-acu-rho}). 

\subsection*{Acknowledgments}

The contents of this paper correspond to the third chapter of my Ph.D. thesis, written at the University of Freiburg under the supervision of Annette Huber-Klawitter. It is a pleasure to thank her for many useful discussions around the subject of this article, as well as for her constant support and encouragement.

\section*{Notation and conventions}

\begin{itemize}
	\item Unless otherwise dictated, categories are assumed to be small with respect to some fixed universe.
	\item Given a category $\Ccal$, the notation $C \in \Ccal$ means that $C$ is an object of $\Ccal$.
	\item Given categories $\Ccal_1, \dots, \Ccal_n$, we let $\Ccal_1 \times \cdots \times \Ccal_n$ denote their direct product category.
\end{itemize}

\section{Extending sections of fibered categories}\label{sect:ext-sect}

Throughout this paper, we work over a fixed small category $\Scal$ admitting finite products; we let $\ast$ denote the final object of $\Scal$. 

In the first part of the present section, while recalling our general conventions and notation about $\Scal$-fibered categories, we introduce a convenient set of axioms on $\Scal$ mimicking certain natural properties of quasi-projective algebraic varieties over a field; all our extension results will be stated uniformly in terms of this axiomatic framework. In the second part of the section, we describe a first example of the extension procedure in the case of sections of $\Scal$-fibered categories.

\subsection{Recollections and notation}

Following the approach of the companion paper \cite{Ter23Fact}, we introduce suitable notions of smooth morphisms and closed immersions in $\Scal$ mimicking the properties of the usual classes of smooth morphisms and closed immersions between quasi-projective algebraic varieties over a field; in particular, this will allow us to factor morphisms of $\Scal$ into closed immersions and smooth morphisms. To be precise, we will work under the following hypothesis, which is \cite[Hyp.~2.1]{Ter23Fact}:

\begin{hyp}\label{hyp:Fact}
	We assume to be given two (possibly non-full) subcategories $\Scal^{sm}$ and $\Scal^{cl}$ of $\Scal$ satisfying the following conditions:
	\begin{enumerate}
		\item[(i)] Both $\Scal^{sm}$ and $\Scal^{cl}$ have the same objects as $\Scal$ and contain all isomorphisms of $\Scal$. 
		\item[(ii)] Let $p: P \rightarrow S$ be a morphism in $\Scal^{sm}$. Then, for every morphism $f: T \rightarrow S$ in $\Scal$, the fibered product $P_T := P \times_S T$ exists in $\Scal$. Moreover, for every Cartesian square of the form
		\begin{equation*}
			\begin{tikzcd}
				P_T \arrow{r}{f'} \arrow{d}{p'} & P \arrow{d}{p} \\
				T \arrow{r}{f} & S
			\end{tikzcd}
		\end{equation*} 
		with $p$ in $\Scal^{sm}$, the morphism $p'$ belongs to $\Scal^{sm}$ as well. 
		\item[(iii)] For every Cartesian square in $\Scal$ of the form
		\begin{equation*}
			\begin{tikzcd}
				P_Z \arrow{r}{z'} \arrow{d}{p'} & P \arrow{d}{p} \\
				Z \arrow{r}{z} & S
			\end{tikzcd}
		\end{equation*}
		with $p$ in $\Scal^{sm}$ and $z$ in $\Scal^{cl}$, the morphism $z'$ belongs to $\Scal^{cl}$ as well. Moreover, for every commutative diagram in $\Scal$ of the form
		\begin{equation*}
			\begin{tikzcd}
				Z \arrow{r}{z} \arrow[bend right]{dr}{gz} & S \arrow{d}{g} \\
				& V
			\end{tikzcd}
		\end{equation*}
		with $gz$ in $\Scal^{cl}$, the morphism $z$ belongs to $\Scal^{cl}$ as well.
		\item[(iv)] Every morphism $f: T \rightarrow S$ in $\Scal$ admits a factorization of the form
		\begin{equation*}
			f: T \xrightarrow{t} P \xrightarrow{p} S
		\end{equation*}
		with $t$ in $\Scal^{cl}$ and $p$ in $\Scal^{sm}$.
	\end{enumerate}
	We say that a morphism in $\Scal$ is a \textit{smooth morphism} if it belongs to $\Scal^{sm}$ and a \textit{closed immersion} if it belongs to $\Scal^{cl}$.
\end{hyp}

As done in \cite[Not.~2.7]{Ter23Fact}, we are naturally led to introduce, for every morphism $f: T \rightarrow S$ in $\Scal$, the associated \textit{factorization category} $\Fact_{\Scal}(f)$ where
\begin{itemize}
	\item objects are triples $(P;t,p)$ where $P$ is an object of $\Scal$, $t: T \rightarrow P$ is a closed immersion and $p: P \rightarrow S$ is a smooth morphism such that $p \circ t = f$,
	\item a morphism $q: (P';t',p') \rightarrow (P;t,p)$ is the datum of a smooth morphism $q: P' \rightarrow P$ such that $q \circ t' = t$ and $p \circ q = p'$.
\end{itemize}
By \cite[Lemma~2.8]{Ter23Fact}, the category $\Fact_{\Scal}(f)$ is non-empty and connected. More precisely, every two objects $(P;t,p)$ and $(P';t',p')$ in $\Fact_{\Scal}(f)$ are dominated by a third one $(P'';t'',p'')$: just take $P'' = P \times_S P'$ with the obvious morphisms $t'': T \rightarrow P''$ and $p'': P'' \rightarrow S$.

The general purpose of the present paper is to explain how many natural constructions on $\Scal$-fibered categories can be canonically recovered from their restriction to a suitable subcategory $\Ucal$ of $\Scal$ playing the role of the subcategory of smooth objects in $\Scal$. The natural choice seems the following:

\begin{nota}\label{nota:Ucal}
	With respect to the subcategory $\Scal^{sm}$ introduced in \Cref{hyp:Fact}, we let $\Ucal$ denote the full subcategory of $\Scal$ collecting all those objects $Y \in \Scal$ such that the canonical morphism to the final object $Y \rightarrow \ast$ belongs to $\Scal^{sm}$.
\end{nota}

The basic stability properties of $\Ucal$ inside of $\Scal$ can be summarized as follows:

\begin{lem}\label{lem:Ucal}
	With the above notation, the following statements hold:
	\begin{enumerate}
		\item $\Ucal$ is closed under finite products in $\Scal$. 
		\item For every $Y \in \Ucal$ and every smooth morphism $p: X \rightarrow Y$ we have $X \in \Ucal$.
		\item For every $S \in \Scal$ there exists a closed immersion $s: S \rightarrow Y$ with $Y \in \Ucal$.
	\end{enumerate}
\end{lem}
\begin{proof}
	\begin{enumerate}
		\item Given two objects $Y, Y' \in \Ucal$, form the Cartesian diagram in $\Scal$
		\begin{equation*}
			\begin{tikzcd}
				Y \times Y' \arrow{r} \arrow{d} & Y \arrow{d} \\
				Y' \arrow{r} & \ast.
			\end{tikzcd}
		\end{equation*} 
	    Since $Y' \in \Ucal$, by \Cref{hyp:Fact}(ii), the projection $Y \times Y' \rightarrow Y$ belongs to $\Scal^{sm}$. Since moreover $Y \in \Ucal$, the composite map $Y \times Y' \rightarrow Y \rightarrow \ast$ belongs to $\Scal^{sm}$ as well.
	    \item This is clear.
	    \item This follows applying \Cref{hyp:Fact}(iv) to the unique morphism $S \rightarrow \ast$.
	\end{enumerate}
\end{proof}

A further list of hypotheses about $\Scal$ and $\Ucal$, which we will need from \Cref{sect:ext-mor} on, is spelled out in detail in \Cref{hyp:open}.

\begin{rem}\label{rem:Usmcl}
	Note that the small category $\Ucal$ of \Cref{nota:Ucal}, together with its natural subcategories $\Ucal^{sm} := \Ucal \cap \Scal^{sm}$ and $\Ucal^{cl} := \Ucal \cap \Scal^{cl}$, satisfies again the conditions of \Cref{hyp:Fact}: indeed, (i) follows from the fullness of $\Ucal$ in $\Scal$; (ii) and (iii) both follow from from the fact that, given two smooth morphisms in $\Ucal$, their fibered product in $\Scal$ belongs in fact to $\Ucal$ by \Cref{lem:Ucal}(2); finally, (iv) follows from the fact that, for every morphism $f: T \rightarrow S$ in $\Ucal$ and every factorization $(X;t,p) \in \Fact_{\Scal}(f)$, the object $X \in \Scal$ belongs in fact to $\Ucal$, once again by \Cref{lem:Ucal}(2). 
\end{rem}

\begin{ex}\label{ex:VarSm}
	As in \cite{Ter23Fact}, we are interested in applying \Cref{hyp:Fact} when $\Scal = \Var_k$ is the category of quasi-projective algebraic varieties over a field $k$, and $\Scal^{sm}$ and $\Scal^{cl}$ denote the usual subcategories of smooth morphisms and closed immersions, respectively. In this setting, we want to take $\Ucal = \Sm_k$ to be the category of smooth quasi-projective $k$-varieties.
\end{ex}

Throughout this section, we fix an \textit{$\Scal$-fibered category} $\Hbb$: as in \cite[\S~1]{Ter23Fib}, by this we mean the datum of
\begin{itemize}
	\item for every $S \in \Scal$, a category $\Hbb(S)$,
	\item for every morphism $f: T \rightarrow S$ in $\Scal$, a functor 
	\begin{equation*}
		f^*: \Hbb(S) \rightarrow \Hbb(T),
	\end{equation*}
    called the \textit{inverse image functor} along $f$,
	\item for every pair of composable morphisms $f: T \rightarrow S$ and $g: S \rightarrow V$, a natural isomorphism of functors $\Hbb(V) \rightarrow \Hbb(T)$
	\begin{equation*}
		\conn = \conn_{f,g}: (gf)^* A \xrightarrow{\sim} f^* g^* A
	\end{equation*}
	called the \textit{connection isomorphism} at $(f,g)$
\end{itemize}
such that the following conditions are satisfied:
\begin{enumerate}
	\item[($\Scal$-fib-0)] For every $S \in \Scal$, we have $\id_S^* = \id_{\Hbb(S)}$.
	\item[($\Scal$-fib-1)] For every triple of composable morphisms $f: T \rightarrow S$, $g: S \rightarrow V$ and $h: V \rightarrow W$ in $\Scal$, the diagram of functors $\Hbb(W) \rightarrow \Hbb(T)$
	\begin{equation*}
		\begin{tikzcd}
			(hgf)^* A \arrow{rr}{\conn_{f,hg}} \arrow{d}{\conn_{gf,h}} && f^* (hg)^* \arrow{d}{\conn_{g,h}} A \\
			(gf)^* h^* A \arrow{rr}{\conn_{f,g}} && f^* g^* h^* A
		\end{tikzcd}
	\end{equation*}
	is commutative.
\end{enumerate}

\begin{rem}
	By \cite[Lemma.~1.3]{Ter23Fib}, axiom ($\Scal$-fib-1) implies that connection isomorphisms in $\Hbb$ can be treated as equalities. In order to simplify notation, starting from the next subsection we will write all connection isomorphisms with an equality sign; also, we will use the result of \cite[Lemma.~1.3]{Ter23Fib} in the proofs of various results in the course of this paper.
\end{rem}

As in \cite[\S~1]{Ter23Fib}, by a \textit{section} $A$ of $\Hbb$ over $\Scal$ we mean the datum of
\begin{itemize}
	\item for every $S \in \Scal$, an object $A_S \in \Hbb(S)$,
	\item for every morphism $f: T \rightarrow S$ in $\Scal$, an isomorphism in $\Hbb(T)$
	\begin{equation*}
		A^* = A_f^*: f^* A_S \xrightarrow{\sim} A_T
	\end{equation*}
\end{itemize}
satisfying the following condition:
\begin{enumerate}
	\item[($\Scal$-sect)] For every pair of composable morphisms $f: V \rightarrow T$ and $g: T \rightarrow S$ in $\Scal$, the diagram in $\Hbb(V)$
	\begin{equation*}
		\begin{tikzcd}
			f^* g^* A_S \arrow{r}{A_g^*} \arrow{d}{\conn_{f,g}} & f^* A_T \arrow{d}{A_f^*} \\
			(gf)^* A_S \arrow{r}{A_{gf}^*} & A_V \\
		\end{tikzcd}
	\end{equation*}
	is commutative.
\end{enumerate}
Similarly, by a \textit{morphism of sections} $w: A \rightarrow B$ of $\Hbb$ over $\Scal$ we mean the datum of
\begin{itemize}
	\item for every $S \in \Scal$, a morphism in $\Hbb(S)$
	\begin{equation*}
		w_S: A_S \rightarrow B_S
	\end{equation*}
\end{itemize}
satisfying the following condition:
\begin{enumerate}
	\item[(mor-$\Scal$-sect)] For every morphism $f: T \rightarrow S$ in $\Scal$, the diagram in $\Hbb(T)$
	\begin{equation*}
		\begin{tikzcd}
			f^* A_S \arrow{r}{w_S} \arrow{d}{A^*_f} & f^* B_S \arrow{d}{B^*_f} \\
			A_T \arrow{r}{w_T} & B_T
		\end{tikzcd}
	\end{equation*}
	is commutative.
\end{enumerate}
Note that the $\Scal$-fibered category $\Hbb$ defines by restriction a $\Ucal$-fibered category, which we still write simply as $\Hbb$ by abuse of notation. Similarly, every section of $\Hbb$ over $\Scal$ defines a section of $\Hbb$ over $\Ucal$ by restriction; the same applies to morphisms of sections.

\subsection{Extension of sections}

We can now show that the restriction operation from $\Scal$ to $\Ucal$ for sections of $\Hbb$ is reversible in a canonical way. Here is the precise statement:

\begin{prop}\label{prop:ext-sect}
	Let $A$ be a section of $\Hbb$ over $\Ucal$. Then there exists a unique (up to unique isomorphism) section $\tilde{A}$ of $\Hbb$ over $\Scal$ whose restriction to $\Ucal$ coincides with $A$. Moreover, the assignment $A \rightsquigarrow \tilde{A}$ is compatible with morphisms of section.
\end{prop}

In order to make our construction precise, it is convenient to introduce some more notation:

\begin{nota}\label{nota:Emb}
	Fix $S \in \Scal$.
	\begin{enumerate}
		\item We introduce the \textit{embedding category} $\Emb_{\Ucal}(S)$ where
		\begin{itemize}
			\item objects are pairs $(Y;s)$ with $Y \in \Ucal$ and $s: S \rightarrow Y$ a closed immersion,
			\item a morphism $p: (Y';s') \rightarrow (Y;s)$ is a smooth morphism $p: Y' \rightarrow Y$ such that $p \circ s' = s$.
		\end{itemize}
		\item We let $\Gcal_{\Ucal}(S)$ denote the associated trivial groupoid: it is the category whose objects are the same as for $\Emb_{\Ucal}(S)$ and where each homomorphism set contains exactly one element.
	\end{enumerate} 
\end{nota}

As a sanity check, let us note the following:

\begin{lem}\label{lem:Emb}
	For every $S \in \Scal$, the category $\Emb_{\Ucal}(S)$ is non-empty and connected.
\end{lem}
\begin{proof}
	The fact that $\Emb_{\Ucal}(S)$ is non-empty follows from \Cref{lem:Ucal}(3). To see that it is connected, we prove a more precise result, namely: every two objects $(Y;s)$ and $(Y';s')$ are dominated by a third one $(Y'';s'')$. Indeed, it suffices to take $Y'' = Y \times Y'$ (which belongs to $\Ucal$ by \Cref{lem:Ucal}(1)) with the canonical projections to $Y'$ and $Y''$ (which are smooth morphisms by \Cref{hyp:Fact}(ii)) and with the obvious morphism $s'': S \rightarrow Y''$ (which is a closed immersion by \Cref{hyp:Fact}(iii)).
\end{proof}

\begin{rem}
	The definition of $\Ucal$ as in \Cref{nota:Ucal} is adapted to the definition of morphisms in the embedding categories $\Emb_{\Ucal}(S)$ of \Cref{nota:Emb} and, in particular, is motivated by our need to make these categories connected.
	
	However, as the reader can easily observe, the fact that morphisms in $\Emb_{\Ucal}(S)$ are defined by smooth morphisms does not play any role in our constructions in the rest of the paper: it is just a natural choice motivated by concrete geometric situations. If one is willing to relax the notion of morphisms $p: (Y';s') \rightarrow (Y,s)$ in $\Emb_{\Ucal}(S)$ by allowing $p: Y' \rightarrow Y$ to be not necessarily smooth, then one can replace $\Ucal$ by any other full subcategory of $\Scal$ satisfying the conclusions of \Cref{lem:Ucal}.
\end{rem}

The construction described in the following lemma and corollary is the prototype for the main constructions of this paper: the first provides an explicit construction in terms of the embedding category $\Emb_{\Ucal}(S)$, while the second restates the conclusion more canonically in terms of the trivial groupoid $\Gcal_{\Ucal}(S)$.

\begin{lem}\label{lem_sect}
	Fix $S \in \Scal$.
	For every embedding $(Y;s) \in \Emb_{\Ucal}(S)$, set $A_S^{(s)} := s^* A_Y$. Moreover, for every morphism $p: (Y';s') \rightarrow (Y;s)$ in $\Emb_{\Ucal}(S)$, consider the isomorphism in $\Hbb(S)$
	\begin{equation*}
		\alpha_p: A_S^{(s)} \xrightarrow{\sim} A_S^{(s')}
	\end{equation*}
	given by the composite
	\begin{equation*}
		A_S^{(s)} := s^* A_Y = {s'}^* p^* A_Y \xrightarrow{A_p^*} {s'}^* A_{Y'} =: A_S^{(s')}.
	\end{equation*}
	Then the following statements hold:
	\begin{enumerate}
		\item For every pair of composable morphisms $q: (Y'';s'') \rightarrow (Y';s')$ and $p: (Y';s') \rightarrow (Y;s)$ in $\Emb_{\Ucal}(S)$, we have the equality between isomorphisms in $\Hbb(S)$
		\begin{equation*}
			\alpha_{p \circ q} = \alpha_q \circ \alpha_p.
		\end{equation*}
		\item For every choice of embeddings $(Y^{(i)};s^{(i)}) \in \Emb_{\Ucal}(S)$, $i = 1,2$, consider the isomorphism in $\Hbb(S)$
		\begin{equation*}
			\beta_{1,2}: \tilde{A}_S^{(s^{(1)})} \xrightarrow{\alpha_{p^{(1)}}} \tilde{A}_S^{(s^{(12)})} \xrightarrow{\alpha_{p^{(2)}}^{-1}} \tilde{A}_S^{(s^{(2)})}
		\end{equation*}
		where $s^{(12)}: S \rightarrow Y^{(1)} \times Y^{(2)}$ denotes the diagonal morphism (a closed immersion, by \Cref{hyp:Fact}(iii)) while $p^{(i)}: Y^{(1)} \times Y^{(2)} \rightarrow Y^{(r)}$, $i = 1,2$, denote the canonical projections (both smooth morphisms, by \Cref{hyp:Fact}(ii)). Then the following cocycle condition holds: 
		
		For every three given embeddings $(Y^{(i)};s^{(i)}) \in \Emb_{\Ucal}(S)$, $i = 1,2,3$, we have (with obvious notation) the equality between isomorphisms in $\Hbb(S)$
		\begin{equation*}
			\beta_{1,3} = \beta_{2,3} \circ \beta_{1,2}.
		\end{equation*}
	\end{enumerate}
\end{lem}
\begin{proof}
	\begin{enumerate}
		\item We have to show that the outer part of the diagram in $\Hbb(S)$
		\begin{equation*}
			\begin{tikzcd}[font=\small]
				s^* A_Y \arrow[equal]{r} \arrow[equal]{dd} & {s'}^* p^* A_Y \arrow{r}{A^*} & {s'}^* A_{Y'} \arrow[equal]{d} \\
				& {s''}^* q^* p^* A_Y \arrow[equal]{u} \arrow[equal]{dl} \arrow{r}{A^*} & {s''}^* q^* A_Y \arrow{d}{A^*} \\
				{s''}^* (pq)^* A_Y \arrow{rr}{A^*} && {s''}^* A_{Y''}
			\end{tikzcd}
		\end{equation*}
		is commutative. But the top-left piece is commutative by axiom ($\Scal$-fib-1), the upper-right piece is commutative by naturality, and the bottom-right piece is commutative by axiom ($\Ucal$-sect).
		\item By construction, the cocycle condition asks for the commutativity of the outer part of the diagram in $\Hbb(S)$
		\begin{equation*}
			\begin{tikzcd}[font=\small]
				A_S^{(s^{(1)})} \arrow{r}{\alpha} \arrow{dd}{\alpha} & A_S^{(s^{(12)})} \arrow{d}{\alpha} & A_S^{(s^{(2)})} \arrow{l}{\alpha} \arrow{dd}{\alpha} \\
				& A_S^{(s^{(123)})} \\
				A_S^{(s^{(13)})} \arrow{ur}{\alpha} & A_S^{(s^{(3)})} \arrow{l}{\alpha} \arrow{r}{\alpha} & A_S^{(s^{(23)})} \arrow{ul}{\alpha}
			\end{tikzcd}
		\end{equation*}
		where $s^{(123)}: S \hookrightarrow Y_1 \times Y_2 \times Y_3$ denotes the diagonal closed immersion. This follows from the commutativity of the three inner pieces, which is an immediate consequence of the previous point.
	\end{enumerate}
\end{proof} 

\begin{cor}
	With the notation of \Cref{lem_sect}, the object of $\Hbb(S)$
	\begin{equation*}
		\tilde{A}_S := \varinjlim_{(Y;s) \in \Gcal_{\Ucal}(S)} A_S^{(s)}
	\end{equation*}
	exists and is canonically isomorphic to each $A_S^{(s)}$ (compatibly with the transition maps of \Cref{lem_sect}(2)).
\end{cor}
\begin{proof}
	This follows immediately from \Cref{lem_sect}(2).
\end{proof}

We want to turn the collection of objects $\left\{\tilde{A}_S\right\}_{S \in \Scal}$ just obtained into a section of $\Hbb$ over $\Scal$. The key idea is to exploit the factorization categories $\Fact_{\Scal}(f)$ in combination with the embedding categories $\Emb_{\Ucal}(S)$ as follows: 

\begin{lem}\label{lem_tildeA-indep}
	Fix a morphism $f: T \rightarrow S$ in $\Scal$. For every embedding $(Y;s) \in \Emb_{\Ucal}(S)$ and every factorization $(X;t,p) \in \Fact_{\Scal}(s \circ f)$, define an isomorphism in $\Hbb(T)$
	\begin{equation*}
		A_f^{*,(s,t,p)}: f^* A_S^{(s)} \xrightarrow{\sim} A_T^{(t)}
	\end{equation*}
	by taking the composite
	\begin{equation*}
		f^* A_S^{(s)} := f^* s^* A_Y = t^* p^* A_Y \xrightarrow{A_p^*} t^* A_X =: A_T^{(t)}.
	\end{equation*}
	Then the induced isomorphism in $\Hbb(T)$
	\begin{equation}\label{tildeA_f^*}
		\tilde{A}^* = \tilde{A}_f^*: f^* \tilde{A}_S = f^* \tilde{A}_S^{(s)} \xrightarrow{A_f^{*,(s,t,p)}} \tilde{A}_T^{(t)} = \tilde{A}_T
	\end{equation}
	is independent of the choice of $(Y;s) \in \Emb_{\Ucal}(S)$ and of $(X;t,p) \in \Fact_{\Scal}(s \circ f)$.
\end{lem}
\begin{proof}
	In view of the way the objects $\tilde{A}_S$ and $\tilde{A}_T$ are defined, as using the connectedness of the category $\Emb_{\Ucal}(S)$ via product of varieties (\Cref{lem:Emb}), we are reduced to proving the following claim: Given a morphism $r: (Y';s') \rightarrow (Y;s)$ in $\Emb_{\Ucal}(S)$, a factorization $(X';t',p') \in \Fact_{\Scal}(s' \circ f)$, and a morphism $q: (X';t') \rightarrow (X;t)$ in $\Emb_{\Ucal}(T)$ fitting into a commutative diagram of the form
	\begin{equation*}
		\begin{tikzcd}
			X' \arrow{r}{q} \arrow{d}{p'} & X \arrow{d}{p} \\
			Y' \arrow{r}{r} & Y,
		\end{tikzcd}
	\end{equation*}
	the diagram in $\Hbb(T)$
	\begin{equation*}
		\begin{tikzcd}
			f^* A_S^{(s)} \arrow{rr}{A_f^{*,(s,t,p)}} \arrow{d}{\alpha_r} && A_T^{(t)} \arrow{d}{\alpha_q} \\
			f^* A_S^{(s')} \arrow{rr}{A_f^{*,(s',t',p')}} && A_T^{(t')}
		\end{tikzcd}
	\end{equation*}
	is commutative. Unwinding the various definitions, we obtain the outer part of the diagram
	\begin{equation*}
		\begin{tikzcd}[font=\small]
			f^* s^* A_Y \arrow[equal]{r} \arrow[equal]{dd} & t^* p^* A_Y \arrow{r}{A^*} \arrow[equal]{d} & t^* A_X \arrow[equal]{d} \\
			& {t'}^* q^* p^* A_Y \arrow{r}{A^*} \arrow[equal]{d} & {t'}^* q^* A_X \arrow{dd}{A^*} \\
			f^* {s'}^* r^* A_Y \arrow[equal]{r} \arrow{d}{A^*} & {t'}^* {p'}^* r^* A_Y \arrow{d}{A^*} \\
			f^* {s'}^* A_{Y'} \arrow[equal]{r} & {t'}^* p^* A_{Y'} \arrow{r}{A^*} & {t'}^* A_{X'}
		\end{tikzcd}
	\end{equation*}
	where the upper-left piece is commutative by \cite[Lemma~1.3]{Ter23Fib}, the upper-right and lower-left squares are commutative by naturality, and the lower-right piece is commutative as a consequence of axiom ($\Ucal$-sect). This proves the claim.
\end{proof}

\begin{proof}[Proof of \Cref{prop:ext-sect}]
	For sake of clarity, we divide the proof into two main steps:
	\begin{enumerate}
		\item[(Step 1)] In the first place, we claim that the natural isomorphisms \eqref{tildeA_f^*} turn the family $\left\{\tilde{A}_S\right\}_{S \in \Scal}$ into a section $\tilde{A}$ of the $\Scal$-fibered category $\Hbb$. This amounts to checking that condition ($\Scal$-sect) is satisfied: given two composable morphisms $f: T \rightarrow S$ and $g: S \rightarrow V$ in $\Scal$, we have to show that the diagram in $\Hbb(T)$
		\begin{equation*}
			\begin{tikzcd}
				f^* g^* \tilde{A}_V \arrow{r}{\tilde{A}^*} \arrow[equal]{d} & f^* \tilde{A}_S \arrow{d}{\tilde{A}^*} \\
				(gf)^* \tilde{A}_V \arrow{r}{\tilde{A}^*} & \tilde{A}_T
			\end{tikzcd}
		\end{equation*}
		is commutative. To this end, choose an embedding $(W;v) \in \Emb_{\Ucal}(V)$, then a factorization $(Y;s,q) \in \Fact_{\Scal}(v \circ g)$, and then a factorization $(X;t,p) \in \Fact_{\Scal}(s \circ f)$; note that the objects $X$ and $Y$ automatically belong to $\Ucal$ (\Cref{lem:Ucal}(2)), and so we obtain embeddings $(Y;s) \in \Emb_{\Ucal}(S)$ and $(X;t) \in \Emb_{\Ucal}(T)$. By construction, it now suffices to show that the diagram in $\Hbb(T)$
		\begin{equation*}
			\begin{tikzcd}[font=\small]
				f^* g^* \tilde{A}_V \arrow{rr}{A_g^{*,(v,s,q)}} \arrow[equal]{d} && f^* \tilde{A}_S \arrow{d}{A_f^{*,(s,t,p)}} \\
				(gf)^* \tilde{A}_V \arrow{rr}{A_{gf}^{*,(v,t,qp)}} && \tilde{A}_T
			\end{tikzcd}
		\end{equation*}
		is commutative. Unwinding the various definitions, we obtain the outer part of the diagram
		\begin{equation*}
			\begin{tikzcd}
				f^* g^* v^* A_W \arrow[equal]{r} \arrow[equal]{dd} & f^* s^* q^* A_W \arrow{r}{A^*} \arrow[equal]{d} & f^* s^* A_Y \arrow[equal]{d} \\
				& t^* p^* q^* A_W \arrow[equal]{d} \arrow{r}{A^*} & t^* p^* A_Y \arrow{d}{A^*} \\
				(gf)^* v^* A_W \arrow[equal]{r} & t^* (qp)^* A_W \arrow{r}{A^*} & t^* A_X
			\end{tikzcd}
		\end{equation*}
		where the left-most piece is commutative by \cite[Lemma~1.3]{Ter23Fib}, the upper-right square is commutative by naturality, and the lower-right square is commutative by axiom ($\Ucal$-sect). This proves the claim.
		\item[(Step 2)] Note that the restriction of the section $\tilde{A}$ to $\Ucal$ coincides with the original section $A$ by construction. We claim that this property characterizes $\tilde{A}$ up to unique isomorphism. To see this, let $A'$ be a section of $\Hbb$ over $\Scal$ extending $A$. For every $S \in \Scal$, consider the isomorphism in $\Hbb(S)$
		\begin{equation*}
			w_S: \tilde{A}_S := \varinjlim_{(Y;s) \in \Gcal_{\Ucal}(S)} \tilde{A}_S^{(s)} := \varinjlim_{(Y;s) \in \Gcal_{\Ucal}(S)} s^* A_Y = \varinjlim_{(Y,s) \in \Gcal(S)} s^* A'_Y \xrightarrow{\varinjlim {A'}_s^*} A'_S.
		\end{equation*}
		Note that, for every embedding $(Y;s) \in \Emb_{\Ucal}(S)$, the isomorphism $w_S$ coincides with the composite
		\begin{equation*}
			\tilde{A}_S = \tilde{A}_S^{(s)} := s^* A_Y = s^* A'_Y \xrightarrow{{A'}_s^*} A'_S.
		\end{equation*}
		We claim that, as $S \in \Scal$ varies, these isomorphisms satisfy condition (mor-$\Scal$-sect), thereby defining an isomorphism $w: \tilde{A} \xrightarrow{\sim} A'$ between sections of $\Hbb$ over $\Scal$: for every morphism $f: T \rightarrow S$ in $\Scal$, we have to check that the diagram in $\Hbb(T)$
		\begin{equation*}
			\begin{tikzcd}
				f^* \tilde{A}_S \arrow{r}{w_S} \arrow{d}{\tilde{A}^*} & f^* A'_S \arrow{d}{{A'}^*} \\
				\tilde{A}_T \arrow{r}{w_T} & A'_T
			\end{tikzcd}
		\end{equation*}
		is commutative. To this end, choose an embedding $(Y;s) \in \Emb_{\Ucal}(S)$ and then a factorization $(X;t,p) \in \Fact_{\Scal}(s \circ f)$. By construction, it now suffices to show that the outer part of the diagram in $\Hbb(T)$
		\begin{equation*}
			\begin{tikzcd}[font=\small]
				f^* s^* A_Y \arrow[equal]{d} \arrow[equal]{r} & f^* s^* A'_Y \arrow[equal]{d} \arrow{rr}{{A'}^*} && f^* A'_S \arrow{d}{{A'}^*} \\
				t^* p^* A_Y \arrow[equal]{r} & t^* p^* A'_Y \arrow{r}{{A'}^*} & t^* A'_X \arrow{r}{{A'}^*} & A'_T
			\end{tikzcd}
		\end{equation*}
		is commutative. But the left-most square is obviously commutative, while the right-most rectangle is commutative by axiom ($\Scal$-sect) applied to $A'$. This proves the claim. Moreover, it is clear from the construction that the isomorphism $w$ just defined extends the identity isomorphism of $A$ and is uniquely determined by this property.
	\end{enumerate}
	Finally, the compatibility of the assignment $A \rightsquigarrow \tilde{A}$ with morphisms of sections follows from the fact that every passage in our construction via \Cref{lem_sect} and \Cref{lem_tildeA-indep} is visibly compatible with morphisms of sections.
\end{proof}

\section{Extending morphisms of fibered categories}\label{sect:ext-mor}

While the extension method for sections of $\Scal$-fibered categories works in complete generality, all the other extension results discussed in this paper require that we restrict to a more specific setting: the first part of the present section is devoted to introducing such a setting. This will allow us to state our first main extension result, for morphisms of $\Scal$-fibered categories, which we prove in the second part. In the last part, we briefly discuss its compatibility with the extension result for sections of $\Scal$-fibered categories already studied in \Cref{sect:ext-sect}.

\subsection{Recollections on localic fibered categories}

In the first place, we need to reinforce the working hypotheses on the base category $\Scal$. From now on, we will work under the following hypothesis, which is a strengthened version of \cite[Hyp.~3.1]{Ter23Fact}:

\begin{hyp}\label{hyp:open}
	We keep the assumptions and notation of \Cref{hyp:Fact} and \Cref{nota:Ucal}. In addition, we assume that:
	\begin{enumerate}
		\item[(i)] The category $\Scal$ has an initial object $\emptyset$. 
		\item[(ii)] For every closed immersion $z: Z \rightarrow S$, there exists a distinguished smooth morphism $u: U \rightarrow S$ with the property that, for every $T \in \Scal$, the map
		\begin{equation*}
			\Hom_{\Scal}(T,U) \xrightarrow{u \circ -} \Hom_{\Scal}(T,S)
		\end{equation*} 
		identifies $\Hom_{\Scal}(T,U)$ with the subset of $\Hom_{\Scal}(T,S)$ consisting of all those morphisms $f: T \rightarrow S$ for which the diagram
		\begin{equation*}
			\begin{tikzcd}
				\emptyset \arrow{r} \arrow{d} & T \arrow{d}{f} \\
				Z \arrow{r}{z} & S		
			\end{tikzcd}
		\end{equation*}
		is Cartesian. We call $u$ the \textit{open immersion} complementary to $z$, and we call $U$ the \textit{open complement} of $Z$ in $S$.
		\item[(iii)] For every $S \in \Scal$, there exists a finite sequence of closed immersions
		\begin{equation*}
			\emptyset = Z_{-1} \rightarrow Z_0 \rightarrow Z_1 \rightarrow \dots \rightarrow Z_n = S
		\end{equation*}
		such that, for each $r = 0, \dots, n$, the open complement of $Z_{r-1}$ in $Z_r$ belongs to $\Ucal$.
	\end{enumerate}
\end{hyp}

\begin{ex}
	As in the previous section, we are interested in applying \Cref{hyp:open} when $\Scal = \Var_k$ is the category of quasi-projective algebraic varieties over a field $k$, so that $\Scal^{sm}$ and $\Scal^{cl}$ denote the subcategories of smooth morphisms and closed immersions, and $\Ucal = \Sm_k$ is the full subcategory of smooth quasi-projective $k$-varieties. In this setting, what we call open immersions in (ii) are the usual open immersions between $k$-varieties. Note that, in order for (iii) to be satisfied, one needs to know that every $S \in \Var_k$ admits some dense open subset which is smooth over $k$: assuming that $k$ is perfect, this is true by \cite[\href{https://stacks.math.columbia.edu/tag/0B8X}{Tag 0B8X}]{stacks-project}. 
\end{ex}

The above hypothesis is well adapted to the setting of \textit{localic $\Scal$-fibered categories} developed in \cite[\S~3]{Ter23Fact}; the conventions about triangulated $\Scal$-fibered categories and triangulated morphisms used there are those listed in \cite[\S~9]{Ter23Fib}. Recall from \cite[Defn.~3.4]{Ter23Fact} that, in our axiomatic setting, a triangulated $\Scal$-fibered category is called \textit{localic} if it satisfies the following conditions:
\begin{enumerate}
	\item[(a)] $\Hbb(\emptyset)$ is the zero triangulated category.
	\item[(b)] For every smooth morphism $p: P \rightarrow S$, the functor $p^*: \Hbb(S) \rightarrow \Hbb(P)$ has a left adjoint
	\begin{equation*}
		p_{\#}: \Hbb(P) \rightarrow \Hbb(S).
	\end{equation*}
    Moreover, for every Cartesian square in $\Scal$
	\begin{equation*}
		\begin{tikzcd}
			P_T \arrow{r}{f'} \arrow{d}{p'} & P \arrow{d}{p} \\
			T \arrow{r}{f} & S
		\end{tikzcd}
	\end{equation*}
	with $p$ a smooth morphism, the composite natural transformation of functors $\Hbb(P) \rightarrow \Hbb(T)$
	\begin{equation*}
		p'_{\#} {f'}^* A \xrightarrow{\eta} p'_{\#} {f'}^* p^* p_{\#} A = p'_{\#} {p'}^* f^* p_{\#} A \xrightarrow{\epsilon} f^* p_{\#} A
	\end{equation*}
	is invertible.
	\item[(c)] For every closed immersion $z: Z \rightarrow S$ in $\Scal$, the functor $z^*: \Hbb(S) \rightarrow \Hbb(W)$ has a right adjoint
	\begin{equation*}
		z_*: \Hbb(Z) \rightarrow \Hbb(S)
	\end{equation*}
    which moreover is fully faithful.
	\item[(d)] For every closed immersion $z: Z \rightarrow S$ in $\Scal$ with complementary open immersion $u: U \rightarrow S$, the pair of functors $(z^*,u^*)$ is conservative.
\end{enumerate}
By \cite[Lemma~3.5]{Ter23Fact}, if $\Hbb$ is a localic $\Scal$-fibered category then, for every closed immersion $z: Z \rightarrow S$ with complementary open immersion $u: U \rightarrow S$, the following hold:
\begin{enumerate}
	\item The functor $u_{\#}: \Hbb(U) \rightarrow \Hbb(S)$ is fully faithful.
	\item The composite functors $z^* \circ u_{\#}$ and $u^* \circ z_*$ vanish.
	\item There exists a canonical \textit{localization distinguished triangle} of functors $\Hbb(S) \rightarrow \Hbb(S)$
	\begin{equation*}
		u_{\#} u^* A \xrightarrow{\epsilon} A \xrightarrow{\eta} z_* z^* A \xrightarrow{+}.
	\end{equation*}
	\item For every Cartesian square in $\Scal$ of the form
	\begin{equation*}
		\begin{tikzcd}
			P_Z \arrow{r}{z'} \arrow{d}{p_Z} & P \arrow{d}{p} \\
			Z \arrow{r}{z} & S
		\end{tikzcd}
	\end{equation*}
	with $p$ a smooth morphism and $z$ a closed immersion, the composite natural transformation of functors $\Hbb(Z) \rightarrow \Hbb(P)$
	\begin{equation*}
		p^* z_* A \xrightarrow{\eta} z'_* {z'}^* p^* z_* A = z'_* p_Z^* z^* z_* A \xrightarrow{\epsilon} z'_* p_Z^* A
	\end{equation*}
	is invertible.
\end{enumerate}
For all these results, one only needs to use the left adjoint functors $p_{\#}$ when the smooth morphism $p$ is an open immersion; the same applies to the results of the present paper.
Note that, using \Cref{rem:Usmcl}, any localic $\Scal$-fibered category defines by restriction a localic $\Ucal$-fibered category.

For the rest of this paper, all $\Scal$-fibered or $\Ucal$-fibered categories considered will be localic. As already mentioned, the main goal of this section is to prove an extension result for morphisms between (localic) $\Scal$-fibered categories. Even if our main focus is on monoidal structures, there are at least three good reasons to start from the case of morphisms of fibered categories: firstly, because it is of independent interest; secondly, since the analogous extension result for tensor structures on fibered categories can be proved by simply repeating the same argument, although in a notationally more intricate setting; thirdly, because we implicitly need the case of morphisms of $\Scal$-fibered categories in order to prove a similar extension result for monoidality of morphisms.

Until the end of this section, we fix two localic triangulated $\Scal$-fibered categories $\Hbb_1$ and $\Hbb_2$. We are interested in studying \textit{morphisms of $\Scal$-fibered categories} $R: \Hbb_1 \rightarrow \Hbb_2$: as in \cite[\S~8]{Ter23Fib}, by this we mean the datum of
\begin{itemize}
	\item for every $S \in \Scal$, a functor $R_S: \Hbb_1(S) \rightarrow \Hbb_2(S)$,
	\item for every morphism $f: T \rightarrow S$ in $\Scal$, a natural isomorphism of functors $\Hbb_1(S) \rightarrow \Hbb_2(T)$
	\begin{equation*}
		\theta = \theta_f: f^* R_S(A) \xrightarrow{\sim} R_T (f^* A),
	\end{equation*}
	called the \textit{$R$-transition isomorphism} along $f$
\end{itemize}
such that the following condition is satisfied:
\begin{enumerate}
	\item[(mor-$\Scal$-fib)] For every pair of composable morphisms $f: T \rightarrow S$ and $g: S \rightarrow V$ in $\Scal$, the diagram of functors $\Hbb_1(V) \rightarrow \Hbb_2(T)$
	\begin{equation*}
		\begin{tikzcd}
			(gf)^* R_V(A) \arrow{rr}{\theta_{gf}} \arrow{d}{\conn^{(2)}_{f,g}} && R_T ((gf)^* A) \arrow{d}{\conn^{(1)}_{f,g}} \\
			f^* g^* R_V(A) \arrow{r}{\theta_g} & f^* R_S (g^* A) \arrow{r}{\theta_f} & R_T (f^* g^* A)
		\end{tikzcd}
	\end{equation*}
	is commutative.
\end{enumerate}
If we are already given a family of functors $\Rcal := \left\{R_S: \Hbb_1(S) \rightarrow \Hbb_2(S) \right\}_{S \in \Scal}$, we say that a collection of natural isomorphisms $\theta = \left\{\theta_f: f^* \circ R_S \xrightarrow{\sim} R_T \circ f^* \right\}_{f: T \rightarrow S}$ satisfying condition (mor-$\Scal$-fib) defines an \textit{$\Scal$-structure} on the family $\Rcal$.

Given two morphisms of $\Scal$-fibered categories $R, R': \Hbb_1 \rightarrow \Hbb_2$, there also exists a natural notion of \textit{$2$-morphism of $\Scal$-fibered categories} $\gamma: R \rightarrow R'$: namely, the datum of
\begin{itemize}
	\item for every $S \in \Scal$, a natural transformation of functors $\Hbb_1(S) \rightarrow \Hbb_2(S)$
	\begin{equation*}
		\gamma_S: R_S(A) \rightarrow R'_S(A)
	\end{equation*}
    satisfying the following condition:
    \begin{enumerate}
    	\item[($2$-mor-$\Scal$-fib)] For every arrow $f: T \rightarrow S$ in $\Scal$, the diagram of functors $\Ccal_1(S) \rightarrow \Ccal_2(T)$
    	\begin{equation*}
    		\begin{tikzcd}
    			f^* R_{1,S}(C) \arrow{r}{\gamma_S} \arrow{d}{\theta_f^{(1)}} & f^* R_{2,S}(C) \arrow{d}{\theta_f^{(2)}} \\
    			R_{1,T}(f^* C) \arrow{r}{\gamma_T} & R_{2,T}(f^* C)
    		\end{tikzcd}
    	\end{equation*}
    	is commutative.
    \end{enumerate}
    A $2$-morphism of $\Scal$-fibered categories $\gamma$ is called a \textit{$2$-isomorphism} if all natural transformations $\gamma_S$ are invertible.
\end{itemize}
The notion of $2$-morphism of $\Scal$-fibered categories will be used only in order to state our extension result for morphisms of $\Scal$-fibered categories in a rigorous form; to this end, we only need to consider $2$-isomorphisms. We are not interested in discussing the question of extending $2$-morphisms.

Since $\Hbb_1$ and $\Hbb_2$ are triangulated $\Scal$-fibered categories, it is natural to only consider those morphisms $R$ for which each functor $R_S: \Hbb_1(S) \rightarrow \Hbb_2(S)$ is triangulated compatibly with inverse image functors: in this way, one obtains the natural notion of \textit{triangulated morphism} of triangulated $\Scal$-fibered categories (see \cite[\S~9]{Ter23Fib}). Since moreover $\Hbb_1$ and $\Hbb_2$ are localic, it is natural to further restrict ourselves to those triangulated morphisms $R: \Hbb_1 \rightarrow \Hbb_2$ satisfying the following adjointability condition: for every smooth morphism $p: P \rightarrow S$ in $\Scal$, the natural transformation of functors $\Hbb_1(P) \rightarrow \Hbb_2(S)$
\begin{equation*}
	p_{\#} R_P(A) \xrightarrow{\eta} p_{\#} R_P(p^* p_{\#} A) \xleftarrow{\theta_p} p_{\#} p^* R_S(p_{\#} A) \xrightarrow{\epsilon} R_S(p_{\#} A)
\end{equation*}
is invertible. As in \cite[Defn.~3.8]{Ter23Fact}, we call \textit{localic} any triangulated morphism enjoying this adjointability property.

Note that, using \Cref{rem:Usmcl}, any (localic, triangulated) morphism $\Hbb_1 \rightarrow \Hbb_2$ defines by restriction a (localic, triangulated) morphism between the underlying $\Ucal$-fibered categories. 

\subsection{Extension of localic morphisms}

We can now show that the restriction operation from $\Scal$ to $\Ucal$ for localic morphisms $R: \Hbb_1 \rightarrow \Hbb_2$ is reversible in a canonical way. Here is the precise statement:

\begin{thm}\label{thm_ext}
	Let $R: \Hbb_1 \rightarrow \Hbb_2$ be a localic morphism of localic $\Ucal$-fibered categories.
	Then there exists a unique (up to unique $2$-isomorphism) morphism of triangulated $\Scal$-fibered categories $\tilde{R}: \Hbb_1 \rightarrow \Hbb_2$ whose restriction to $\Ucal$ coincides with $R$. Moreover, the morphism $\tilde{R}$ is automatically localic.  
\end{thm}

The bulk of the construction is contained in the following lemma and corollary, which mimic the strategy outlined in the course of the previous section:

\begin{lem}\label{lem_emb}
	Fix $S \in \Scal$. For every embedding $(Y;s) \in \Emb_{\Ucal}(S)$, consider the triangulated functor
	\begin{equation}\label{R_S_smooth-s}
		R_S^{(s)}: \Hbb_1(S) \xrightarrow{s_*} \Hbb_1(Y) \xrightarrow{R_Y} \Hbb_2(Y) \xrightarrow{s^*} \Hbb_2(S).
	\end{equation}
	Then the following statements hold:
	\begin{enumerate}
		\item For every morphism $p: (Y';s') \rightarrow (Y;s)$ in $\Emb_{\Ucal}(S)$, the natural transformation of functors $\Hbb_1(S) \rightarrow \Hbb_2(S)$
		\begin{equation*}
			\lambda_p: R_S^{(s)}(A) \rightarrow R_S^{(s')}(A)
		\end{equation*}
		defined by taking the composite
		\begin{equation*}
			\begin{tikzcd}[font=\small]
				R_S^{(s)}(A) \arrow[equal]{d} &&& R_S^{(s')}(A) \arrow[equal]{d} \\
				s^* R_Y(s_* A) \arrow[equal]{r} & s'^* p^* R_Y(p_* s'_* A) \arrow{r}{\theta_p} & s'^* R_{Y'} (p^* p_* s'_* A) \arrow{r}{\epsilon} & s'^* R_Y(s'_* A)
			\end{tikzcd}
		\end{equation*}
		is invertible.
		\item For every pair of composable morphisms $q: (Y'';s'') \rightarrow (Y'';s')$, $p: (Y';s') \rightarrow (Y;s)$ in $\Emb_{\Ucal}(S)$, we have the equality between natural isomorphisms of functors $\Hbb_1(S) \rightarrow \Hbb_2(S)$
		\begin{equation*}
			\lambda_{p \circ q} = \lambda_q \circ \lambda_p.
		\end{equation*}
		\item For every choice of embeddings $(Y^{(i)};s^{(i)}) \in \Emb_{\Ucal}(S)$, $i = 1,2$, consider the natural isomorphism of functors $\Hbb_1(S) \rightarrow \Hbb_2(S)$
		\begin{equation*}
			\mu_{1,2}: R_S^{(s^{(1)})}(A) \xrightarrow{\lambda_{p^{(1)}}} R_S^{(s^{(12)})}(A) \xrightarrow{\lambda_{p^{(2)}}^{-1}} R_S^{(s^{(2)})}(A),
		\end{equation*}
		where $s^{(12)}: S \hookrightarrow Y^{(1)} \times Y^{(2)}$ denotes the diagonal immersion while $p^{(r)}: Y^{(1)} \times Y^{(2)} \rightarrow Y^{(r)}$, $r = 1,2$, denote the canonical projections. Then the following cocycle condition holds: 
		
		For every three given embeddings $(Y^{(i)};s^{(i)}) \in \Emb_{\Ucal}(S)$, $i = 1,2,3$, we have (with obvious notation) the equality between natural isomorphisms of functors $\Hbb_1(S) \rightarrow \Hbb_2(S)$
		\begin{equation*}
			\mu_{1,3} = \mu_{2,3} \circ \mu_{1,2}.
		\end{equation*}
	\end{enumerate}
\end{lem}
\begin{proof}
	\begin{enumerate}
		\item It suffices to show that the co-unit arrow $\epsilon: {s'}^* R_{Y'} (p^* p_* {s'}_* A) \rightarrow {s'}^* R_Y({s'}_* A)$ is invertible. To this end, let $u: U \rightarrow Y'$ denote the open immersion complementary to $s'$. By the associated localization triangle, it suffices to show that the two co-unit arrows
		\begin{equation*}
			\epsilon: {s'}^* R_{Y'}({s'}_* {s'}^* p^* p_* {s'}_* A) \rightarrow {s'}^* R_{Y'} ({s'}_* {s'}^* {s'}_* A)
		\end{equation*}
		and
		\begin{equation*}
			\epsilon: {s'}^* R_{Y'} (u_{\#} u^* p^* p_* {s'}_* A) \rightarrow {s'}^* R_{Y'} (u_{\#} u^* {s'}_* A)
		\end{equation*}
		are both invertible. For the first one, this follows from the existence of the commutative diagram
		\begin{equation*}
			\begin{tikzcd}
				{s'}^* R_{Y'} ({s'}_* {s'}^* p^* p_* {s'}_* A) \arrow{r}{\epsilon} \arrow[equal]{d} & {s'}^* R_{Y'} ({s'}_* {s'}^* {s'}_* A) \arrow{d}{\epsilon} \\
				{s'}^* R_{Y'} ({s'}_* s^* s_* A) \arrow{r}{\epsilon} & {s'}^* R_{Y'} ({s'}_* A)
			\end{tikzcd}
		\end{equation*}
		where the other two co-unit arrows are invertible because $s_*$ and $s'_*$ are both fully faithful (as $\Hbb_1$ is localic). For the second one, this follows from the fact that both terms are zero, which is a consequence of the existence of an isomorphism ${s'}^* \circ R_{Y'} \circ u_{\#} \simeq {s'}^* u_{\#} \circ R_{U}$ (as $R$ is a localic morphism) and of the vanishing ${s'}^* \circ u_{\#} = 0$ (as $\Hbb_2$ is localic).
		\item Unwinding the various definitions, we have to show that the outer part of the diagram of functors $\Hbb_1(S) \rightarrow \Hbb_2(S)$
		\begin{equation*}
			\begin{tikzcd}[font=\small]
				{s''}^* (pq)^* R_Y ((pq)_* {s''}_* A) \arrow{rr}{\theta} \arrow[equal]{d} && {s''}^* R_{Y''} ((pq)^* (pq)_* {s''}_* A) \arrow{d}{\epsilon}  \\
				s^* R_S (s_* A) \arrow[equal]{d} && {s''}^* R_{Y''} ({s''}_* A) \\
				{s'}^* p^* R_Y (p_* {s'}_* A) \arrow{d}{\theta} && {s''}^* R_{Y''} (q^* q_* {s''}_* A) \arrow{u}{\epsilon} \\
				{s'}^* R_{Y'} (p^* p_* {s'}_* A) \arrow{r}{\epsilon} & {s'}^* R_{Y'} ({s'}_* A) \arrow[equal]{r} & {s''}^* q^* R_{Y'} (q_* {s''}_* A) \arrow{u}{\theta}
			\end{tikzcd}
		\end{equation*}
		is commutative. To this end, we decompose it as
		\begin{equation*}
			\begin{tikzcd}[font=\tiny]
				\bullet \arrow{rrrr}{\theta} \arrow[equal]{d} \arrow[equal]{dr} &&&& \bullet \arrow{d}{\epsilon} \arrow[equal]{dl}  \\
				\bullet \arrow[equal]{d} & {s''}^* q^* p^* R_Y (p_* q_* {s''}_* A) \arrow{dr}{\theta} && {s''}^* R_{Y''} (q^* p^* p_* q_* {s''}_* A) \arrow{dr}{\epsilon} & \bullet \\
				\bullet \arrow{d}{\theta} \arrow[equal]{ur} && {s''}^* q^* R_{Y'} (p^* p_* q_* {s''}_* A) \arrow{ur}{\theta} \arrow{drr}{\epsilon} && \bullet \arrow{u}{\epsilon} \\
				\bullet \arrow{rr}{\epsilon} \arrow[equal]{urr} && \bullet \arrow[equal]{rr} && \bullet \arrow{u}{\theta}
			\end{tikzcd}
		\end{equation*}
		Here, the upper central piece is commutative by axiom (mor-$\Ucal$-fib), the upper-left piece is commutative by axiom ($\Scal$-fib-1), while all the remaining pieces are commutative by naturality. This proves the claim.
		\item By construction, the cocycle condition asks for the commutativity of the outer part of the diagram of functors $\Hbb_1(S) \rightarrow \Hbb_2(S)$
		\begin{equation*}
			\begin{tikzcd}[font=\small]
				R_S^{(s^{(1)})}(A) \arrow{r}{\lambda} \arrow{dd}{\lambda} & R_{S}^{s^{(12)}}(A) \arrow{d}{\lambda} & R_S^{(s^{(2)})}(A) \arrow{l}{\lambda} \arrow{dd}{\lambda} \\
				& R_S^{(s^{(123)})}(A) \\
				R_S^{(s^{13})}(A) \arrow{ur}{\lambda} & R_{S}^{(s^{(3)})}(A) \arrow{l}{\lambda} \arrow{r}{\lambda} & R_S^{(s^{(23)})}(A) \arrow{ul}{\lambda}
			\end{tikzcd}
		\end{equation*}
		where $s^{(123)}: S \hookrightarrow Y_1 \times Y_2 \times Y_3$ denotes the diagonal closed immersion. This follows from the commutativity of the three inner pieces, which is an immediate consequence of the previous point.
	\end{enumerate}
\end{proof}

\begin{cor}\label{cor_emb}
	With the notation of \Cref{lem_emb}, the formula
	\begin{equation}\label{formula:tilde-R}
		\tilde{R}_S(A) := \varinjlim_{(Y,s) \in \Gcal(S)} R_S^{(s)}(A)
	\end{equation}
	defines a triangulated functor $\tilde{R}_S: \Hbb_1(S) \rightarrow \Hbb_2(S)$ which is canonically isomorphic to each of the functors \eqref{R_S_smooth-s} (compatibly with the transition maps of \Cref{lem_emb}(3)).
\end{cor}
\begin{proof}
	This follows immediately from \Cref{lem_emb}(3).
\end{proof}

We want to turn the collection of triangulated functors $\Rcal = \left\{\tilde{R}_S: \Hbb_1(S) \rightarrow \Hbb_2(S)\right\}_{S \in \Scal}$ just obtained into a (triangulated) morphism of (triangulated) $\Scal$-fibered categories $\tilde{R}: \Hbb_1 \rightarrow \Hbb_2$ by constructing a canonical (triangulated) $\Scal$-structure on it. The following lemma contains the key idea:

\begin{lem}\label{lem_prep-thm-mor-ext}
	Fix a morphism $f: T \rightarrow S$ in $\Scal$. For every embedding $(Y;s) \in \Emb_{\Ucal}(S)$ and every factorization $(X;t,p) \in \Fact_{\Scal}(s \circ f)$, define a natural transformation of functors $\Hbb_1(S) \rightarrow \Hbb_2(T)$
	\begin{equation*}
		\theta_f^{(s,t,p)}: f^* R_S^{(s)}(A) \rightarrow R_T^{(t)}(f^* A)
	\end{equation*}
	by taking the composite
	\begin{equation*}
		\begin{tikzcd}[font=\small]
			f^* R_S^{(s)}(A) \arrow[equal]{d} &&& R_T^{(t)}(f^* A). \arrow[equal]{d} \\
			f^* s^* R_Y (s_* A) \arrow[equal]{d} &&& t^* R_X (t_* f^* A) \\
			t^* p^* R_Y (s_* A) \arrow{r}{\theta_p} & t^* R_X (p^* s_* A) \arrow{r}{\eta} & t^* R_X (t_* t^* p^* s_* A) \arrow[equal]{r} & t^* R_X  (t_* f^* s^* s_* A) \arrow{u}{\epsilon} 
		\end{tikzcd}
	\end{equation*}
	Then $\theta_f^{(s,t,p)}$ is invertible, and the induced natural isomorphism of functors $\Hbb_1(S) \rightarrow \Hbb_2(T)$
	\begin{equation}\label{theta_R_sm}
		\tilde{\theta} = \tilde{\theta}_f: f^* \tilde{R}_S(A) = f^* R_S^{(s)}(A) \xrightarrow{\theta_f^{(s,t,p)}} R_T^{(t)} (f^* A) = \tilde{R}_T (f^* A)
	\end{equation}
	is independent of the choice of $(Y;s) \in \Emb_{\Ucal}(S)$ and of $(X;t,p) \in \Fact_{\Scal}(s \circ f)$.
\end{lem}
\begin{proof}
	We start by showing the invertibility of $\theta_f^{(s,t,p)}$. By construction, this is equivalent to showing the invertibility of the unit arrow $\eta: t^* R_X(A) \rightarrow t^* R_X (t_* t^* A)$. To this end, let $u: U \hookrightarrow X$ denote the open immersion complementary to $t$. Then, by the associated localization triangle, the invertibility of the above unit arrow is equivalent to the vanishing of the composite functor $t^* \circ R_X \circ u_{\#}$, which follows from the existence of a natural isomorphism $R_X \circ u_{\#} \simeq u_{\#} \circ R_U$ (as $R$ is a localic morphism) and from the vanishing $t^* \circ u_{\#} = 0$ (as $\Hbb_2$ is localic).
	
	Next, we show that the natural isomorphism \eqref{theta_R_sm} is independent of the choices of $(Y;s)$ and of $(X;t,p)$. In view of the way the functors $\tilde{R}_S$ and $\tilde{R}_T$ are defined, and using the connectedness of the category $\Emb_{\Ucal}(S)$ via products (\Cref{lem:Emb}), we are reduced to proving the following claim: Given morphisms $r: (Y';s') \rightarrow (Y;s)$ in $\Emb_{\Ucal}(S)$ and $q: (X';t') \rightarrow (X;t)$ in $\Emb_{\Ucal}(T)$ fitting into a commutative diagram of the form
	\begin{equation*}
		\begin{tikzcd}
			X' \arrow{r}{q} \arrow{d}{p'} & X \arrow{d}{p} \\
			Y' \arrow{r}{r} & Y,
		\end{tikzcd}
	\end{equation*}
	the diagram of functors $\Hbb_1(S) \rightarrow \Hbb_2(T)$
	\begin{equation*}
		\begin{tikzcd}
			f^* R_S^{(s')}(A) \arrow{rr}{\theta_f^{(s',t',p')}} && R_T^{(t')}(f^* A) \\
			f^* R_S^{(s)}(A) \arrow{u}{\lambda_r} \arrow{rr}{\theta_f^{(s,t,p)}} && R_T^{(t)}(f^* A) \arrow{u}{\lambda_q}
		\end{tikzcd}
	\end{equation*}
	is commutative. Expanding the definitions, we obtain the more explicit diagram
	\begin{equation*}
		\begin{tikzcd}[font=\small]				
			{t'}^* {p'}^* R_{Y'} ({s'}_* A) \arrow{r}{\theta} & {t'}^* R_{X'} ({p'}^* {s'}_* A) \arrow{r}{\eta} & {t'}^* R_{X'} ({t'}_* {t'}^* {p'}^* {s'}_* A) \arrow[equal]{r} & {t'}^* R_{X'} ({t'}_* f^* {s'}^* {s'}_* A) \arrow{d}{\epsilon} \\
			f^* {s'}^* R_{Y'} ({s'}_* A) \arrow[equal]{u} &&& {t'}^* R_{X'} ({t'}_* f^* A) \\
			f^* {s'}^* R_Y (r^* r_* {s'}_* A) \arrow{u}{\epsilon} &&& {t'}^* R_{X'} (q^* q_* {t'}_* f^* A) \arrow{u}{\epsilon} \\
			f^* {s'}^* r^* R_Y (r_* {s'}_* A) \arrow{u}{\theta} &&& {t'}^* q^* R_X (q_* {t'}_* f^* A) \arrow{u}{\theta} \\
			f^* s^* R_Y (s_* A) \arrow[equal]{d} \arrow[equal]{u} &&& t^* R_X (t_* f^* A) \arrow[equal]{u} \\
			t^* p^* R_Y (s_* A) \arrow{r}{\theta} & t^* R_X (p^* s_* A) \arrow{r}{\eta} & t^* R_X (t_* t^* p^* s_* A) \arrow[equal]{r} & t^* R_X (t_* f^* s^* s_* A) \arrow{u}{\epsilon}
		\end{tikzcd}
	\end{equation*}
	that we decompose as
	\begin{equation*}
		\begin{tikzcd}[font=\tiny]
			\bullet \arrow{rr}{\theta} \arrow[equal]{d} && \bullet \arrow{r}{\eta} & \bullet \arrow[equal]{rr} && \bullet \arrow{d}{\epsilon} \\
			\bullet & {t'}^* {p'}^* R_{Y'} (r^* r_* {s'}_* A) \arrow{r}{\theta} \arrow{ul}{\epsilon} \arrow[equal]{dl}  & {t'}^* R_{X'} ({p'}^* r^* r_* {s'}_* A) \arrow{r}{\eta} \arrow{u}{\epsilon} & {t'}^* R_{X'} ({t'}_* {t'}^* {p'}^* r^* r_* {s'}_* A) \arrow[equal]{r} \arrow{u}{\epsilon} & {t'}^* R_{X'} ({t'}_* f^* {s'}^* r^* r_* {s'}_* A) \arrow[equal]{d} \arrow{ur}{\epsilon} & \bullet \\
			\bullet \arrow{u}{\epsilon} &&&& {t'}^* R_{X'} (t_* f^* s^* s_* A) \arrow{ur}{\epsilon} & \bullet \arrow{u}{\epsilon} \\
			& {t'}^* {p'}^* r^* R_Y (r_* {s'}_* A) \arrow{uu}{\theta} \arrow[equal]{dl} &&& {t'}^* R_{X'} (q^* q_* {t'}_* f^* s^* s_* A) \arrow{u}{\epsilon} \arrow{ur}{\epsilon} \\
			\bullet \arrow{uu}{\theta} & {t'}^* {p'}^* r^* R_Y (s_* A) \arrow[equal]{u} \arrow[equal]{d} &&& {t'}^* q^* R_X (q_* {t'}_* f^* s^* s_* A) \arrow{u}{\theta} \arrow{r}{\epsilon} & \bullet \arrow{uu}{\theta} \\
			\bullet \arrow[equal]{u} \arrow[equal]{d} & {t'}^* q^* p^* R_Y (s_* A) \arrow{r}{\theta} & {t'}^* q^* R_X (p^* s_* A) \arrow{r}{\eta} & {t'}^* q^* R_X (t_* t^* p^* s_* A) \arrow[equal]{r} & {t'}^* q^* R_X (q_* {t'}_* t^* p^* s_* A) \arrow[equal]{u} & \bullet \arrow[equal]{u} \\
			\bullet \arrow{rr}{\theta} \arrow[equal]{ur} && \bullet \arrow{r}{\eta} \arrow[equal]{u} & \bullet \arrow[equal]{rr} \arrow[equal]{u} && \bullet \arrow{u}{\epsilon}
		\end{tikzcd}
	\end{equation*}
	Here, the lower-left lateral piece is commutative by \cite[Lemma~1.3]{Ter23Fib}, the lower-right lateral piece is commutative by construction, and all the remaining lateral pieces are commutative by naturality. Thus it now suffices to show that the central piece
    is commutative as well. To this end, we further decompose it as
	\begin{equation*}
		\begin{tikzcd}[font=\tiny]
			\bullet \arrow{rr}{\theta} \arrow[equal]{ddr} && \bullet \arrow{r}{\eta} \arrow[equal]{dd} & \bullet \arrow[equal]{rr} \arrow[equal]{d} && \bullet \arrow[equal]{d} \\
			&&& {t'}^* R_{X'} ({t'}_* {t'}^* {p'}^* r^* s_* A) && \bullet \\
			\bullet \arrow{uu}{\theta} & {t'}^* {p'}^* R_{Y'}(r^* s_* A) \arrow{r}{\theta} & {t'}^* R_{X'}({p'}^* r^* s_* A) \arrow[equal]{d} \arrow{ur}{\eta} & {t'}^* R_{X'} ({t'}_* {t'}^* q^* p^* s_* A) \arrow[equal]{r} \arrow[equal]{u} & {t'}^* R_{X'}({t'}_* {t'}^* p^* s_* A) \arrow[equal]{ur} & \bullet \arrow{u}{\epsilon} \\
			\bullet \arrow[equal]{u} \arrow[equal]{d} \arrow{ur}{\theta} && {t'}^* R_{X'} (q^* p^* s_* A) \arrow{ur}{\eta} \arrow{r}{\eta} & {t'}^* R_{X'}(q^* t_* t^* p^* s_* A) \arrow[equal]{r} & {t'}^* R_{X'}(q^* q_* {t'}_* t^* p^* s_* A) \arrow{u}{\epsilon} \arrow[equal]{ur} & \bullet \arrow{u}{\theta} \\
			\bullet \arrow{rr}{\theta} && \bullet \arrow{r}{\eta} \arrow{u}{\theta} & \bullet \arrow[equal]{rr} \arrow{u}{\theta} && \bullet \arrow[equal]{u} \arrow{ul}{\theta}
		\end{tikzcd}
	\end{equation*}
	Here, the upper-right piece is commutative by \cite[Lemma~1.3]{Ter23Fib}, the lower-left piece is commutative by \cite[Lemma~1.8]{Ter23Fib}, the central five-term piece is commutative by construction, and all the remaining pieces are commutative by naturality. This proves the claim.
\end{proof}

\begin{proof}[Proof of \Cref{thm_ext}]
	For sake of clarity, we divide the proof into three main steps:
	\begin{enumerate}
		\item[(Step 1)] In the first place, we check that the isomorphisms \eqref{theta_R_sm} turn the collection of triangulated functors \eqref{formula:tilde-R} into a morphism of $\Scal$-fibered categories. To see this, we only need to check that they satisfy axioms (mor-$\Scal$-fib): given two composable morphisms $f: T \rightarrow S$ and $g: S \rightarrow V$ in $\Scal$, we have to show that the diagram of functors $\Hbb_1(V) \rightarrow \Hbb_2(T)$
		\begin{equation*}
			\begin{tikzcd}
				(gf)^* \tilde{R}_V(A) \arrow[equal]{d} \arrow{rr}{\tilde{\theta}} && \tilde{R}_T((gf)^* A) \arrow[equal]{d} \\
				f^* g^* \tilde{R}_V(A) \arrow{r}{\tilde{\theta}} & f^* \tilde{R}_T(g^* A) \arrow{r}{\tilde{\theta}} & \tilde{R}_S(f^* g^* A)
			\end{tikzcd}
		\end{equation*}
		commutes. To this end, choose an object $(W;v) \in \Emb_{\Ucal}(V)$, a factorization $(Y;s,q) \in \Fact_{\Scal}(v \circ g)$, and then a factorization $(X;t,p) \in \Fact_{\Scal}(s \circ f)$; note that the objects $X$ and $Y$ automatically belong to $\Ucal$ (\Cref{lem:Ucal}(2)), and so we obtain embeddings $(Y;s) \in \Emb_{\Ucal}(S)$ and $(X;t) \in \Emb_{\Ucal}(T)$. By construction, it now suffices to show that the diagram
		\begin{equation*}
			\begin{tikzcd}[font=\small]
				(gf)^* R_V(A) \arrow[equal]{d} \arrow{rr}{\theta_{gf}^{(v,t,qp)}} && R_T((gf)^* A) \arrow[equal]{d} \\
				f^* g^* R_V(A) \arrow{r}{\theta_g^{(v,s,q)}} & f^* R_T(g^* A) \arrow{r}{\theta_f^{(s,t,p)}} & R_S(f^* g^* A)
			\end{tikzcd}
		\end{equation*}
		is commutative. Unwinding the various definition, we obtain the more explicit diagram
		\begin{equation*}
			\begin{tikzcd}[font=\tiny]
				t^* (qp)^* R_W (v_* A) \arrow[equal]{d} \arrow{r}{\theta} & t^* R_X ((qp)^* v_* A) \arrow{rr}{\eta} && t^* R_X (t_* t^* (qp)^* v_* A) \arrow[equal]{r} & t^* R_X (t_* (gf)^* v^* v_* A) \arrow{d}{\epsilon} \\
				(gf)^* v^* R_W (v_* A) \arrow[equal]{d} &&&& t^* R_X (t_* (gf)^* A) \arrow[equal]{d} \\
				f^* g^* v^* R_W (v_* A) \arrow[equal]{d} &&&& t^* R_X (t_* f^* g^* A) \\
				f^* s^* q^* R_W (v_* A) \arrow{d}{\theta} &&&& t^* R_X (t_* f^* s^* s_* g^* A) \arrow{u}{\epsilon} \\
				f^* s^* R_Y (q^* v_* A) \arrow{d}{\eta} &&&& t^* R_X (t_* t^* p^* s_* g^* A) \arrow[equal]{u} \\
				f^* s_* R_Y (s_* s^* q^* v_* A) \arrow[equal]{d} &&&& t^* R_X (p^* s_* g^* A) \arrow{u}{\eta} \\
				f^* s^* R_Y (s_* g^* v^* v_* A) \arrow{rr}{\epsilon} && f^* s^* R_Y (s_* g^* A) \arrow[equal]{rr} && t^* p^* R_Y (s_* g^* A) \arrow{u}{\theta}
			\end{tikzcd}
		\end{equation*}
		that we decompose as
		\begin{equation*}
			\begin{tikzcd}[font=\tiny]
				\bullet \arrow[equal]{d} \arrow{r}{\theta} \arrow[equal]{ddr} & \bullet \arrow{r}{\eta} \arrow[equal]{dddr} & \bullet \arrow[equal]{rrr} \arrow[equal]{d} &&& \bullet \arrow{d}{\epsilon} \arrow[equal]{dl} \\
				\bullet \arrow[equal]{d} && t^* R_X(t_* t^* p^* q^* v_* A) \arrow[equal]{r} \arrow{ddr}{\eta} & t^* R_X(t_* f^* s^* q^* v_* A) \arrow[equal]{r} \arrow{d}{\eta} & t^* R_X(t_* f^* g^* v^* v_* A) \arrow{dr}{\epsilon} & \bullet \arrow[equal]{d} \\
				\bullet \arrow[equal]{d} & t^* p^* q^* R_W(v_* A) \arrow{d}{\theta} \arrow[equal]{dl} && t^* R_X(t_* f^* s^* s_* s^* q^* v_* A) \arrow[equal]{dr} \arrow[equal]{d} && \bullet \\
				\bullet \arrow{d}{\theta} & t^* p^* R_Y(q^* v_* A) \arrow[equal]{dl} \arrow{r}{\theta} \arrow{d}{\eta} & t^* R_X(p^* q^* v_* A) \arrow{uu}{\eta} \arrow{d}{\eta} & t^* R_X(t_* t^* p^* s_* s^* q^* v_* A) \arrow[equal]{dr} & t^* R_X(t_* f^* s^* s_* g^* v^* v_* A) \arrow{r}{\epsilon} \arrow{uu}{\epsilon} & \bullet \arrow{u}{\epsilon} \\
				\bullet \arrow{d}{\eta} & t^* p^* R_Y(s_* s^* q^* v_* A) \arrow[equal]{dr}  \arrow{r}{\theta} & t^* R_X(p^* s_* s^* q^* v_* A) \arrow{ur}{\eta} \arrow[equal]{dr} && t^* R_X(t_* t^* p^* s_* g^* v^* v_* A) \arrow{r}{\epsilon} \arrow[equal]{u} & \bullet \arrow[equal]{u} \\
				\bullet \arrow[equal]{d} \arrow[equal]{ur} &  & t^* p^* R_Y(s_* g^* v^* v_* A) \arrow[equal]{dll} \arrow{r}{\theta} \arrow{drrr}{\epsilon} & t^* R_X(p^* s_* g^* v^* v_* A) \arrow{rr}{\epsilon} \arrow{ur}{\eta} && \bullet \arrow{u}{\eta} \\
				\bullet \arrow{rr}{\epsilon} && \bullet \arrow[equal]{rrr} &&& \bullet \arrow{u}{\theta}
			\end{tikzcd}
		\end{equation*}
		Here, the upper-left and upper-right six-term pieces are commutative by \cite[Lemma~1.8]{Ter23Fib} and \cite[Lemma~1.3]{Ter23Fib}, respectively, while all remaining pieces are commutative by naturality and by construction. This proves the claim.
		
		Therefore we obtain a morphism of $\Scal$-fibered categories $\tilde{R}: \Hbb_1 \rightarrow \Hbb_2$ as promised. As a consequence of the explicit description of $\tilde{R}$ provided by \Cref{cor_emb} and \Cref{lem_prep-thm-mor-ext}, it is clear that the restriction of $\tilde{R}$ to $\Ucal$ coincides with $R$. It is also clear that $\tilde{R}$ is canonically a morphism of triangulated $\Scal$-fibered categories in the sense of \cite[\S~9]{Ter23Fib}, since $R$ is assumed to be triangulated and the construction of $\tilde{R}$ only uses triangulated functors. 
		\item[(Step 2)] 
		We now check that $\tilde{R}$ is a localic morphism: given a smooth morphism $p: P \rightarrow S$ in $\Scal$, we have to show that the natural transformation of functors $\Hbb_1(P) \rightarrow \Hbb_2(S)$
		\begin{equation*}
			p_{\#} \tilde{R}_P(A) \xrightarrow{\eta} p_{\#} \tilde{R}_P(p^* p_{\#} A) \xleftarrow{\theta_p} p_{\#} p^* \tilde{R}_S(p_{\#} A) \xrightarrow{\epsilon} \tilde{R}_S (p_{\#} A)
		\end{equation*}
		is invertible. If $S \in \Ucal$ (in which case we also have $P \in \Ucal$, by \Cref{lem:Ucal}(2)) then the claim holds true since $R$ is a localic morphism by hypothesis.
		To prove the claim in general, we argue by induction on the minimal length $n \in \N$ of a sequence of closed immersions of the form
		\begin{equation*}
			\emptyset = Z_{-1} \rightarrow Z_0 \rightarrow Z_1 \rightarrow \dots \rightarrow Z_n = S
		\end{equation*} 
		as in \Cref{hyp:open}(iii). If $n = 0$, we have $S \in \Ucal$ and the claim is already known to hold. For the inductive step, assume that $n \geq 1$ and that the claim is known to hold for $Z_{n-1}$. For notational simplicity, write $Z_{n-1}$ simply as $Z$, and let $z: Z \rightarrow S$ denote the corresponding closed immersion; moreover, let $u: U \rightarrow S$ denote the complementary open immersion.
		Consider the commutative diagram of functors $\Hbb_1(P) \rightarrow \Hbb_2(S)$
		\begin{equation}\label{dia:ext-mor-loc}
			\begin{tikzcd}
				u_{\#} u^* p_{\#} \tilde{R}_T(A) \arrow{r} \arrow{d} & p_{\#} \tilde{R}_T(A) \arrow{r} \arrow{d} & z_* z^* p_{\#} \tilde{R}_T(A) \arrow{r}{+1} \arrow{d} & {} \\
				u_{\#} u^* \tilde{R}_S (p_{\#} A) \arrow{r} & \tilde{R}_S (p_{\#} A) \arrow{r} & z_* z^* \tilde{R}_S (p_{\#} A) \arrow{r}{+1} & {}
			\end{tikzcd}
		\end{equation}
		where both rows are induced by the associated localization triangle. In order to show that the middle vertical arrow is invertible, it suffices to prove the invertibility of the other two vertical arrows. To this end, form the diagram with Cartesian squares
		\begin{equation*}
			\begin{tikzcd}
				P_U \arrow{r}{u'} \arrow{d}{p_U} & P \arrow{d}{p} & P_Z \arrow{d}{p_Z} \arrow{l}{z'} \\
				U \arrow{r}{u} & S & Z. \arrow{l}{z}
			\end{tikzcd}
		\end{equation*}
		It then suffices to show that the two diagrams of functors $\Hbb_1(P) \rightarrow \Hbb_2(S)$
		\begin{equation*}
			\begin{tikzcd}
				u_{\#} u^* p_{\#} \tilde{R}_P(A) \arrow{d}{\star} & u_{\#} p_{U,\#} {u'}^* \tilde{R}_P(A) \arrow{r}{\theta} \arrow{l}{\sim} & u_{\#} p_{U,\#} \tilde{R}_{U_P} ({u'}^* A) \isoarrow{d} \\
				u_{\#} u^* \tilde{R}_S (p_{\#} A) \arrow{r}{\theta} & u_{\#} \tilde{R}_U (u^* p_{\#} A) & u_{\#} \tilde{R}_U (p_{U,\#} {u'}^* A) \arrow{l}{\sim}
			\end{tikzcd}
		\end{equation*}
		and
		\begin{equation*}
			\begin{tikzcd}
				z_* z^* p_{\#} \tilde{R}_P(A) \arrow{d}{\star} & z_* p_{Z,\#} {z'}^* \tilde{R}_P(A) \arrow{r}{\theta} \arrow{l}{\sim} & z_* p_{Z,\#} \tilde{R}_{Z_P} ({z'}^* A) \isoarrow{d} \\
				z_* z^* \tilde{R}_S (p_{\#} A) \arrow{r}{\theta} & z_* \tilde{R}_Z (z^* p_{\#} A) & z_* \tilde{R}_Z (p_{Z,\#} {z'}^* A) \arrow{l}{\sim}
			\end{tikzcd}
		\end{equation*}
		are commutative: indeed, in both of these diagrams, all arrows except possibly the marked one are invertible (by condition definition of localic $\Scal$-fibered category and by inductive hypothesis), and the two marked arrows are exactly the two lateral vertical arrows in \eqref{dia:ext-mor-loc}.
		Let us show, for example, that the first diagram is commutative; the case of the second diagram can be treated in the same way. Unwinding the various definitions in the first diagram, we obtain the more explicit diagram
		\begin{equation*}
			\begin{tikzcd}[font=\small]
				u_{\#} u^* p_{\#} \tilde{R}_P(A) \arrow{d}{\eta} & u_{\#} p_{U,\#} {u'}^* \tilde{R}_P(A) \arrow{l}{\sim} \arrow{r}{\theta} & u_{\#} p_{U,\#} \tilde{R}_{U_P} ({u'}^* A) \arrow{d}{\eta} \\
				u_{\#} u^* p_{\#} \tilde{R}_P (p^* p_{\#} A) && u_{\#} p_{U,\#} \tilde{R}_{U_P} (p_U^* p_{U,\#} {u'}^* A) \\
				u_{\#} u^* p_{\#} p^* \tilde{R}_S (p_{\#} A) \arrow{u}{\theta} \arrow{d}{\epsilon} && u_{\#} p_{U,\#} p_U^* \tilde{R}_U (p_{U,\#} {u'}^* A) \arrow{u}{\theta} \arrow{d}{\epsilon} \\
				u_{\#} u^* \tilde{R}_S (p_{\#} A) \arrow{r}{\theta} & u_{\#} \tilde{R}_U (u^* p_{\#} A) & u_{\#} \tilde{R}_U (p_{U,\#} {u'}^* A) \arrow{l}{\sim}
			\end{tikzcd}
		\end{equation*}
		that we decompose as
		\begin{equation*}
			\begin{tikzcd}[font=\small]
				\bullet \arrow{dd}{\eta} && \bullet \arrow{ll}{\sim} \arrow{rr}{\theta} \arrow{dl}{\eta} && \bullet \arrow{dd}{\eta} \arrow{dl}{\eta} \\
				& u_{\#} p_{U,\#} {u'}^* \tilde{R}_P (p^* p_{\#} A) \arrow{dl}{\sim} \arrow{rr}{\theta} && u_{\#} p_{U,\#} \tilde{R}_{U_P} ({u'}^* p^* p_{\#} A) \arrow[equal]{dd} \\
				\bullet &&&& \bullet \arrow{dl}{\sim} \\
				& u_{\#} p_{U,\#} {u'}^* p^* \tilde{R}_S (p_{\#} A) \arrow{uu}{\theta} \arrow{dl}{\sim} && u_{\#} p_{U,\#} \tilde{R}_{U_P} (p_U^* u^* p_{\#} A) \\
				\bullet \arrow{uu}{\theta} \arrow{dd}{\epsilon} &&&& \bullet \arrow{uu}{\theta} \arrow{dd}{\epsilon} \arrow{dl}{\sim} \\
				& u_{\#} p_{U,\#} p_U^* u^* \tilde{R}_S (p_{\#} A) \arrow{dl}{\epsilon} \arrow[equal]{uu} \arrow{rr}{\theta} && u_{\#} p_{U,\#} p_U^* \tilde{R}_U (u^* p_{\#} A) \arrow{uu}{\theta} \arrow{dl}{\epsilon} \\
				\bullet \arrow{rr}{\theta} && \bullet && \bullet \arrow{ll}{\sim}
			\end{tikzcd}
		\end{equation*}
		Here, the central rectangle is commutative by axiom (mor-$\Ucal$-fib), while all the remaining pieces are commutative by construction and by naturality. This proves the claim.
		\item[(Step 3)] Finally, we check that our morphism $\tilde{R}$ is the unique possible extension of $R$ to a morphism of triangulated $\Scal$-fibered categories. To this end, let $R': \Hbb_1 \rightarrow \Hbb_2$ be a second morphism of triangulated $\Scal$-fibered categories extending $R$; for every morphism $f: T \rightarrow S$ in $\Scal$, let $\theta' = \theta'_f: f^* \circ R_S \xrightarrow{\sim} R_T \circ f^*$ denote the associated $R'$-transition isomorphism. 
		In view of the way the $\tilde{R}$-transition isomorphisms $\tilde{\theta}_f$ are constructed in \Cref{lem_prep-thm-mor-ext}, we see that the equality $\tilde{R} = R'$ will follow as soon as we prove the following claim: For every commutative diagram in $\Scal$ of the form
		\begin{equation*}
			\begin{tikzcd}
				T \arrow{r}{t} \arrow{d}{f} & X \arrow{d}{p} \\
				S \arrow{r}{s} & Y
			\end{tikzcd}
		\end{equation*}
		with $t$ and $s$ closed immersions, the outer part of the diagram of functors $\Hbb_1(S) \rightarrow \Hbb_2(T)$
		\begin{equation*}
			\begin{tikzcd}
				f^* R'_S(A) \arrow{rrr}{\theta'} &&& R'_T (f^* A) \\
				f^* R'_S (s^* s_* A) \arrow{u}{\epsilon} &&& R'_T (t^* t_* f^* A) \arrow{u}{\epsilon}  \\
				f^* s^* R'_Y(s_* A) \arrow{u}{\theta'} &&& t^* R'_X (t_* f^* A) \arrow{u}{\theta'} \\
				t^* p^* R'_Y (s_* A) \arrow[equal]{u} \arrow{r}{\theta'} & t^* R'_X (p^* s_* A) \arrow{r}{\eta} & t^* R'_X (t_* t^* p^* s_* A) \arrow[equal]{r} & t^* R'_X (t_* f^* s^* s_* A) \arrow{u}{\epsilon}
			\end{tikzcd}
		\end{equation*}
		is commutative. To this end, it suffices to decompose it as
		\begin{equation*}
			\begin{tikzcd}[font=\small]
				\bullet \arrow{rrr}{\theta'} &&& \bullet \\
				\bullet \arrow{u}{\epsilon} \arrow{r}{\theta'} & R'_T (f^* s^* s_* A) \arrow[equal]{d} \arrow{urr}{\epsilon} & R'_T (t^* t_* f^* s^* s_* A) \arrow{l}{\epsilon} \arrow{r}{\epsilon} \arrow[equal]{d} & \bullet \arrow{u}{\epsilon}  \\
				\bullet \arrow{u}{\theta'} & R'_T (t^* p^* s_* A) \arrow{r}{\eta}  & R'_T (t^* t_* t^* p^* s_* A) & \bullet \arrow{u}{\theta'} \\
				\bullet \arrow{r}{\theta'} \arrow[equal]{u} & \bullet \arrow{r}{\eta} \arrow{u}{\theta'} & \bullet \arrow[equal]{r} \arrow{u}{\theta'} & \bullet \arrow{u}{\epsilon} \arrow[bend right]{uul}{\theta'}
			\end{tikzcd}
		\end{equation*}
		Here, the bottom-left piece is commutative by \cite[Lemma~1.8]{Ter23Fib}, while the remaining pieces are commutative by naturality and by construction. This proves the claim, thereby concluding the proof.
	\end{enumerate}
\end{proof}

\subsection{Extending images of sections}

We conclude this section by analyzing the compatibility between the extension result for localic morphisms just proved and our previous extension result for sections of $\Scal$-fibered categories obtained in \Cref{sect:ext-sect}. This is based on the notion of \textit{image} of a section $A$ of $\Hbb_1$ over $\Scal$ under a morphism of $\Scal$-fibered categories $R: \Hbb_1 \rightarrow \Hbb_2$: as in \cite{Ter23Fib}, by this we mean the section $R(A)$ of $\Hbb_2$ over $\Scal$ defined by taking
\begin{itemize}
	\item for every $S \in \Scal$, the object $R(A)_S \in \Hbb_2(S)$ is just $R_S(A_S)$,
	\item for every morphism $f: T \rightarrow S$ in $\Scal$, the isomorphism in $\Hbb_2(T)$
	\begin{equation*}
		R(A)_f^*: f^* R(A)_S \xrightarrow{\sim} R(A)_T
	\end{equation*}
	is the composite
	\begin{equation*}
		f^* R_S(A_S) \xrightarrow{\theta_f} R_T(f^* A_S) \xrightarrow{\sim} R_T(A_T).
	\end{equation*}
\end{itemize} 
This assignment defines indeed a section of $\Hbb_2$ over $\Scal$, as checked in \cite[Constr.~1.11]{Ter23Fib}.

With this notion at our disposal, we do the following sanity check:

\begin{lem}\label{lem:ext-sect-comp}
	Let $A$ be a section of $\Hbb_1$ over $\Ucal$. Then there exists a unique isomorphism of sections of $\Hbb_2$ over $\Scal$
	\begin{equation*}
		\tilde{R}({\tilde{A}}) \xrightarrow{\sim} \widetilde{R(A)}
	\end{equation*}
    extending the identity of $R(A)$.
\end{lem}
\begin{proof}
	In the first place, for every $S \in \Scal$ we want to define an isomorphism in $\Hbb_2(S)$
	\begin{equation}\label{iso:tildeR(A)_S}
		\tilde{R}({\tilde{A}})_S \xrightarrow{\sim} \widetilde{R(A)}_S.
	\end{equation}
	To this end, for every $(Y;s) \in \Emb_{\Ucal}(S)$, consider the isomorphism in $\Hbb_2(S)$
	\begin{equation*}
		R_S^{(s)}(A_S^{(s)}) \xrightarrow{\sim} R(A)_S^{(s)}
	\end{equation*}
	defined as the composite
	\begin{equation*}
		R_S^{(s)}(A_S^{(s)}) := s^* R_Y(s_* s^* A_Y) \xleftarrow{\eta} s^* R_Y(A_Y) =: s^* R(A)_Y =: R(A)_S^{(s)}.
	\end{equation*}
	We claim that, as $(Y;s)$ varies in $\Gcal_{\Ucal}(S)$, these isomorphisms are compatible with the transition maps of \Cref{lem_sect} and \Cref{lem_emb}; if this is the case, we obtain the sought-after isomorphism \eqref{iso:tildeR(A)_S} by taking their colimit over the trivial groupoid $\Gcal_{\Ucal}(S)$ as usual.
	In order to prove the claim, it suffices to show that, for every arrow $p: (Y';s') \rightarrow (Y;s)$ in $\Emb_{\Ucal}(S)$, the diagram in $\Hbb_2(S)$
	\begin{equation*}
		\begin{tikzcd}[font=\small]
			R_S^{(s)}(A_S^{(s)}) \arrow{rr}{\sim} \arrow{d}{\lambda_p} && R(A)_S^{(s)} \arrow{dd}{\alpha_p} \\
			R_S^{(s')}(A_S^{(s)}) \arrow{d}{\alpha_p} \\
			R_S^{(s')}(A_S^{(s')}) \arrow{rr}{\sim} && R(A)_S^{(s')}
		\end{tikzcd}
	\end{equation*} 
	is commutative. Expanding the various definitions, we obtain the outer part of the diagram
	\begin{equation*}
		\begin{tikzcd}[font=\small]
			{s'}^* p^* R_Y(p_* s'_* s^* A_Y) \arrow[equal]{rr} \arrow{d}{\theta} && s^* R_Y(s_* s^* A_Y) \arrow[equal]{d} & s^* R_Y(A_Y) \arrow{l}{\eta} \arrow[equal]{d} \\
			{s'}^* R_{Y'}(p^* p_* s'_* s^* A_Y) \arrow[equal]{r} \arrow[equal]{dr} \arrow{d}{\epsilon} & s^* R_{Y'}(p^* s_* s^* A_Y) \arrow[equal]{d} & {s'}^* p^* R_Y(s_* s^* A_Y) \arrow{l}{\theta} & {s'}^* p^* R_Y(A_Y) \arrow{l}{\eta} \arrow{d}{\theta} \\
			{s'}^* R_{Y'}(s'_* s^* A_Y) \arrow[equal]{d} & {s'}^* R_{Y'}(p^* p_* s'_* {s'}^* p^* A_Y) \arrow{dl}{\epsilon} && {s'}^* R_{Y'}(p^* A_Y) \arrow{ll}{\eta} \arrow{d}{A^*} \arrow{dlll}{\eta} \\
			{s'}^* R_{Y'}(s'_* {s'}^* p^* A_Y) \arrow{rr}{A^*} && {s'}^* R_{Y'}(s'_* {s'}^* A_{Y'}) & {s'}^* R_{Y'}(A_{Y'}) \arrow{l}{\eta}
		\end{tikzcd}
	\end{equation*}
	where all pieces are commutative by naturality and by construction. This proves the claim. 
	
	In order to conclude the proof, it now suffices to prove that the isomorphisms \eqref{iso:tildeR(A)_S} assemble into an isomorphism of sections of $\Hbb_2$ over $\Scal$. This amounts to checking that they satisfy condition (mor-$\Scal$-sect): for every morphism $f: T \rightarrow S$ in $\Scal$, we have to show that the diagram in $\Hbb_2(T)$
	\begin{equation*}
		\begin{tikzcd}
			f^* \tilde{R}(\tilde{A})_S \arrow{d}{\tilde{R}(\tilde{A})^*} \arrow{rr}{\sim} && f^* \widetilde{R(A)}_S \arrow{d}{\tilde{R(A)}^*} \\
			\tilde{R}(\tilde{A})_T \arrow{rr}{\sim} && \widetilde{R(A)}_T
		\end{tikzcd}
	\end{equation*}
	is commutative. To this end, choose an embedding $(Y;s) \in \Emb_{\Ucal}(S)$ and a factorization $(X;t,p) \in \Fact_{\Scal}(s \circ f)$; note that the object $X$ automatically belongs to $\Ucal$ (\Cref{lem:Ucal}(2)), and so we obtain an embedding $(X;t) \in \Emb_{\Ucal}(T)$. By construction, it now suffices to show that the diagram
	\begin{equation*}
		\begin{tikzcd}[font=\small]
			f^* R_S^{(s)}(A_S^{(s)}) \arrow{rr}{\sim} \arrow{d}{\theta_f^{(s,t,p)}} && f^* R(A)_S^{(s)} \arrow{dd}{R(A)_f^{*,(s,t,p)}} \\
			R_T^{(t)}(f^* A_S^{(s)}) \arrow{d}{A_f^{*,(s,t,p)}} \\
			R_T^{(t)}(A_T^{(t)}) \arrow{rr}{\sim} && R(A)_T^{(t)}
		\end{tikzcd}
	\end{equation*}
	is commutative. Unwinding the various definitions, we obtain the outer part of the diagram
	\begin{equation*}
		\begin{tikzcd}[font=\small]
			t^* p^* R_Y(s_* s^* A) \arrow{d}{\theta} & f^* s^* R_Y(s_* s^* A_Y) \arrow[equal]{l} & f^* s^* R_Y(A_Y) \arrow{l}{\eta} \arrow[equal]{d} \\
			t^* R_X(p^* s_* s^* A_Y) \isoarrow{d} && t^* p^* R_Y(A_Y) \arrow{ull}{\eta} \arrow{d}{\theta} \\
			t^* R_X(t_* f^* s^* A_Y) \arrow[equal]{d} && t^* R_X(p^* A_Y) \arrow{ull}{\eta} \arrow{dll}{\eta} \arrow{d}{A^*} \\
			t^* R_X(t_* t^* p^* A_Y) \arrow{r}{A^*} & t^* R_X(t_* t^* A_X) & t^* R_X(A_X) \arrow{l}{\eta}
		\end{tikzcd}
	\end{equation*}
	where the central triangular piece is commutative by construction while the other three pieces are commutative by naturality. This concludes the proof.
\end{proof}

\section{Extending monoidal structures}\label{sect_ext-boxtimes}

Throughout this section, we fix a localic triangulated $\Scal$-fibered category $\Hbb$. Our goal is to extend monoidal structures defined on the underlying $\Ucal$-fibered category to monoidal structures over the whole of $\Scal$; similarly to the case of morphisms of $\Scal$-fibered categories treated in the previous section, this extension procedure only works nicely for suitable monoidal structures, as explained below. For convenience, we formulate our results in the language of internal tensor structures employed in \cite{Ter23Fib}. 

The first part of the present section is devoted to extending internal tensor structures; although we consider the extension result for monoidal structures as the most important result of the present paper for our applications, we prefer to omit most of the details in the proofs as the argument is formally analogous to the one for the proof of \Cref{thm_ext}. In the second and longest part, we show in detail how associativity, commutativity and unit constraints admit a similar extension from $\Ucal$ to $\Scal$. In the last part, we check that the natural mutual compatibility conditions among these constraints are preserved under extension.

\subsection{Extending tensor structures}

In the first place, we are interested in studying \textit{internal tensor structures} $(\otimes,m)$ on $\Hbb$ over $\Scal$: as in \cite[Defn.~2.1]{Ter23Fib}, by this we mean the datum of
\begin{itemize}
	\item for every $S \in \Scal$, a functor
	\begin{equation*}
		-\otimes- = -\otimes_S-: \Hbb(S) \times \Hbb(S) \rightarrow \Hbb(S),
	\end{equation*}
	called the \textit{internal tensor product functor} over $S$,
	\item for every morphism $f: T \rightarrow S$ in $\Scal$, a natural isomorphism of functors $\Hbb(S) \times \Hbb(S) \rightarrow \Hbb(T)$
	\begin{equation*}
		m = m_f: f^* A \otimes f^* B \xrightarrow{\sim} f^*(A \otimes B),
	\end{equation*}
	called the \textit{internal monoidality isomorphism} along $f$
\end{itemize}
satisfying the following condition:
\begin{enumerate}
	\item[\hypertarget{mITS}{($m$ITS)}] For every choice of composable morphisms $f: T \rightarrow S$, $g: S \rightarrow V$ in $\Scal$, the diagram of functors $\Hbb(V) \times \Hbb(V) \rightarrow \Hbb(T)$
	\begin{equation*}
		\begin{tikzcd}
			(gf)^* A \otimes (gf)^* B \arrow{rr}{m} \arrow[equal]{d} && (gf)^*(A \otimes B)  \arrow[equal]{d} \\
			f^* g^* A \otimes f^* g^* B \arrow{r}{m} & f^* (g^* A \otimes g^* B) \arrow{r}{m} & f^* g^* (A \otimes B)
		\end{tikzcd}
	\end{equation*}
	is commutative.
\end{enumerate}
Since $\Hbb$ is a triangulated $\Scal$-fibered category, it is natural to only consider those internal tensor structures $(\otimes,m)$ for which each functor $- \otimes -: \Hbb(S) \times \Hbb(S) \rightarrow \Hbb(S)$ is triangulated compatibly with inverse image functors: in this way, one obtains the natural notion of \textit{triangulated internal tensor structure} (see \cite[\S~9]{Ter23Fib}). Since moreover $\Hbb$ is localic, it is natural to further restrict our attention to those triangulated internal tensor structures $(\otimes,m)$ satisfying the \textit{projection formulae}: for every smooth morphism $p: P \rightarrow S$ in $\Scal$, the natural transformation of functors $\Hbb(P) \times \Hbb(S) \rightarrow \Hbb(S)$
\begin{equation*}
	p_{\#}(A \otimes p^* B) \xrightarrow{\eta} p_{\#}(p^* p_{\#} A \otimes p^* B) \xrightarrow{m} p_{\#} p^*(p_{\#} A \otimes B) \xrightarrow{\epsilon} p_{\#} A \otimes B
\end{equation*}
and the natural transformation of functors $\Hbb(S) \times \Hbb(P) \rightarrow \Hbb(S)$
\begin{equation*}
	p_{\#}(p^* A \otimes B) \xrightarrow{\eta} p_{\#}(p^* A \otimes p^* p_{\#} B) \xrightarrow{m} p_{\#} p^*(A \otimes p_{\#} B) \xrightarrow{\epsilon} A \otimes p_{\#} B
\end{equation*}
are both invertible. We say that such a triangulated internal tensor structure is \textit{localic}. Note that, using the construction of \cite[\S~7]{Ter23Fib}, this notion corresponds to the notion of \textit{localic external tensor structure} introduced in \cite[\S~5]{Ter23Fact}.

Note that, using \Cref{rem:Usmcl}, any (localic, triangulated) internal tensor structure on $\Hbb$ defines by restriction a (localic, triangulated) internal tensor structure on the underlying $\Ucal$-fibered category. In the reverse direction, we have the following result:

\begin{thm}\label{thm:ext-otimes}
	Let $(\otimes,m)$ be a localic internal tensor structure on the $\Ucal$-fibered category $\Hbb$. Then there exists a unique (up to unique equivalence) triangulated internal tensor structure $(\tilde{\otimes},\tilde{m})$ on the $\Scal$-fibered category $\Hbb$ whose restriction to $\Ucal$ coincides with $(\otimes,m)$. Moreover, the internal tensor structure $(\tilde{\otimes},\tilde{m})$ is automatically localic.
\end{thm}

The proof of this result follows the same lines as the proof of \Cref{thm_ext}. For sake of brevity, we only indicate how to adapt the arguments of the various intermediate results.

\begin{lem}\label{lem:ext-otimes}
	Fix $S \in \Scal$. For every embedding $(Y;s) \in \Emb_{\Ucal}(S)$, define the functor
	\begin{equation}\label{formula_otimes_s}
		- \otimes_s -: \Hbb(S) \times \Hbb(S) \rightarrow \Hbb(S), \quad A \otimes_s B := s^*(s_* A \otimes s_* B).
	\end{equation}
	Then:
	\begin{enumerate}
		\item For every morphism $p: (Y';s') \rightarrow (Y;s)$ in $\Emb_{\Ucal}(S)$, the natural transformation of functors $\Hbb(S) \times \Hbb(S) \rightarrow \Hbb(S)$
		\begin{equation*}
			\phi_p: \quad A \otimes_s B \rightarrow A \otimes_{s'} B
		\end{equation*}
		defined by taking the composite
		\begin{equation*}
			\begin{tikzcd}[font=\small]
				A \otimes_s B \arrow[equal]{d} &&& A \otimes_{s'} B \arrow[equal]{d} \\
				{s}^*(s_* A \otimes s_* B) \arrow[equal]{r} & {s'}^* p^*(p_* s'_* A \otimes p_* s'_* B) & 
				{s'}^*(p^* p_* s'_* A \otimes p^* p_* s'_* B) \arrow{l}{m} \arrow{r}{\epsilon} & {s'}^* (s'_* A \otimes s'_* B) \\
			\end{tikzcd}
		\end{equation*}
		is invertible.
		\item For every pair of composable morphisms $q: (Y'';s'') \rightarrow (Y';s')$, $p: (Y';s') \rightarrow (Y;s)$ in $\Emb_{\Ucal}(S)$, we have the equality between natural isomorphisms of functors $\Hbb(S) \times \Hbb(S) \rightarrow \Hbb(S)$
		\begin{equation*}
			\phi_{p \circ q} = \phi_q \circ \phi_p.
		\end{equation*}
		\item For every choice of embeddings $(Y^{(i)};s^{(i)}) \in \Emb_{\Ucal}(S)$, $i = 1,2$, consider the natural isomorphism of functors $\Hbb(S) \times \Hbb(S) \rightarrow \Hbb(S)$
		\begin{equation*}
			\psi_{1,2}: A \otimes_{s^{(1)}} B \xrightarrow{\phi_{p^{(1)}}} A \otimes_{s^{(12)}} B \xrightarrow{\phi_{p^{(2)}}^{-1}} A  \otimes_{s^{(2)}} B,
		\end{equation*}
		where $s^{(12)}: S \rightarrow Y^{(1)} \times Y^{(2)}$ denotes the diagonal closed immersion while $p^{(r)}: Y^{(1)} \times Y^{(2)} \rightarrow Y^{(i)}$, $i = 1,2$, denote the canonical projections. Then the following cocycle condition holds:
		
		For every three given embeddings $(Y^{(i)};s^{(i)}) \in \Emb_{\Ucal}(S)$, $i = 1,2,3$, we have (with obvious notation) the equality between natural isomorphisms of functors $\Hbb(S) \times \Hbb(S) \rightarrow \Hbb(S)$
		\begin{equation*}
			\psi_{1,3} = \psi_{2,3} \circ \psi_{1,2}.
		\end{equation*}
	\end{enumerate}
\end{lem}
\begin{proof}
	\begin{enumerate}
		\item It suffices to show that the co-unit arrow $\epsilon: {s'}^*(p^* p_* s'_* A \otimes p^* p_* s'_* B) \rightarrow {s'}^* (s'_* A \otimes s'_* B)$ is invertible or, even better, that the two single co-unit arrows $\epsilon: {s'}^*(p^* p_* s'_* A \otimes p^* p_* s'_* B) \rightarrow {s'}^* (p^* p_* s'_* A \otimes s'_* B)$ and $\epsilon: {s'}^*(p^* p_* s'_* A \otimes s'_* B) \rightarrow {s'}^* (s'_* A \otimes s'_* B)$ are both invertible. This can be done as for \Cref{lem_emb}(1).
		\item This can be proved by the same argument used for \Cref{lem_emb}(2).
		\item This follows formally from the previous point, as in the case of \Cref{lem_emb}(3).
	\end{enumerate}
\end{proof}

\begin{cor}\label{cor:ext-otimes}
	For every $S \in \Scal$, the formula
	\begin{equation}\label{formula:otimes-ext}
		A \tilde{\otimes}_S B := \varinjlim_{(Y;s) \in \Gcal(S)} A \otimes_s B
	\end{equation}
	defines a bi-triangulated functor $- \tilde{\otimes} - = - \tilde{\otimes}_S -: \Hbb(S) \times \Hbb(S) \rightarrow \Hbb(S)$ which is canonically isomorphic to each of the functors \eqref{formula_otimes_s} (compatibly with the transition maps of \Cref{lem:ext-otimes}(3)).
\end{cor}
\begin{proof}
	This follows immediately from \Cref{lem:ext-otimes}(3), in the same way as \Cref{cor_emb} follows from \Cref{lem_emb}(3).
\end{proof}

\begin{lem}\label{lem:ext-monoint}
	Fix a morphism $f: T \rightarrow S$ in $\Scal$. For every embedding $(Y;s) \in \Emb_{\Ucal}(S)$ and every factorization $(X;t,p) \in \Fact_{\Scal}(s \circ f)$, define a natural transformation of functors $\Hbb(S) \times \Hbb(S) \rightarrow \Hbb(T)$
	\begin{equation}\label{m_f^stp}
		m_f^{(s,t,p)}: f^* A \otimes_t f^* B \xrightarrow{\sim} f^* (A \otimes_s B)
	\end{equation}
	by taking the composite
	\begin{equation*}
		\begin{tikzcd}[font=\small]
			f^* A \otimes_t f^* B \arrow[equal]{d} &&& f^* (A \otimes_s B) \arrow[equal]{d} \\
			t^* (t_* f^* A \otimes t_* f^* B) & t^* (p^* s_* A \otimes p^* s_* B) \arrow{l}{\sim} \arrow{r}{m} & t^* p^* (s_* A \otimes s_* B) \arrow[equal]{r} & f^* s^* (s_* A \otimes s_* B) 
		\end{tikzcd}
	\end{equation*}
    where the unnamed natural isomorphism is the composite
    \begin{equation}\label{iso:aux-otimes}
    	t^* (p^* s_* A \otimes p^* s_* B) \xrightarrow{\eta} t^* (t_* t^* p^* s_* A \otimes t_* t^* p^* s_* B) = t^* (t_* f^* s^* s_* A \otimes t_* f^* s^* s_* B) \xrightarrow{\epsilon} t^* (t_* f^* A \otimes t_* f^* B).
    \end{equation}
    Then $m_f^{(s,t,p)}$ is invertible, and the induced natural isomorphism of functors $\Hbb(S) \times \Hbb(S) \rightarrow \Hbb(T)$
    \begin{equation}\label{isoUS:mono-int}
    	m = m_f: f^* A \otimes f^* B = f^* A \otimes_t f^* B \xrightarrow{m_f^{(s,t,p)}} f^* (A \otimes_s B) = f^* (A \otimes B)
    \end{equation}
    is independent of the choice of $(Y;s) \in \Emb_{\Ucal}(S)$ and of $(X;t,p) \in \Fact_{\Scal}(s \circ f)$.
\end{lem}
\begin{proof}
	In the first place, one needs to check that the composite \eqref{iso:aux-otimes} is indeed invertible. This amounts to checking that the unit arrow $\eta: t^* (p^* s_* A \otimes p^* s_* B) \rightarrow t^* (t_* t^* p^* s_* A \otimes t_* t^* p^* s_* B)$ is invertible or, better, that the two single unit arrows $t^* (p^* s_* A \otimes p^* s_* B) \xrightarrow{\eta} t^* (t_* t^* p^* s_* A \otimes p^* s_* B)$ and $t^* (t_* t^* p^* s_* A \otimes p^* s_* B) \xrightarrow{\eta} t^* (t_* t^* p^* s_* A \otimes t_* t^* p^* s_* B)$ are both invertible. This can be done as usual via the associated localization triangles. The rest of the proof follows by the same argument used for \Cref{lem_prep-thm-mor-ext}.
\end{proof}

\begin{proof}[Proof of \Cref{thm:ext-otimes}]
	One needs to check that the natural isomorphisms \eqref{m_f^stp} turn the collection of bi-triangulated functors \eqref{formula:otimes-ext} into an internal tensor structure on $\Hbb$: this amounts to showing that they satisfy condition ($m$ITS), which can be checked by suitably adapting the argument used for \Cref{thm_ext}. The statements about the restriction to $\Ucal$ and the uniqueness of the internal tensor structure $(\tilde{\otimes},\tilde{m})$ follow by easy considerations analogous to those for \Cref{thm_ext}; we leave the details to the interested reader.
\end{proof}

\subsection{Extending associativity, commutativity, and unit constraints}

Keeping the notation of the previous subsection, we suppose to be given a localic internal tensor structure $(\otimes,m)$ on the $\Ucal$-fibered category $\Hbb$; we let $(\tilde{\otimes},\tilde{m})$ denote the internal tensor structure over $\Scal$ obtained as in \Cref{sect_ext-boxtimes}. We want to show that internal associativity, commutativity and unit constraints on $(\otimes,m)$ extend to analogous constraints on $(\tilde{\otimes},\tilde{m})$. In each case, we construct the constraints over $\Scal$ by exploiting the embedding categories $\Emb_{\Ucal}(S)$.

We start by studying \textit{internal associativity constraints} $a$ on $(\tilde{\otimes},\tilde{m})$: as in \cite[Defn.~3.1]{Ter23Fib}, by this we mean the datum of
\begin{itemize}
	\item for every $S \in \Scal$, a natural isomorphism of functors $\Hbb(S) \times \Hbb(S) \times \Hbb(S) \rightarrow \Hbb(S)$
	\begin{equation*}
		a = a_S: (A \tilde{\otimes} B) \tilde{\otimes} C \xrightarrow{\sim} A \tilde{\otimes} (B \tilde{\otimes} C)
	\end{equation*}
\end{itemize}
satisfying the following conditions:
\begin{enumerate}
	\item[($a$ITS-1)] For every $S \in \Scal$, the diagram of functors $\Hbb(S) \times \Hbb(S) \times \Hbb(S) \times \Hbb(S) \rightarrow \Hbb(S)$
	\begin{equation*}
		\begin{tikzcd}
			((A \tilde{\otimes} B) \tilde{\otimes} C) \tilde{\otimes} D \arrow{d}{a} \arrow{rr}{a} && (A \tilde{\otimes} B) \tilde{\otimes} (C \tilde{\otimes} D) \arrow{d}{a} \\
			(A \tilde{\otimes} (B \tilde{\otimes} C)) \tilde{\otimes} D \arrow{r}{a} & A \tilde{\otimes} ((B \tilde{\otimes} C) \tilde{\otimes} D) \arrow{r}{a} & A \tilde{\otimes} (B \tilde{\otimes} (C \tilde{\otimes} D))
		\end{tikzcd}
	\end{equation*}
	is commutative.
	\item[($a$ITS-2)] For every morphism $f: T \rightarrow S$ in $\Scal$, the diagram of functors $\Hbb(S) \times \Hbb(S) \times \Hbb(S) \rightarrow \Hbb(T)$
	\begin{equation*}
		\begin{tikzcd}
			(f^* A \tilde{\otimes} f^* B) \tilde{\otimes} f^* C \arrow{r}{\tilde{m}} \arrow{d}{a} & f^* (A \tilde{\otimes} B) \tilde{\otimes} f^* C \arrow{r}{\tilde{m}} & f^* ((A \tilde{\otimes} B) \tilde{\otimes} C) \arrow{d}{a} \\
			f^* A \tilde{\otimes} (f^* B \tilde{\otimes} f^* C) \arrow{r}{\tilde{m}} & f^* A \tilde{\otimes} f^* (B \tilde{\otimes} C) \arrow{r}{\tilde{m}} & f^* (A \tilde{\otimes} (B \tilde{\otimes} C))
		\end{tikzcd}
	\end{equation*}
	is commutative.
\end{enumerate}
Note that every internal associativity constraint on $(\tilde{\otimes},\tilde{m})$ defines an analogous constraint on the internal tensor structure $(\otimes,m)$ over $\Ucal$. In the reverse direction, we have the following result: 

\begin{prop}\label{prop:asso_int-ext}
	Let $a$ be an internal associativity constraint on $(\otimes,m)$. Then there exists a unique internal associativity constraint $\tilde{a}$ on $(\tilde{\otimes},\tilde{m})$ whose restriction to $\Ucal$ coincides with $a$.
\end{prop}

The proof is based on the following construction:

\begin{lem}\label{lem:asso_int-ext}
	Fix $S \in \Scal$. For every embedding $(Y;s) \in \Emb_{\Ucal}(S)$, define a natural isomorphism of functors $\Hbb(S) \times \Hbb(S) \times \Hbb(S) \rightarrow \Hbb(S)$
	\begin{equation}\label{a_S^s}
		a^{(s)} = a^{(s)}_{S}: (A \otimes_s B) \otimes_s C \xrightarrow{\sim} A \otimes_s (B \otimes_s C)
	\end{equation}
	by taking the composite
	\begin{equation*}
		\begin{tikzcd}[font=\small]
			(A \otimes_s B) \otimes_s C \arrow[equal]{d} &&& A \otimes_s (B \otimes_s C). \arrow[equal]{d} \\
			s^*(s_* s^* (s_* A \otimes s_* B) \otimes s_* C) & s^* ((s_* A \otimes s_* B) \otimes s_* C) \arrow{r}{a} \arrow{l}{\eta} & s^* (s_* A \otimes (s_* B \otimes s_* C)) \arrow{r}{\eta} & s^* (s_* A \otimes s_* s^* (s_* B \otimes s_* C)) \\
		\end{tikzcd}
	\end{equation*}
	Then the induced natural isomorphism of functors $\Hbb(S) \times \Hbb(S) \times \Hbb(S) \rightarrow \Hbb(S)$
	\begin{equation}\label{asso-int}
		\tilde{a} = \tilde{a}_S: (A \tilde{\otimes} B) \tilde{\otimes} C = (A \otimes_s B) \otimes_s C \xrightarrow{a^{(s)}} A \otimes_s (B \otimes_s C) = A \tilde{\otimes} (B \tilde{\otimes} C)
	\end{equation}
	is independent of the choice of $(Y;s) \in \Emb_{\Ucal}(S)$.
\end{lem}
\begin{proof}
	The fact that the unit arrows $\eta: s^* ((s_* A \otimes s_* B) \otimes s_* C) \rightarrow s^*(s_* s^* (s_* A \otimes s_* B) \otimes s_* C)$ and $\eta: s^* (s_* A \otimes (s_* B \otimes s_* C)) \rightarrow s^* (s_* A \otimes s_* s^* (s_* B \otimes s_* C))$ are both invertible can be checked as usual via the associated localization triangles; we omit the details. It follows that the natural transformation \eqref{a_S^s} is indeed well-defined and invertible.
	 
	We are left to showing the independence statement. Using the connectedness of the category $\Emb_{\Ucal}(S)$ via direct products (\Cref{lem:Emb}), we are reduced to the following claim: For every morphism $p: (Y';s') \rightarrow (Y;s)$ in $\Emb_{\Ucal}(S)$, the diagram of functors $\Hbb(S) \times \Hbb(S) \times \Hbb(S) \rightarrow \Hbb(S)$
	\begin{equation*}
		\begin{tikzcd}[font=\small]
			(A \otimes_s B) \otimes_s C \arrow{r}{a^{(s)}} \arrow{d}{\phi_p} & A \otimes_s (B \otimes_s C) \arrow{d}{\phi_p} \\
			(A \otimes_s B) \otimes_{s'} C \arrow{d}{\phi_p} & A \otimes_{s'} (B \otimes_s C) \arrow{d}{\phi_p} \\
			(A \otimes_{s'} B) \otimes_{s'} C \arrow{r}{a^{(s')}} & A \otimes_{s'} (B \otimes_{s'} C)
		\end{tikzcd}
	\end{equation*}
	is commutative. Unwinding the various definitions, we obtain the more explicit diagram
	\begin{equation*}
		\begin{tikzcd}[font=\small]
			s^* ((s_* A \otimes s_* B) \otimes s_* C) \arrow{r}{a} \arrow{d}{\eta} & s^* (s_* A \otimes (s_* B \otimes s_* C)) \arrow{d}{\eta} \\
			s^* (s_* s^* (s_* A \otimes s_* B) \otimes s_* C) \arrow[equal]{d} & s^* (s_* A \otimes s_* s^* (s_* B \otimes s_* C)) \arrow[equal]{d} \\
			{s'}^* p^* (p_* s'_* s^* (s_* A \otimes s_* B) \otimes p_* s'_* C) & {s'}^* p^* (p_* {s'}_* A \otimes p_* {s'}_* s^* (s_* B \otimes s_* C)) \\
			{s'}^* (p^* p_* s'_* s^* (s_* A \otimes s_* B) \otimes p^* p_* s'_* C) \arrow{u}{m} \arrow{d}{\epsilon} & {s'}^* (p^* p_* {s'}_* A \otimes p^* p_* {s'}_* s^* (s_* B \otimes s_* C)) \arrow{u}{m} \arrow{d}{\epsilon} \\
			{s'}^* (s'_* s^* (s_* A \otimes s_* B) \otimes s'_* C) \arrow[equal]{d} & {s'}^* (s'_* A \otimes s'_* s^* (s_* B \otimes s_* C)) \arrow[equal]{d} \\
			{s'}^* (s'_* {s'}^* p^* (p_* s'_* A \otimes p_* s'_* B) \otimes s'_* C) & {s'}^* (s'_* A \otimes s'_* {s'}^* p^* (p_* s'_* B \otimes p_* s'_* C)) \\
			{s'}^* (s'_* {s'}^* (p^* p_* s'_* A \otimes p^* p_* s'_* B) \otimes s'_* C) \arrow{u}{m} \arrow{d}{\epsilon} & {s'}^* (s'_* A \otimes s'_* {s'}^* (p^* p_* s'_* B \otimes p^* p_* s'_* C)) \arrow{u}{m} \arrow{d}{\epsilon} \\
			{s'}^* (s'_* {s'}^* (s'_* A \otimes s'_* B) \otimes s'_* C) & {s'}^* (s'_* A \otimes s'_* {s'}^* (s'_* B \otimes s'_* C)) \\
			{s'}^* ((s'_* A \otimes s'_* B) \otimes s'_* C) \arrow{r}{a} \arrow{u}{\eta} & {s'}^*(s'_* A \otimes (s'_* B \otimes s'_* C)) \arrow{u}{\eta}
		\end{tikzcd}
	\end{equation*}
	that we decompose as
	\begin{equation*}
		\begin{tikzcd}[font=\small]
			\bullet \arrow{rrr}{a} \arrow{d}{\eta} &&& \bullet \arrow{d}{\eta} \\
			\bullet \arrow[equal]{d} & {s'}^* p^* ((p_* {s'}_* A \otimes p_* {s'}_* B) \otimes p_* {s'}_* C) \arrow{r}{a} \arrow[equal]{ul} & {s'}^* p^* (p_* {s'}_* A \otimes (p_* {s'}_* B \otimes p_* {s'}_* C)) \arrow[equal]{ur} & \bullet \arrow[equal]{d} \\
			\bullet &&& \bullet \\
			\bullet \arrow{u}{m} \arrow{d}{\epsilon} & {s'}^* (p^* (p_* {s'}_* A \otimes p_* {s'}_* B) \otimes p^* p_* {s'}_* C) \arrow{uu}{m} & {s'}^* (p^* p_* {s'}_* A \otimes p^* (p_* {s'}_* B) \otimes p_* {s'}_* C) \arrow{uu}{m} & \bullet \arrow{u}{m} \arrow{d}{\epsilon} \\
			\bullet \arrow[equal]{d} &&& \bullet \arrow[equal]{d} \\
			\bullet & {s'}^* ((p^* p_* {s'}_* A \otimes p^* p_* {s'}_* B) \otimes p^* p_* {s'}_* C) \arrow{uu}{m} \arrow{r}{a} \arrow{dd}{\epsilon} & {s'}^* (p^* p_* {s'}_* A \otimes (p^* p_* {s'}_* B \otimes p^* p_* {s'}_* C)) \arrow{uu}{m} \arrow{dd}{\epsilon} & \bullet \\
			\bullet \arrow{u}{m} \arrow{d}{\epsilon} &&& \bullet \arrow{u}{m} \arrow{d}{\epsilon} \\
			\bullet & {s'}^* ((p^* p_* {s'}_* A \otimes p^* p_* {s'}_* B) \otimes {s'}_* C) \arrow{dl}{\epsilon} & {s'}^* ({s'}_* A \otimes (p^* p_* {s'}_* B \otimes p^* p_* {s'}_* C)) \arrow{dr}{\epsilon} & \bullet \\
			\bullet \arrow{rrr}{a} \arrow{u}{\eta} &&& \bullet \arrow{u}{\eta}
		\end{tikzcd}
	\end{equation*}
	Here, the upper piece is commutative by naturality, the central piece is commutative by axiom ($a$ITS-2) over $\Ucal$, and the lower piece is commutative by construction. Hence it suffices to show that the left-most and right-most pieces are commutative as well. Let us show, for example, that the left-most piece is commutative; the commutativity of the right-most piece will then follow by symmetry. We can decompose the left-most piece as
	\begin{equation*}
		\begin{tikzcd}[font=\small]
			\bullet \arrow{d}{\eta} \\
			\bullet \arrow[equal]{d} &&& \bullet \arrow[equal]{ulll} \arrow{dl}{\eta} \\
			\bullet & {s'}^* p^* (p_* {s'}_* {s'}^* p^* (p_* {s'}_* A \otimes p_* {s'}_* B) \otimes p_* {s'}_* C) \arrow[equal]{l} & {s'}^* p^* (p_* p^* (p_* {s'}_* A \otimes p_* {s'}_* B) \otimes p_* {s'}_* C) \arrow{l}{\eta} \\
			\bullet \arrow{u}{m} \arrow{d}{\epsilon} &&& \bullet \arrow{uu}{m} \arrow{dl}{\eta} \\
			\bullet \arrow[equal]{d} & {s'}^* (p^* p_* {s'}_* {s'}^* p^* (p_* {s'}_* A \otimes p_* {s'}_* B) \otimes p^* p_* {s'}_* C) \arrow[equal]{ul} \arrow{dl}{\epsilon} \arrow{uu}{m} & {s'}^* (p^* p_* p^* (p_* {s'}_* A \otimes p_* {s'}_* B) \otimes p^* p_* {s'}_* C) \arrow{uu}{m} \arrow{l}{\eta} \\
			\bullet &&& \bullet \arrow{uu}{m} \arrow{dd}{\epsilon} \arrow{dl}{\eta} \\
			\bullet \arrow{u}{m} \arrow{d}{\epsilon} & {s'}^* (p^* p_* {s'}_* {s'}^* (p^* p_* {s'}_* A \otimes p^* p_* {s'}_* B) \otimes p^* p_* {s'}_* C) \arrow{l}{\epsilon} \arrow{uu}{m} & {s'}^* (p^* p_* (p^* p_* {s'}_* A \otimes p^* p_* {s'}_* B) \otimes p^* p_* {s'}_* C) \arrow{l}{\eta} \arrow{uu}{m} \arrow{dr}{\epsilon} \\
			\bullet &&& \bullet \arrow{dlll}{\epsilon} \arrow{ulll}{\epsilon} \\
			\bullet \arrow{u}{\eta}
		\end{tikzcd}
	\end{equation*}
	Here, all pieces are commutative by naturality and construction. This proves the claim.
\end{proof}

\begin{proof}[Proof of \Cref{prop:asso_int-ext}]
	We only prove that the natural isomorphisms \eqref{asso-int} define an internal associativity constraint on $(\tilde{\otimes},\tilde{m})$; the statements about its restriction to $\Ucal$ and its uniqueness are left to the interested reader.
	
	Let us check that the natural isomorphisms \eqref{asso-int} satisfy conditions ($a$ITS-1) and ($a$ITS-2).
	We start from condition ($a$ITS-1): given $S \in \Scal$, we have to show that the diagram of functors $\Hbb(S) \times \Hbb(S) \times \Hbb(S) \times \Hbb(S) \rightarrow \Hbb(S)$
	\begin{equation*}
		\begin{tikzcd}
			((A \tilde{\otimes} B) \tilde{\otimes} C) \tilde{\otimes} D \arrow{rr}{\tilde{a}} \arrow{d}{\tilde{a}} && (A \tilde{\otimes} B) \tilde{\otimes} (C \tilde{\otimes} D) \arrow{d}{\tilde{a}} \\
			(A \tilde{\otimes} (B \tilde{\otimes} C)) \tilde{\otimes} D \arrow{r}{\tilde{a}} & A \tilde{\otimes} ((B \tilde{\otimes} C) \tilde{\otimes} D) \arrow{r}{\tilde{a}} & A \tilde{\otimes} (B \tilde{\otimes} (C \tilde{\otimes} D))
		\end{tikzcd}
	\end{equation*}
	is commutative. To this end, choose an embedding $(Y;s) \in \Emb_{\Ucal}(S)$. Then, by construction, it suffices to show that the diagram
	\begin{equation*}
		\begin{tikzcd}[font=\small]
			((A \otimes_s B) \otimes_s C) \otimes_s D \arrow{rr}{a_S^{(s)}} \arrow{d}{a_S^{(s)}} && (A \otimes_s B) \otimes_s (C \otimes_s D) \arrow{d}{a_S^{(s)}} \\
			(A \otimes_s (B \otimes_s C)) \otimes_s D \arrow{r}{a_S^{(s)}} & A \otimes_s ((B \otimes_s C) \otimes_s D) \arrow{r}{a_S^{(s)}} & A \otimes_s (B \otimes_s (C \otimes_s D))
		\end{tikzcd}
	\end{equation*}
	is commutative. Unwinding the various definitions, we obtain the more explicit diagram
	\begin{equation*}
		\begin{tikzcd}[font=\small]
			s^* ((s_* s^* (A' \otimes B') \otimes C') \otimes D') \arrow{rr}{a} \arrow{d}{\eta} && s^* (s_* s^* (A' \otimes B') \otimes (C' \otimes D')) \arrow{d}{\eta} \\
			s^* (s_* s^* (s_* s^* (A' \otimes B') \otimes C') \otimes D') && s^* (s_* s^* (A' \otimes B') \otimes s_* s^* (C' \otimes D')) \\
			s^* (s_* s^* ((A' \otimes B') \otimes C') \otimes D') \arrow{u}{\eta} \arrow{d}{a} && s^* ((A' \otimes B') \otimes s_* s^* (C' \otimes D')) \arrow{u}{\eta} \arrow{d}{a} \\
			s^* (s_* s^* (A' \otimes (B' \otimes C')) \otimes D') \arrow{d}{\eta} && s^* (A' \otimes (B' \otimes s_* s^* (C' \otimes D'))) \arrow{d}{\eta} \\
			s^* (s_* s^* (A' \otimes s_* s^* (B' \otimes C')) \otimes D') && s^* (A' \otimes s_* s^* (B' \otimes s_* s^* (C' \otimes D'))) \\
			s^* ((A' \otimes s_* s^* (B' \otimes C')) \otimes D') \arrow{u}{\eta} \arrow{d}{a} && s^* (A' \otimes s_* s^* (B' \otimes (C' \otimes D'))) \arrow{u}{\eta} \arrow{d}{a} \\
			s^* (A' \otimes (s_* s^* (B' \otimes C') \otimes D')) \arrow{r}{\eta} & s^* (A' \otimes s_* s^* (s_* s^* (B' \otimes C') \otimes D')) & s^* (A' \otimes s_* s^* ((B' \otimes C') \otimes D')) \arrow{l}{\eta}
		\end{tikzcd}
	\end{equation*}
    where, for notational convenience, we have set $A' := s_* A$, $B' := s_* B$, $C' := s_* C$ and $D' := s_* D$. We can decompose the latter diagram as
    that we decompose as
    \begin{equation*}
    	\begin{tikzcd}[font=\small]
    		\bullet \arrow{rrrr}{a} \arrow{d}{\eta} &&&& \bullet \arrow{d}{\eta} \\
    		\bullet & s^* (((A' \otimes B') \otimes C') \otimes D') \arrow{rr}{a} \arrow{ddd}{a} \arrow{ul}{\eta} \arrow{dl}{\eta} && s^* ((A' \otimes B') \otimes (C' \otimes D')) \arrow{ddd}{a} \arrow{ur}{\eta} \arrow{dr}{\eta} & \bullet \\
    		\bullet \arrow{u}{\eta} \arrow{d}{a} &&&& \bullet \arrow{u}{\eta} \arrow{d}{a} \\
    		\bullet \arrow{d}{\eta} &&&& \bullet \arrow{d}{\eta} \\
    		\bullet & s^* ((A' \otimes (B' \otimes C')) \otimes D') \arrow{r}{a} \arrow{ul}{\eta} \arrow{dl}{\eta} & s^* (A' \otimes ((B' \otimes C') \otimes D')) \arrow{r}{a} \arrow{ddll}{\eta} \arrow{ddrr}{\eta} & s^* (A' \otimes (B' \otimes (C' \otimes D'))) \arrow{ur}{\eta} \arrow{dr}{\eta} & \bullet \\
    		\bullet \arrow{u}{\eta} \arrow{d}{a} &&&& \bullet \arrow{u}{\eta} \arrow{d}{a} \\
    		\bullet \arrow{rr}{\eta} && \bullet && \bullet \arrow{ll}{\eta}
    	\end{tikzcd}
    \end{equation*}
	Here, the central rectangle is commutative by axiom ($a$ITS-1) over $\Ucal$, while all the remaining pieces are commutative by naturality. This proves the claim.
	
	We now check condition ($a$ITS-2): given a morphism $f: T \rightarrow S$ in $\Scal$, we have to show that the diagram of functors $\Hbb(S) \times \Hbb(S) \times \Hbb(S) \rightarrow \Hbb(T)$
	\begin{equation*}
		\begin{tikzcd}
			(f^* A \tilde{\otimes} f^* B) \tilde{\otimes} f^* C \arrow{r}{\tilde{m}} \arrow{d}{\tilde{a}} & f^* (A \tilde{\otimes} B) \tilde{\otimes} f^* C \arrow{r}{\tilde{m}} & f^* ((A \tilde{\otimes} B) \tilde{\otimes} C) \arrow{d}{\tilde{a}} \\
			f^* A \tilde{\otimes} (f^* B \tilde{\otimes} f^* C) \arrow{r}{\tilde{m}} & f^* A \tilde{\otimes} f^* (B \tilde{\otimes} C) \arrow{r}{\tilde{m}} & f^* (A \tilde{\otimes} (B \tilde{\otimes} C))
		\end{tikzcd}
	\end{equation*}
	is commutative. To this end, choose an embedding $(Y;s) \in \Emb_{\Ucal}(S)$ and then a factorization $(X;t,p) \in \Fact_{\Scal}(s \circ f)$; note that the object $X$ automatically belongs to $\Ucal$ (\Cref{lem:Ucal}(2)), and so we obtain an embedding $(X;t) \in \Emb_{\Ucal}(T)$. Then, by construction, it suffices to show that the diagram
	\begin{equation*}
		\begin{tikzcd}[font=\small]
			(f^* A \otimes_t f^* B) \otimes_t f^* C \arrow{r}{m_f^{(s,t,p)}} \arrow{d}{a_T^{(t)}} & f^* (A \otimes_s B) \otimes_t f^* C \arrow{r}{m_f^{(s,t,p)}} & f^* ((A \otimes_s B) \otimes_s C) \arrow{d}{a_S^{(s)}} \\
			f^* A \otimes_t (f^* B \otimes_t f^* C) \arrow{r}{m_f^{(s,t,p)}} & f^* A \otimes_t f^* (B \otimes_s C) \arrow{r}{m_f^{(s,t,p)}} & f^* (A \otimes_s (B \otimes_s C))
		\end{tikzcd}
	\end{equation*}
	is commutative. Unwinding the various definitions, we obtain the more explicit diagram
	\begin{equation*}
		\begin{tikzcd}[font=\small]
			t^* (t_* t^* p^* (s_* A \otimes s_* B) \otimes t_* f^* C) \arrow[equal]{r} & t^* (t_* f^* s^* (s_* A \otimes s_* B) \otimes t_* f^* C) & t^* (p^* s_* s^* (s_* A \otimes s_* B) \otimes p^* s_* C) \arrow{l}{\sim} \arrow{d}{m} \\
			t^* (t_* t^* (p^* s_* A \otimes p^* s_* B) \otimes t_* f^* C) \arrow{u}{m} \isoarrow{d} && t^* p^* (s_* s^* (s_* A \otimes s_* B) \otimes s_* C) \arrow[equal]{d} \\
			t^* (t_* t^* (t_* f^* A \otimes t_* f^* B) \otimes t_* f^* C) && f^* s^* (s_* s^* (s_* A \otimes s_* B) \otimes s_* C) \\
			t^* ((t_* f^* A \otimes t_* f^* B) \otimes t_* f^* C) \arrow{u}{\eta} \arrow{d}{a} && f^* s^* ((s_* A \otimes s_* B) \otimes s_* C) \arrow{u}{\eta} \arrow{d}{a} \\
			t^* (t_* f^* A \otimes (t_* f^* B \otimes t_* f^* C)) \arrow{d}{\eta} && f^* s^* (s_* A \otimes (s_* B \otimes s_* C)) \arrow{d}{\eta} \\
			t^* (t_* f^* A \otimes t_* t^* (t_* f^* B \otimes t_* f^* C)) && f^* s^* (s_* A \otimes s_* s^* (s_* B \otimes s_* C)) \\
			t^* (t_* f^* A \otimes t_* t^* (p^* s_* B \otimes p^* s_* C)) \isoarrow{u} \arrow{d}{m} && t^* p^* (s_* A \otimes s_* s^* (s_* B \otimes s_* C)) \arrow[equal]{u} \\
			t^* (t_* f^* A \otimes t_* t^* p^* (s_* B \otimes s_* C)) \arrow[equal]{r} & t^* (t_* f^* A \otimes t_* f^* s^* (s_* B \otimes s_* C)) & t^* (p^* s_* A \otimes p^* s_* s^* (s_* B \otimes s_* C)) \arrow{u}{m} \arrow{l}{\sim}
		\end{tikzcd}
	\end{equation*}
	that we decompose as
	\begin{equation*}
		\begin{tikzcd}[font=\small]
			\bullet \arrow[equal]{rr} && \bullet && \bullet \arrow{ll}{\sim} \arrow{d}{m} \\
			\bullet \arrow{u}{m} \isoarrow{d} & t^* (p^* (s_* A \otimes s_* B) \otimes t_* f^* C) \arrow{ul}{\eta} & t^* (t_* t^* p^* (s_* A \otimes s_* B) \otimes s_* C) \arrow{ull}{\sim} && \bullet \arrow[equal]{d} \\
			\bullet & t^* ((p^* s_* A \otimes p^* s_* B) \otimes t_* f^* C) \arrow{u}{m} \arrow{ul}{\eta} \arrow{dl}{\sim} &&& \bullet \\
			\bullet \arrow{u}{\eta} \arrow{d}{a} & t^* ((p^* s_* A \otimes p^* s_* B) \otimes p^* s_* C) \isoarrow{u} \arrow{r}{m} \arrow{d}{a} & t^* (p^* (s_* A \otimes s_* B) \otimes p^* s_* C) \arrow{uu}{\eta} \arrow{r}{m} \arrow{uuurr}{\eta} \arrow{uul}{\sim} & t^* p^* ((s_* A \otimes s_* B) \otimes s_* C) \arrow{d}{a} \arrow{uur}{\eta} \arrow[equal]{r} & \bullet \arrow{u}{\eta} \arrow{d}{a} \\
			\bullet \arrow{d}{\eta} & t^* (p^* s_* A \otimes (p^* s_* B \otimes p^* s_* C)) \isoarrow{d} \arrow{r}{m} & t^* (p^* s_* A \otimes p^* (s_* B \otimes s_* C)) \arrow{r}{m} \arrow{dd}{\eta} \arrow{ddl}{\sim} \arrow{dddrr}{\eta} & t^* p^* (s_* A \otimes (s_* B \otimes s_* C)) \arrow[equal]{r} \arrow{ddr}{\eta} & \bullet \arrow{d}{\eta} \\
			\bullet & t^* (t_* f^* A \otimes (p^* s_* B \otimes p^* s_* C)) \arrow{d}{m} \arrow{dl}{\eta} \arrow{ul}{\sim} &&& \bullet \\
			\bullet \isoarrow{u} \arrow{d}{m} & t^* (t_* f^* A \otimes p^* (s_* B \otimes s_* C)) \arrow{dl}{\eta} & t^* (p^* s_* A \otimes t_* t^* p^* (s_* B \otimes s_* C)) \arrow{dll}{\sim} && \bullet \arrow[equal]{u} \\
			\bullet \arrow[equal]{rr} && \bullet && \bullet \arrow{u}{m} \arrow{ll}{\sim}
		\end{tikzcd}
	\end{equation*}
    Here, the central rectangle is commutative by axiom ($a$ITS-2) over $\Ucal$, while all the remaining pieces are commutative by naturality and by construction. This proves the claim and concludes the proof.
\end{proof}

In the same spirit, we study \textit{internal commutativity constraints} $c$ on $(\otimes,m)$: as in \cite[Defn.~4.1]{Ter23Fib}, by this we mean the datum of
\begin{itemize}
	\item for every $S \in \Scal$, a natural isomorphism of functors $\Hbb(S) \times \Hbb(S) \rightarrow \Hbb(S)$
	\begin{equation*}
		c = c_S: A \tilde{\otimes} B \xrightarrow{\sim} B \tilde{\otimes} A
	\end{equation*}
\end{itemize}
satisfying the following condition
\begin{enumerate}
	\item[($c$ITS-1)] For every $S \in \Scal$, the diagram of functors $\Hbb(S) \times \Hbb(S) \rightarrow \Hbb(S)$
	\begin{equation*}
		\begin{tikzcd}
			A \tilde{\otimes} B \arrow{r}{c} \arrow{dr}{\id} & B \tilde{\otimes} A \arrow{d}{c} \\
			& A \tilde{\otimes} B
		\end{tikzcd}
	\end{equation*}
	is commutative.
	\item[($c$ITS-2)] For every morphism $f: T \rightarrow S$ in $\Scal$, the diagram of functors $\Hbb(S) \times \Hbb(S) \rightarrow \Hbb(T)$
	\begin{equation*}
		\begin{tikzcd}
			f^* A \tilde{\otimes} f^* B \arrow{r}{c} \arrow{d}{\tilde{m}} & f^* B \tilde{\otimes} f^* A \arrow{d}{\tilde{m}} \\
			f^* (A \tilde{\otimes} B)  \arrow{r}{c} & f^* (B \tilde{\otimes} A)
		\end{tikzcd}
	\end{equation*}
	is commutative.
\end{enumerate}
Note that every internal commutativity constraint on $(\otimes,m)$ defines by restriction an analogous constraint on the internal tensor structure $(\otimes,m)$ over $\Ucal$. In the reverse direction, we have the following result:

\begin{prop}\label{prop:comm_int-ext}
	Let $c$ be an internal commutativity constraint on $(\otimes,m)$. Then there exists a unique internal commutativity constraint $\tilde{c}$ on $(\tilde{\otimes},\tilde{m})$ whose restriction to $\Ucal$ coincides with $c$.
\end{prop}

As in the case of associativity constraints, the proof is based on an auxiliary construction:

\begin{lem}\label{lem:comm_int-ext}
	Fix $S \in \Scal$. For every embedding $(Y;s) \in \Emb_{\Ucal}(S)$, define a natural transformation of functors $\Hbb(S) \times \Hbb(S) \rightarrow \Hbb(S)$
	\begin{equation*}
		c_S^{(s)}: A \otimes_s B \xrightarrow{\sim} (B \otimes_s A)
	\end{equation*}
	by taking the composite
	\begin{equation*}
		A \otimes_s B = s^* (s_* A \otimes s_* B) \xrightarrow{c} s^* (s_* B \otimes s_* A) = (B \otimes_s A).
	\end{equation*}
	Then the induced isomorphism of functors $\Hbb(S) \times \Hbb(S) \rightarrow \Hbb(S)$
	\begin{equation}\label{comm-int}
		\tilde{c} = \tilde{c}_S: A \tilde{\otimes} B = A \otimes_s B \xrightarrow{c_S^{(s)}} (B \otimes_s A)  = (B \tilde{\otimes} A)
	\end{equation}
	is independent of the choice of $(Y;s) \in \Emb_{\Ucal}(S)$.
\end{lem}
\begin{proof}
	Using the connectedness of the category $\Emb_{\Ucal}(S)$ via direct products (\Cref{lem:Emb}), we are reduced to proving the following claim: For every morphism $p: (Y';s') \rightarrow (Y;s)$ in $\Emb_{\Ucal}(S)$, the diagram of functors $\Hbb(S) \times \Hbb(S) \rightarrow \Hbb(S)$
	\begin{equation*}
		\begin{tikzcd}
			A \otimes_s B \arrow{r}{c_S^{(s)}}  \arrow{d}{\phi_p} & B \otimes_s A \arrow{d}{\phi_p} \\
			A \otimes_{s'} B \arrow{r}{c_S^{(s')}} & B \otimes_{s'} A 
		\end{tikzcd}
	\end{equation*}
	is commutative. Unwinding the various definitions, we obtain the outer part of the diagram
	\begin{equation*}
		\begin{tikzcd}[font=\small]
			s^* (s_* A \otimes s_* B) \arrow{r}{c} \arrow[equal]{d} & s^* (s_* B \otimes s_* A) \arrow[equal]{d} \\
			{s'}^* p^* (p_* s'_* A \otimes p_* s'_* B) \arrow{r}{c} & {s'}^* p^* (p_* s'_* B \otimes p_* s'_* A) \\
			{s'}^* (p^* p_* s'_* A \otimes p^* p_* s'_* B) \arrow{u}{m} \arrow{d}{\epsilon} \arrow{r}{c} & {s'}^* (p^* p_* s'_* B \otimes p^* p_* s'_* A) \arrow{u}{m} \arrow{d}{\epsilon} \\
			{s'}^* (s'_* A \otimes s'_* B) \arrow{r}{c} & {s'}^* (s'_* B \otimes s'_* A)
		\end{tikzcd}
	\end{equation*}
	where all inner pieces are commutative by naturality. This proves the claim.
\end{proof}

\begin{proof}[Proof of \Cref{prop:comm_int-ext}]
	As in the case of associativity constraints, we only prove that the natural isomorphisms \eqref{comm-int} define an internal commutativity constraint on $(\tilde{\otimes},\tilde{m})$; the other two statements are left to the interested reader.
	
	Let us check that the natural isomorphisms \eqref{comm-int} satisfy conditions ($c$ITS-1) and ($c$ITS-2).
	We start from condition ($c$ITS-1): given $S \in \Scal$, we have to show that the diagram of functors $\Hbb(S) \times \Hbb(S) \rightarrow \Hbb(S)$
	\begin{equation*}
		\begin{tikzcd}
			A \tilde{\otimes} B \arrow{r}{\tilde{c}} \arrow{dr}{\id} & B \tilde{\otimes} A \arrow{d}{\tilde{c}} \\
			& A \tilde{\otimes} B
		\end{tikzcd}
	\end{equation*}
	is commutative. To this end, choose $(Y;s) \in \Emb_{\Ucal}(S)$. Then, by construction, it suffices to show that the diagram
	\begin{equation*}
		\begin{tikzcd}[font=\small]
			A \otimes_s B \arrow{dr}{\id} \arrow{r}{c_S^{(s)}} & B \otimes_s A \arrow{d}{c_S^{(s)}} \\
			& A \otimes_s B
		\end{tikzcd}
	\end{equation*}
	is commutative. Unwinding the definitions, we can write the latter more explicitly as the image under $s^*$ of the diagram
	\begin{equation*}
		\begin{tikzcd}[font=\small]
			s_* A \otimes s_* B \arrow{dr}{\id} \arrow{r}{c} & s_* B \otimes s_* A \arrow{d}{c} \\
			& s_* A \otimes s_* B
		\end{tikzcd}
	\end{equation*}
	which is commutative by axiom ($c$ITS-1) over $\Ucal$.
	
	We now check condition ($c$ITS-2): given a morphism $f: T \rightarrow S$ in $\Scal$, we have to show that the diagram of functors $\Hbb(S) \times \Hbb(S) \rightarrow \Hbb(T)$
	\begin{equation*}
		\begin{tikzcd}
			f^* A \tilde{\otimes} f^* B \arrow{r}{\tilde{c}} \arrow{d}{\tilde{m}} & f^* B \tilde{\otimes} f^* A \arrow{d}{\tilde{m}} \\
			f^* (A \tilde{\otimes} B) \arrow{r}{\tilde{c}} & f^* (B \tilde{\otimes} A)
		\end{tikzcd}
	\end{equation*}
	is commutative. To this end, choose an embedding $(Y;s) \in \Emb_{\Ucal}(S)$ and then a factorization $(X;t,p) \in \Fact_{\Scal}(s \circ f)$; note that the object $X$ automatically belongs to $\Ucal$, and so we obtain an embedding $(X;t) \in \Emb_{\Ucal}(T)$. Then, by construction, it suffices to show that the diagram
	\begin{equation*}
		\begin{tikzcd}[font=\small]
			f^* A \otimes_t f^* B \arrow{r}{c_T^{(t)}} \arrow{d}{m_f^{(s,t,p)}} & f^* B \otimes_t f^* A \arrow{d}{m_f^{(s,t,p)}} \\
			f^* (A \otimes_s B) \arrow{r}{c_S^{(s)}} & f^* (B \otimes_s A)
		\end{tikzcd}
	\end{equation*}
	is commutative. Unwinding the various definitions, we obtain the outer part of the diagram
	\begin{equation*}
		\begin{tikzcd}[font=\small]
			t^* (t_* f^* A \otimes t_* f^* B) \arrow{r}{c} & t^* (t_* f^* B \otimes t_* f^* A) \\
			t^* (p^* s_* A \otimes p^* s_* B) \arrow{r}{c} \isoarrow{u} \arrow{d}{m} & t^* (p^* s_* B \otimes p^* s_* A) \isoarrow{u} \arrow{d}{m} \\
			t^* p^* (s_* A \otimes s_* B) \arrow{r}{c} \arrow[equal]{d} & t^* p^* (s_* B \otimes s_* A) \arrow[equal]{d} \\
			f^* s^* (s_* A \otimes s_* B) \arrow{r}{c} & f^* s^* (s_* B \otimes s_* A)
		\end{tikzcd}
	\end{equation*}
	where all pieces are commutative by naturality. This proves the claim and concludes the proof.
\end{proof}

In the same spirit, we study \textit{internal unit constraints} $u$ on $(\tilde{\otimes},\tilde{m})$: as in \cite[Defn.~5.1]{Ter23Fib}, by this we mean the datum of
\begin{itemize}
	\item a section $\unit$ of the $\Scal$-fibered category $\Hbb$, called the \textit{unit section},
	\item for every $S \in \Scal$, two natural isomorphisms of functors $\Hbb(S) \rightarrow \Hbb(S)$
	\begin{equation*}
		u_r = u_{r,S}: A \tilde{\otimes} \unit_S \xrightarrow{\sim} A, \qquad \qquad u_l = u_{l,S}: \unit_S \tilde{\otimes} B \xrightarrow{\sim} B
	\end{equation*}
\end{itemize}
satisfying the following conditions:
\begin{enumerate}
	\item[($u$ITS-0)] For every $S \in \Scal$, the diagram in $\Hbb(S)$
	\begin{equation*}
		\begin{tikzcd}
			\unit_S \otimes \unit_S \arrow{dr}{u_r} \arrow[equal]{rr} && \unit_S \otimes \unit_S \arrow{dl}{u_l} \\
			& \unit_S
		\end{tikzcd}
	\end{equation*}
	is commutative.
	\item[($u$ITS-1)] For every morphism $f: T \rightarrow S$ in $\Scal$, the two diagrams of functors $\Hbb(S) \rightarrow \Hbb(T)$
	\begin{equation*}
		\begin{tikzcd}
			f^* A \tilde{\otimes} f^* \unit_S \arrow{r}{\unit^*} \arrow{d}{\tilde{m}} & f^* A \tilde{\otimes} \unit_T \arrow{d}{u_r} \\
			f^* (A \tilde{\otimes} \unit_S) \arrow{r}{u_r}  & f^* A
		\end{tikzcd}
		\qquad
		\begin{tikzcd}
			f^* \unit_S \tilde{\otimes} f^* B  \arrow{d}{\tilde{m}} \arrow{r}{\unit^*} & \unit_T \tilde{\otimes} f^* B \arrow{d}{u_l} \\
			f^* (\unit_S \tilde{\otimes} B)  \arrow{r}{u_l} & f^* B
		\end{tikzcd}
	\end{equation*}
	are commutative.
\end{enumerate}
Note that every internal unit constraint on $(\tilde{\otimes},\tilde{m})$ defines by restriction an analogous constraint on the internal tensor structure $(\otimes,m)$ over $\Ucal$. In the reverse direction, we have the following result:

\begin{prop}\label{prop:ext-unit}
	Let $u$ be an internal unit constraint on $(\otimes,m)$ with unit section $\unit$; let $\tilde{\unit}$ denote the section of $\Hbb$ over $\Scal$ extending $\unit$ via \Cref{prop:ext-sect}. Then there exists a unique internal unit constraint $\tilde{u}$ on $(\tilde{\otimes},\tilde{m})$ with unit section $\tilde{\unit}$ whose restriction to $\Ucal$ coincides with $u$.
\end{prop}

As in the previous two cases, the proof is based on an auxiliary construction:

\begin{lem}\label{lem:ext-unit}
	Fix $S \in \Scal$. For every embedding $(Y;s) \in \Emb_{\Ucal}(S)$, define a natural isomorphism of functors $\Hbb(S) \rightarrow \Hbb(S)$
		\begin{equation}\label{u_S^s}
			u_{r}^{(s)} = u_{r,S}^{(s)}: A \otimes_s \unit_S^{(s)} \xrightarrow{\sim} A, \qquad u_{l}^{(s)} = u_{l,S}^{(s)}: \unit_S^{(s)} \otimes_s B \xrightarrow{\sim} B
		\end{equation}
		by taking the composites
		\begin{equation*}
			\begin{tikzcd}
				A \otimes_s \unit_S^{(s)} \arrow[equal]{d} & A \\
				s^*(s_* A \otimes s_* s^* \unit_Y)  \\
				s^*(s_* A \otimes \unit_Y) \arrow{u}{\eta} \arrow{r}{u_r} & s^* s_* A \arrow{uu}{\epsilon}
			\end{tikzcd}
			\qquad
			\begin{tikzcd}
				\unit_S^{(s)} \otimes_s B \arrow[equal]{d} & B. \\
				s^*(s_* s^* \unit_Y \otimes s_* B) \\
				s^* (\unit_Y \otimes s_* B) \arrow{u}{\eta} \arrow{r}{u_l} & s^* s_* B \arrow{uu}{\epsilon}
			\end{tikzcd}
		\end{equation*}
		respectively.
		Then the induced isomorphisms of functors $\Hbb(S) \rightarrow \Hbb(S)$
		\begin{equation}\label{formula:ext-unit}
			\tilde{u}_{r} = \tilde{u}_{r,S}: A \tilde{\otimes} \tilde{\unit}_S = A \otimes_{s} \unit_S^{(s)} \xrightarrow{u_{r}^{(s)}} A, \qquad \tilde{u}_{l} = \tilde{u}_{l,S}: \tilde{\unit}_S \tilde{\otimes} B = \unit_S^{(s)} \otimes_{s} B \xrightarrow{u_{l}^{(s)}} B
		\end{equation}
		are independent of the choice of $(Y;s) \in \Emb_{\Ucal}(S)$.
\end{lem}
\begin{proof}
	The fact that the two unit arrows $\eta: s^*(s_* A \otimes \unit_Y) \rightarrow s^*(s_* A \otimes s_* s^* \unit_Y)$ and $\eta: s^* (\unit_Y \otimes s_* B) \rightarrow s^*(s_* s^* \unit_Y \otimes s_* B)$ are both invertible can be checked using the associated localization triangles; we omit the details. Thus the natural transformations \eqref{u_S^s} are indeed well-defined and invertible.
	
	We are left to showing the independence statement. We only prove the result for the left-most composite; the other case is analogous. Using the connectedness of the category $\Emb_{\Ucal}(S)$ via direct products (\Cref{lem:Emb}), we are reduced to proving the following claim: Given a morphism $p: (Y';s') \rightarrow (Y;s)$ in $\Emb_{\Ucal}(S)$, the diagram of functors $\Hbb(S) \rightarrow \Hbb(S)$
	\begin{equation*}
		\begin{tikzcd}
			A \otimes_s \unit_S^{(s)} \arrow{drr}{u_r^{(s)}} \arrow{d}{\alpha_p} \\
			A \otimes_s \unit_S^{(s')} \arrow{d}{\phi_p} && A \\
			A \otimes_{s'} \unit_S^{(s')} \arrow{urr}{u_r^{(s')}}
		\end{tikzcd}
	\end{equation*}
	is commutative. Unwinding the various definitions, we obtain the diagram
	\begin{equation*}
		\begin{tikzcd}[font=\small]
			s^* (a_* A \otimes s_* {s'}^* q^* \unit_Y) \arrow{d}{\unit^*} \arrow[equal]{r} & s^* (s_* A \otimes s_* s^* \unit_Y) & s^* (s_* A \otimes \unit_Y) \arrow{l}{\eta} \arrow{r}{u_r} & s^* s_* A \arrow{dd}{\epsilon} \\
			s^* (s_* A \otimes s_* {s'}^* \unit_{Y'}) \arrow[equal]{dd} \\
			&&& A \\
			{s'}^* q^* (q_* s'_* A \otimes q_* s'_* {s'}^* \unit_{Y'}) \\
			{s'}^* (q^* q_* s'_* A \otimes q^* q_* s'_* {s'}^* \unit_{Y'}) \arrow{u}{m} \arrow{r}{\epsilon} & {s'}^* (s'_* A \otimes s'_* {s'}^* \unit_{Y'}) & {s'}^* (s'_* A \otimes \unit_{Y'}) \arrow{l}{\eta} \arrow{r}{u_r} & {s'}^* s'_* A \arrow{uu}{\epsilon}
		\end{tikzcd}
	\end{equation*}
	that we decompose as
	\begin{equation*}
		\begin{tikzcd}[font=\tiny]
			\bullet \arrow{dd}{\unit^*} \arrow[equal]{rr} \arrow[equal]{dr} && \bullet \arrow[equal]{d} & \bullet \arrow{l}{\eta} \arrow{rr}{u_r} \arrow[equal]{d} && \bullet \arrow{ddd}{\epsilon} \arrow[equal]{dddl} \\
			& {s'}^* q^* (q_* {s'}_* A \otimes s_* {s'}^* q^* \unit_Y) \arrow[equal]{dr} \arrow{dd}{\unit^*} & {s'}^* q^* (q_* {s'}_* A \otimes s_* s^* \unit_Y) \arrow[equal]{d} & {s'}^* q^* (q_* {s'}_* A \otimes \unit_Y) \arrow{l}{\eta} \arrow{d}{\eta} \arrow[bend left]{ddr}{u_r} \\
			\bullet \arrow[equal]{dd} && {s'}^* q^* (q_* {s'}_* A \otimes q_* {s'}_* {s'}^* q^* \unit_Y) \arrow{d}{\unit^*} & {s'}^* q^* (q_* {s'}_* A \otimes q_* q^* \unit_Y) \arrow{d}{\bar{m}} \arrow{ddl}{\unit^*} \arrow{l}{\eta} \\
			& {s'}^* q^* (q_* {s'}_* A \otimes s_* {s'}^* \unit_{Y'}) \arrow[equal]{dl} \arrow[equal]{r} & {s'}^* q^* (q_* {s'}_* A \otimes q_* {s'}_* {s'}^* \unit_{Y'}) & {s'}^* q^* q_* ({s'}_* A \otimes q^* \unit_{Y'}) \arrow{d}{\unit^*} & {s'}^* q^* q_* {s'}_* A \arrow{dddr}{\epsilon} & \bullet \\
			\bullet && {s'}^* q^* (q_* {s'}_* A \otimes q_* \unit_{Y'}) \arrow{u}{\eta} \arrow{ll}{\eta} \arrow{r}{\bar{m}} & {s'}^* q^* q_* ({s'}_* A \otimes \unit_{Y'}) \arrow[bend right]{ur}{u_r} \arrow{dd}{\epsilon} \\
			&& {s'}^* (q^* q_* {s'}_* A \otimes q^* q_* \unit_{Y'}) \arrow{dll}{\eta} \arrow{dr}{\epsilon} \arrow{u}{m} \\
			\bullet \arrow{uu}{m} \arrow{rr}{\epsilon} && \bullet & \bullet \arrow{l}{\eta} \arrow{rr}{u_r} && \bullet \arrow{uuu}{\epsilon}
		\end{tikzcd}
	\end{equation*}
	Here, the central-right five-term piece is commutative as a consequence of axiom ($u$ITS-2) over $\Ucal$, while all the remaining pieces are commutative by naturality and by construction. This proves the claim.
\end{proof}

\begin{proof}[Proof of \Cref{prop:ext-unit}]
	Similarly to the previous two cases, we only prove that the natural isomorphisms \eqref{formula:ext-unit} define an internal unit constraint on $(\tilde{\otimes},\tilde{u})$; the other two statements are left to the interested reader.
	
	Let us check that the natural isomorphisms \eqref{formula:ext-unit} satisfy conditions ($u$ITS-1) and ($u$ITS-2). We start from condition ($u$ITS-1): given $S \in \Scal$, we have to show that the diagram in $\Hbb(S)$
	\begin{equation*}
		\begin{tikzcd}
			\tilde{\unit}_S \otimes \tilde{\unit}_S \arrow{rr}{\id} \arrow{dr}{\tilde{u}_r} && \tilde{\unit}_S \otimes \tilde{\unit}_S \arrow{dl}{\tilde{u}_l} \\
			& \tilde{\unit}_S
		\end{tikzcd}
	\end{equation*}
	is commutative. Unwinding the various definitions, this amount to showing that, for every $(Y;s) \in \Emb(S)$, the diagram
	\begin{equation*}
		\begin{tikzcd}[font=\small]
			\unit_S^{(s)} \otimes_s \unit_S^{(s)} \arrow[equal]{rr} \arrow{dr}{u^{(s)}_{r,S}} && \unit_S^{(s)} \otimes_s \unit_S^{(s)} \arrow{dl}{u^{(s)}_{l,S}} \\
			& \unit_S^{(s)}
		\end{tikzcd}
	\end{equation*}
	is commutative. But the latter coincides with the outer part of the diagram
	\begin{equation*}
		\begin{tikzcd}[font=\small]
			s^*(s_* s^* \unit_Y \otimes s_* s^* \unit_Y) \arrow[equal]{rrrr} &&&& s^*(s_* s^* \unit_Y \otimes s_* s^* \unit_Y) \\
			s^*(s_* s^* \unit_Y \otimes \unit_Y) \arrow{u}{\eta} \arrow{d}{u_r} & s^*(\unit_Y \otimes \unit_Y) \arrow{l}{\eta} \arrow[equal]{rr} \arrow{dr}{u_r} && s^*(\unit_Y \otimes \unit_Y) \arrow{r}{\eta} \arrow{dl}{u_l} & s^*(\unit_Y \otimes s_* s^* \unit_Y) \arrow{u}{\eta} \arrow{d}{u_l} \\
			s^* s_* s^* \unit_Y \arrow{rr}{\epsilon} && s^* \unit_Y && s^* s_* s^* \unit_Y \arrow{ll}{\epsilon}
		\end{tikzcd}
	\end{equation*}
	where the central triangular piece is commutative by axiom ($u$ITS-1) over $\Ucal$, while the other inner pieces are commutative by naturality.
	
	We now check condition ($u$ITS-2): given a morphism $f: T \rightarrow S$ in $\Scal$, we have to show that the diagram of functors $\Hbb(S) \rightarrow \Hbb(T)$
	\begin{equation*}
		\begin{tikzcd}
			f^* A \tilde{\otimes} f^* \tilde{\unit}_S \arrow{r}{\tilde{m}} \arrow{d}{\tilde{\unit}^*} & f^*(A \tilde{\otimes} \tilde{\unit}_S) \arrow{d}{\tilde{u}_r} \\
			f^* A \tilde{\otimes} \tilde{\unit}_T \arrow{r}{\tilde{u}_r} & f^* A
		\end{tikzcd}
		\qquad
		\begin{tikzcd}
			f^* \tilde{\unit}_S \tilde{\otimes} f^* B \arrow{r}{\tilde{m}} \arrow{d}{\tilde{\unit}^*} & f^*(\tilde{\unit}_S \tilde{\otimes} B) \arrow{d}{\tilde{u}_l} \\
			\tilde{\unit}_T \tilde{\otimes} f^* B \arrow{r}{\tilde{u}_l} & f^* B
		\end{tikzcd}
	\end{equation*}
	are commutative. We only check the commutativity of the left-most diagram; the other case is analogous. To this end, choose an embedding $(Y;s) \in \Emb_{\Ucal}(S)$ and then a factorization $(X;t,p) \in \Fact_{\Scal}(s \circ f)$; note that the object $X$ automatically belongs to $\Ucal$, and so we obtain an embedding $(X;t) \in \Emb_{\Ucal}(T)$. Then, by construction, it suffices to show that the diagram of functors $\Hbb(S) \rightarrow \Hbb(T)$
	\begin{equation*}
		\begin{tikzcd}[font=\small]
			f^* A \otimes_t f^* \unit_S^{(s)} \arrow{rr}{m_f^{(s,t,p)}} \arrow{d}{\unit_f^{*,(s,t,p)}} && f^*(A \otimes_s \unit_S^{(s)}) \arrow{d}{u_{r,S}^{(s)}} \\
			f^* A \otimes_t \unit_T^{(t)} \arrow{rr}{u_{r,T}^{(t)}} && f^* A
		\end{tikzcd}
	\end{equation*}
	is commutative. Unwinding the definitions, we obtain the outer part of the diagram
	
	\begin{equation*}
		\begin{tikzcd}[font=\tiny]
			t^* (t_* f^* A \otimes t_* f^* s^* \unit_Y) \arrow[equal]{dd} & t^* (p^* s_* A \otimes p^* s_* s^* \unit_Y) \arrow{l}{\sim} \arrow{rr}{m}  && t^* p^*(s_* A \otimes s_* s^* \unit_Y)  \arrow[equal]{r} & f^* s^* (s_* A \otimes s_* s^* \unit_Y) \\
			& t^* (p^* s_* A \otimes p^* \unit_Y) \arrow{u}{\eta} \arrow{rr}{m} \isoarrow{d} \arrow{ddr}{\unit^*} && t^* p^* (s_* A \otimes \unit_Y) \arrow[equal]{r} \arrow{dd}{u_r} \arrow{u}{\eta} & f^* s^* (s_* A \otimes \unit_Y) \arrow{dd}{u_r} \arrow{u}{\eta} \\
			t^* (t_* f^* A \otimes t_* t^* p^* \unit_Y) \arrow{dd}{\unit^*} & t^* (t_* f^* A \otimes p^* \unit_Y) \arrow{l}{\eta} \arrow{ddr}{\unit^*} \\
			&& t^* (p^* s_* A \otimes \unit_X) \isoarrow{d} \arrow{r}{u_r} & t^* p^* s_* A \isoarrow{d} \arrow[equal]{r} & f^* s^* s_* A \arrow{d}{\epsilon} \\
			t^* (t_* f^* A \otimes t_* t^* \unit_X) && t^* (t_* f^* A \otimes \unit_X) \arrow{ll}{\eta} \arrow{r}{u_r} & t^* t_* f^* A \arrow{r}{\epsilon} & f^* A
		\end{tikzcd}
	\end{equation*}
	
	where the central inner piece is commutative by axiom ($u$ITS-2) over $\Ucal$, while the other inner pieces are commutative by naturality and by construction.
\end{proof}

\subsection{Extended compatibility conditions}

Lastly, we check that the extension procedure for associativity, commutativity and unit constraints just described preserves the natural mutual compatibility condition between them. Given internal associativity, commutativity and unit constraints $a$, $c$ and $u$ on $(\tilde{\otimes},\tilde{m})$, we consider the compatibility conditions described in \cite[\S 6]{Ter23Fib}:
\begin{enumerate}
	\item The constraints $a$ and $c$ are compatible if they satisfy the following additional condition:
	\begin{enumerate}
		\item[($ac$ITS)] For every $S \in \Scal$, the diagram of functors $\Hbb(S) \times \Hbb(S) \times \Hbb(S) \rightarrow \Hbb(S)$
		\begin{equation*}
			\begin{tikzcd}
				(A \otimes B) \otimes C \arrow{r}{c} \arrow{d}{a} & (B \otimes A) \otimes C \arrow{r}{a} & B \otimes (A \otimes C) \arrow{d}{c} \\
				A \otimes (B \otimes C) \arrow{r}{c} & (B \otimes C) \otimes A \arrow{r}{a} & B \otimes (C \otimes A)
			\end{tikzcd}
		\end{equation*}
		is commutative.
	\end{enumerate}
    \item The constraints $a$ and $u$ are \textit{compatible} if they satisfy the following condition:
    \begin{enumerate}
    	\item[($au$ITS)] For every $S \in \Scal$, the three diagrams of functors $\Hbb(S) \times \Hbb(S) \rightarrow \Hbb(S)$
    	\begin{equation*}
    		\begin{tikzcd}
    			(\unit_S \otimes B) \otimes C \arrow{rr}{a} \arrow{dr}{u_l} && \unit_S \otimes (B \otimes C) \arrow{dl}{u_l} \\
    			& B \otimes C
    		\end{tikzcd}
    	    \qquad
    	    \begin{tikzcd}
    	    	(A \otimes \unit_S) \otimes C \arrow{rr}{a} \arrow{dr}{u_r} && A \otimes (\unit_S \otimes C) \arrow{dl}{u_l} \\
    	    	& A \otimes C
    	    \end{tikzcd}
    	\end{equation*}
        and
    	\begin{equation*}
    		\begin{tikzcd}
    			(A \otimes B) \otimes \unit_S \arrow{rr}{a} \arrow{dr}{u_r} && A \otimes (B \otimes \unit_S) \arrow{dl}{u_r} \\
    			& A \otimes B
    		\end{tikzcd}
    	\end{equation*}
    	are commutative.
    \end{enumerate}
    \item The constraints $c$ and $u$ are compatible if they satisfy the following condition:
    \begin{enumerate}
    	\item[($cu$ITS)] For every $S \in \Scal$, the diagram of functors $\Hbb(S) \rightarrow \Hbb(S)$
    	\begin{equation*}
    		\begin{tikzcd}
    			\unit_S \otimes B \arrow{rr}{c} \arrow{dr}{u_l} && B \otimes \unit_S \arrow{dl}{u_r} \\
    			& B
    		\end{tikzcd}
    	\end{equation*}
    	is commutative.
    \end{enumerate}
\end{enumerate}
As usual, note that each compatibility condition for $(\tilde{\otimes},\tilde{m})$ determines an analogous compatibility condition for $(\otimes,m)$ over $\Ucal$. Conversely:

\begin{lem}\label{lem:ac-comp_int_ext}
	Let $(\otimes,m)$ be a localic internal tensor structure on $\Hbb$ over $\Ucal$, and let $(\tilde{\otimes},\tilde{m})$ denote the internal tensor structure over $\Scal$  obtained from it via \Cref{thm:ext-otimes}. Then the following statements hold:
	\begin{enumerate}
		\item Suppose that we are given compatible internal associativity and commutativity constraints $a$ and $c$ on $(\otimes,m)$. Then the extended constraints $\tilde{a}$ and $\tilde{c}$ on $(\tilde{\otimes},\tilde{m})$ are compatible as well.
		\item Suppose that we are given compatible internal associativity and unit constraints $a$ and $u$ on $(\otimes,m)$. Then the extended constraints $\tilde{a}$ and $\tilde{u}$ on $(\tilde{\otimes},\tilde{m})$ are compatible as well.
		\item Suppose that we are given compatible internal commutativity and unit constraint $c$ and $u$ on $(\otimes,m)$. then the extended constraints $\tilde{c}$ and $\tilde{u}$ on $(\tilde{\otimes},\tilde{m})$ are compatible as well.
	\end{enumerate}
\end{lem}
\begin{proof}
	\begin{enumerate}
		\item Let us check that the constraints $\tilde{a}$ and $\tilde{c}$ in the statement satisfy condition ($ac$ITS): given $S \in \Scal$, we have to show that the diagram of functors $\Hbb(S) \times \Hbb(S) \times \Hbb(S) \rightarrow \Hbb(S)$
		\begin{equation*}
			\begin{tikzcd}
				(A \tilde{\otimes} B) \tilde{\otimes} C \arrow{r}{\tilde{c}} \arrow{d}{\tilde{a}} & (B \tilde{\otimes} A) \tilde{\otimes} C \arrow{r}{\tilde{a}} & B \tilde{\otimes} (A \tilde{\otimes} C) \arrow{d}{\tilde{c}} \\
				A \tilde{\otimes} (B \tilde{\otimes} C) \arrow{r}{\tilde{a}} & (B \tilde{\otimes} C) \tilde{\otimes} A \arrow{r}{\tilde{a}} & B \tilde{\otimes} (C \tilde{\otimes} A)
			\end{tikzcd}
		\end{equation*}
		is commutative. To this end, choose an embedding $(Y;s) \in \Emb_{\Ucal}(S)$. Then, by construction, it suffices to show that the diagram
		\begin{equation*}
			\begin{tikzcd}[font=\small]
				(A \otimes_s B) \otimes_s C \arrow{r}{c_S^{(s)}} \arrow{d}{a_S^{(s)}} & (B \otimes_s A) \otimes_S C \arrow{r}{a_S^{(s)}} & B \otimes_s (A \otimes_s C) \arrow{d}{c_S^{(s)}} \\
				A \otimes_s (B \otimes_s C) \arrow{r}{c_S^{(s)}} & (B \otimes_s C) \otimes_s A \arrow{r}{a_S^{(s)}} & B \otimes_s (C \otimes_s A)
			\end{tikzcd}
		\end{equation*}
		is commutative. Unwinding the definitions, we can write the latter more explicitly as the outer part of the diagram
		\begin{equation*}
			\begin{tikzcd}[font=\tiny]
				s^* (s_* s^* (s_* A \otimes s_* B) \otimes s_* C) \arrow{r}{c} & s^* (s_* s^* (s_* B \otimes s_* A) \otimes s_* C)  & s^* ((s_* B \otimes s_* A) \otimes s_* C) \arrow{l}{\eta} \arrow{r}{a} & s^* (s_* B \otimes (s_* A \otimes s_* C)) \arrow{d}{\eta} \arrow[bend right = 75]{ddd}{c} \\
				s^* ((s_* A \otimes s_* B) \otimes s_* C) \arrow{u}{\eta} \arrow{urr}{c} \arrow{d}{a} &&& s^* (s_* B \otimes s_* s^* (s_* A \otimes s_* C)) \arrow{d}{c} \\
				s^* (s_* A \otimes (s_* B \otimes s_* C)) \arrow{d}{\eta} \arrow{drr}{c} &&& s^* (s_* B \otimes s_* s^* (s_* C \otimes s_* A)) \\
				s^* (s_* A \otimes s_* s^* (s_* B \otimes s_* C)) \arrow{r}{c} & s^* (s_* s^* (s_* B \otimes s_* C) \otimes s_* A) & s^* ((s_* B \otimes s_* C) \otimes s_* A) \arrow{l}{\eta} \arrow{r}{a} & s^* (s_* B \otimes (s_* C \otimes s_* A)) \arrow{u}{\eta}
			\end{tikzcd}
		\end{equation*}
		Here, the central piece is commutative by axiom ($ac$ITS) over $\Ucal$ while the remaining pieces are commutative by naturality.
		\item Let us check that the definitions satisfy condition ($au$ITS): given $S \in \Scal$, we have to show that the three diagrams of functors $\Hbb(S) \times \Hbb(S) \rightarrow \Hbb(S)$
		\begin{equation*}
			\begin{tikzcd}
				(\tilde{\unit}_S \tilde{\otimes} B) \tilde{\otimes} C \arrow{dr}{\tilde{u}_l} \arrow{rr}{\tilde{a}} && \tilde{\unit}_S \tilde{\otimes} (B \tilde{\otimes} C) \arrow{dl}{\tilde{u}_l} \\
				& B \tilde{\otimes} C,
			\end{tikzcd}
		    \qquad 
		    \begin{tikzcd}
		    	(A \tilde{\otimes} \tilde{\unit}_S) \tilde{\otimes} C \arrow{dr}{\tilde{u}_r} \arrow{rr}{\tilde{a}} && A \tilde{\otimes} (\tilde{\unit}_S \tilde{\otimes} C) \arrow{dl}{\tilde{u}_l} \\
		    	& A \tilde{\otimes} C
		    \end{tikzcd}
		\end{equation*}
		and
		\begin{equation*}
			\begin{tikzcd}
				(A \tilde{\otimes} B) \tilde{\otimes} \tilde{\unit}_S \arrow{dr}{\tilde{u}_r} \arrow{rr}{\tilde{a}} && A \tilde{\otimes} (B \tilde{\otimes} \tilde{\unit}_S) \arrow{dl}{\tilde{u}_r} \\
				& A \tilde{\otimes} B
			\end{tikzcd}	
		\end{equation*}
		are commutative. We only treat the case of the first diagram; the other two diagrams can be studied in a similar way. As usual, choose an embedding $(Y;s) \in \Emb_{\Ucal}(S)$. Then, by construction, it suffices to show that the diagram of functors $\Hbb(S) \times \Hbb(S) \rightarrow \Hbb(S)$
		\begin{equation*}
			\begin{tikzcd}[font=\small]
				(\unit_S^{(s)} \otimes_s B) \otimes_s C \arrow{dr}{u_{l,S}^{(s)}} \arrow{rr}{a_S^{(s)}} && \unit_S^{(s)} \otimes_s(B \otimes_s C) \arrow{dl}{u_{l,S}^{(s)}} \\
				& B \otimes_s C
			\end{tikzcd}
		\end{equation*}
		is commutative. Unwinding the various definitions, we obtain the more explicit diagram
		\begin{equation*}
			\begin{tikzcd}[font=\small]
				s^* ((s_* s^* \unit_Y \otimes s_* B) \otimes s_* C) \arrow{rr}{a} \arrow{d}{\eta} && s^* (s_* s^* \unit_Y \otimes (s_* B \otimes s_* C)) \arrow{d}{\eta} \\
				s^* (s_* s^* (s_* s^* \unit_Y \otimes s_* B) \otimes s_* C) && s^* (s_* s^* \unit_Y \otimes s_* s^* (s_* B \otimes s_* C)) \\
				s^* (s_* s^* (\unit_Y \otimes s_* B) \otimes s_* C) \arrow{u}{\eta} \arrow{d}{u_l} && s^* (\unit_Y \otimes s_* s^* (s_* B \otimes s_* C)) \arrow{u}{\eta} \arrow{d}{u_l} \\
				s^* (s_* s^* s_* B \otimes s_* C) \arrow{dr}{\epsilon} && s^* s_* s^* (s_* B \otimes s^* C) \arrow{dl}{\epsilon} \\
				& s^* (s_* B \otimes s_* C)
			\end{tikzcd}
		\end{equation*}
		that we decompose as
		\begin{equation*}
			\begin{tikzcd}
				\bullet \arrow{rrrr}{a} \arrow{d}{\eta} &&&& \bullet \arrow{d}{\eta} \\
				\bullet &&&& \bullet \\
				\bullet \arrow{u}{\eta} \arrow{d}{u_l} & s^* ((\unit_Y \otimes s_* B) \otimes s_* C) \arrow{l}{\eta} \arrow{rr}{a} \arrow{uul}{\eta} \arrow{ddr}{u_l} && s^* (\unit_Y \otimes (s_* B \otimes s_* C)) \arrow{r}{\eta} \arrow{uur}{\eta} \arrow{ddl}{u_l} & \bullet \arrow{u}{\eta} \arrow{d}{u_l} \\
				\bullet \arrow{drr}{\epsilon} &&&& \bullet \arrow{dll}{\epsilon} \\
				&& \bullet
			\end{tikzcd}
		\end{equation*}
		Here, the central triangle is commutative by axiom ($au$ITS) over $\Ucal$, while all the remaining pieces are commutative by naturality and by construction. This proves the claim. 
		\item Let us check that the definitions satisfy condition ($cu$ITS): given $S \in \Scal$, we have to show that the diagram of functors $\Hbb(S) \rightarrow \Hbb(S)$
		\begin{equation*}
			\begin{tikzcd}
				\tilde{\unit}_S \tilde{\otimes} B \arrow{rr}{\tilde{c}} \arrow{dr}{\tilde{u}_l} && B \tilde{\otimes} \tilde{\unit}_S \arrow{dl}{\tilde{u}_r} \\
				& B
			\end{tikzcd}
		\end{equation*}
		is commutative. To this end, choose an embedding $(Y;s) \in \Emb_{\Ucal}(S)$. Then, by construction, it suffices to show that the diagram of functors
		\begin{equation*}
			\begin{tikzcd}[font=\small]
				\unit_S^{(s)} \otimes_s B \arrow{dr}{u_{l,S}^{(s)}} \arrow{rr}{c_S^{(s)}} && B \otimes_s \unit_S^{(s)} \arrow{dl}{u_{r,S}^{(s)}} \\
				& B
			\end{tikzcd}
		\end{equation*}
		is commutative. Unwinding the various definitions, we can write it more explicitly as the outer part of the diagram
		\begin{equation*}
			\begin{tikzcd}[font=\small]
				s^* (s_* s^* \unit_Y \otimes s_* B) \arrow{rr}{c} && s^* (s_* B \otimes s_* s^* \unit_Y) \\
				s^* (\unit_Y \otimes s_* B) \arrow{rr}{c} \arrow{u}{\eta} \arrow{dr}{u_l} && s^* (s_* B \otimes \unit_Y) \arrow{u}{\eta} \arrow{dl}{u_l} \\
				&  s^* s_* B \arrow{d}{\epsilon} \\
				& B
			\end{tikzcd}
		\end{equation*}
		where the upper square is commutative by naturality while the lower triangle is commutative by axiom ($cu$ITS) over $\Ucal$. This proves the claim.
	\end{enumerate}
\end{proof}

\section{Extending monoidal morphisms}\label{sect:ext-rho}

Throughout this section, we fix two localic $\Scal$-fibered categories $\Hbb_1$ and $\Hbb_2$ endowed with localic internal tensor structures $(\otimes_1,m_1)$ and $(\otimes_2,m_2)$, respectively; we also fix a localic morphism $R: \Hbb_1 \rightarrow \Hbb_2$. We are interested in applying our extension method to the monoidality properties of $R$ with respect to the two given monoidal structures on $\Hbb_1$ and $\Hbb_2$; as in the previous section, we formulate our results in the language of internal tensor structures employed in \cite{Ter23Fib}. 

The first part of this section is devoted to proving the main extension result. In the second part, we show that the possible associativity, symmetry and unitarity properties of $R$ are all preserved along this extension.

As the reader can guess, one could also start from localic internal tensor structures on the underlying $\Ucal$-fibered categories of $\Hbb_1$ and $\Hbb_2$ and from a morphism of $\Ucal$-fibered category, and then extend all these structures at the same time together with the monoidality properties of the morphism. We prefer to work directly with the extended structures in order to simplify the notation; however, in the proof of the results below we often need to use the constructions of \Cref{sect:ext-mor} and \Cref{sect_ext-boxtimes} implicitly.

\subsection{Extending internal tensor structures on morphisms}

We are interested in studying \textit{internal tensor structures} $\rho$ on $R$ (with respect to $(\otimes_1,m_1)$ and $(\otimes_2,m_2)$): as in \cite[Defn.~8.1]{Ter23Fib}, by this we mean the datum of
\begin{itemize}
	\item for every $S \in \Scal$, a natural isomorphism of functors $\Hbb_1(S) \times \Hbb_1(S) \rightarrow \Hbb_2(S)$
	\begin{equation*}
		\rho = \rho_S: R_S(A) \otimes_2 R_S(B) \xrightarrow{\sim} R_S(A \otimes_1 B),
	\end{equation*}
	called the \textit{internal $R$-monoidality isomorphism} at $S$
\end{itemize}
satisfying the following condition:
\begin{enumerate}
	\item[(mor-ITS)] For every morphism $f: T \rightarrow S$ in $\Scal$, the natural diagram of functors $\Hbb_1(S) \times \Hbb_1(S) \rightarrow \Hbb_2(T)$
	\begin{equation*}
		\begin{tikzcd}
			f^* R_S(A) \otimes_2 f^* R_S(B) \arrow{r}{m_2} \arrow{d}{\theta} & f^* (R_S(A) \otimes_2 R_S(B)) \arrow{r}{\rho} & f^* R_S(A \otimes_1 B) \arrow{d}{\theta} \\
			R_T(f^* A) \otimes_2 R_T(f^* B) \arrow{r}{\rho} & R_T(f^* A \otimes_1 f^* B) \arrow{r}{m_1} & R_T(f^*(A \otimes_1 B))
		\end{tikzcd}
	\end{equation*}
	is commutative.
\end{enumerate}
Since $\Hbb_1$ and $\Hbb_2$ are triangulated $\Scal$-fibered categories and $R$ is a triangulated morphism, it is natural to only consider those internal tensor structures $\rho$ for which each natural isomorphism $\rho_S$ respects the triangulated structures: in this way, one obtains the notion of a \textit{triangulated internal tensor structure} on $R$ (see \cite[\S~9]{Ter23Fib}). No additional condition is required in the localic setting.

Note that any (triangulated) internal tensor structure on $R$ defines by restriction a (triangulated) internal tensor structure on the underlying morphism of $\Ucal$-fibered categories. In the reverse direction, we have the following result:

\begin{thm}\label{thm_ext-rho}
	Let $\rho$ be a triangulated internal tensor structure on the morphism of $\Ucal$-fibered categories $R: \Hbb_1 \rightarrow \Hbb_2$. Then there exists a unique triangulated internal tensor structure $\tilde{\rho}$ on the morphism of $\Scal$-fibered categories $R$ whose restriction to $\Ucal$ coincides with $\rho$.
\end{thm}

As usual, we isolate the main auxiliary construction needed for the proof in the following result:

\begin{lem}\label{lem_rho_sm}
	Fix $S \in \Scal$. For every embedding $(Y;s) \in \Emb_{\Ucal}(S)$, define a natural isomorphism of functors $\Hbb_1(S) \times \Hbb_1(S) \rightarrow \Hbb_2(S)$
	\begin{equation*}
		\rho_S^{(s)} : R_S^{(s)}(A) \otimes_{2,s} R_S^{(s)}(B) \xrightarrow{\sim} R_S^{(s)}(A \otimes_{1,s} B)
	\end{equation*}
	by taking the composite
	\begin{equation*}
		\begin{tikzcd}[font=\small]
			R_S^{(s)}(A) \otimes_{2,s} R_S^{(s)}(B) \arrow[equal]{d} &&& R_S^{(s)}(A \otimes_{1,s} B). \arrow[equal]{d} \\
			s^* R_Y(s_* A) \otimes_{2,s} s^* R_Y(s_* B) \arrow{r}{m_2} & s^*(R_Y(s_* A) \otimes_{2,s} R_Y(s_* B)) \arrow{r}{\rho} & s^* R_Y (s_* A \otimes_{1,s} s_* B) \arrow{r}{\bar{m}_1} & s^* R_Y(s_* (A \otimes_{1,s} B)) 
		\end{tikzcd}
	\end{equation*}
	Then the induced natural isomorphism of functors $\Hbb_1(S) \times \Hbb_1(S) \rightarrow \Hbb_2(S)$
	\begin{equation}\label{tilde:rho}
		\tilde{\rho} = \tilde{\rho}_S: R_S(A) \otimes_{2,s} R_S(B) = R_S^{(s)}(A) \otimes_{2,s} R_S^{(s)}(B) \xrightarrow{\rho_S^{(s)}} R_S^{(s)}(A \otimes_{1,s} B) = R_S(A \otimes_{1,s} B)
	\end{equation}
	is independent of the choice of $(Y;s) \in \Emb_{\Ucal}(S)$.
\end{lem}
\begin{proof}
	In view of the connectedness of the category $\Emb_{\Ucal}(S)$ via direct products (\Cref{lem:Emb}), it suffices to prove the following claim: Given a morphism $p: (Y';s') \rightarrow (Y;s)$ in $\Emb_{\Ucal}(S)$, the diagram of functors $\Hbb_1(S) \times \Hbb_1(S) \rightarrow \Hbb_2(S)$
	\begin{equation*}
		\begin{tikzcd}
			R_S^{(s)}(A) \otimes_{2,s} R_S^{(s)}(B) \arrow{r}{\rho_S^{(s)}} \arrow{d}{\phi_p} & R_S^{(s)}(A \otimes_{1,s} B) \arrow{d}{\phi_p} \\
			R_S^{(s')}(A) \otimes_{2,s'} R_S^{(s')}(B) \arrow{r}{\rho_S^{(s')}} & R_S^{(s')}(A \otimes_{1,s'} B) 
		\end{tikzcd}
	\end{equation*}
	is commutative. Unwinding the various definition, we obtain the more explicit diagram
	\begin{equation*}
		\begin{tikzcd}[font=\small]
			s^*(R_Y(s_* A) \otimes_2 R_Y(s_* B)) \arrow{r}{\rho} & s^* R_Y(s_* A \otimes_1 s_* B) \arrow{d}{\bar{m}_1} \\
			s^* R_Y(s_* A) \otimes_2 s^* R_Y(s_* B) \arrow[equal]{d} \arrow{u}{m_2} & s^* R_Y(s_*(A \otimes_1 B)) \arrow[equal]{d} \\
			{s'}^* p^* R_Y(p_* s'_* A) \otimes_2 {s'}^* p^* R_Y(p_* s'_* A) \arrow{d}{\theta} & {s'}^* p^* R_Y(p_* s'_*(A \otimes_1 B)) \arrow{d}{\theta} \\
			{s'}^* R_{Y'}(p^* p_* s'_* A) \otimes_2 {s'}^* R_{Y'}(p^* p_* s'_* B) \arrow{d}{\epsilon}  & {s'}^* R_{Y'}(p^* p_* s'_* (A \otimes_1 B)) \arrow{d}{\epsilon} \\
			{s'}^* R_{Y'}(s'_* A) \otimes_2 {s'}^* R_{Y'}(s'_* B) \arrow{d}{m_2} & {s'}^* R_{Y'}(s'_*(A \otimes_1 B)) \\
			{s'}^*(R_{Y'}(s'_* A) \otimes_2 R_{Y'}(s'_* B)) \arrow{r}{\rho} & {s'}^* R_{Y'}(s'_* A \otimes_1 s'_* B). \arrow{u}{\bar{m}_1}
		\end{tikzcd}
	\end{equation*}
    that we decompose as
    \begin{equation*}
    	\begin{tikzcd}[font=\tiny]
    		\bullet \arrow{rrrr}{\rho} \arrow[equal]{dr} &&&& \bullet \arrow{d}{\bar{m}_1} \arrow[equal]{dll} \\
    		\bullet \arrow[equal]{d} \arrow{u}{m_2} & {s'}^* p^* (R_Y(p_* {s'}_* A) \otimes_2 R_Y(p_* {s'}_* B)) \arrow{r}{\rho} & {s'}^* p^* R_Y(p_* {s'}_* A \otimes_1 p_* {s'}_* B) \arrow{d}{\theta} \arrow{dr}{\bar{m}_1} && \bullet \arrow[equal]{d} \\
    		\bullet \arrow{d}{\theta} \arrow{r}{m_2} & {s'}^* (p^* R_Y(p_* {s'}_* A) \otimes_2 p^* R_Y(p_* {s'}_* B)) \arrow{u}{m_2} \arrow{dd}{\theta} & {s'}^* R_{Y'}(p^* (p_* {s'}_* A \otimes_1 p_* {s'}_* B)) \arrow{dr}{\bar{m}_1} & {s'}^* p^* R_Y(p_* ({s'}_* A \otimes {s'}_* B)) \arrow{r}{\bar{m}_1} \arrow{d}{\theta} & \bullet \arrow{d}{\theta} \\
    		\bullet \arrow{d}{\epsilon} \arrow{dr}{m_2} &&& {s'}^* R_{Y'}(p^* p_* ({s'}_* A \otimes_1 {s'}_* B)) \arrow{ddr}{\epsilon} \arrow{r}{\bar{m}_1} & \bullet \arrow{d}{\epsilon} \\
    		\bullet \arrow{d}{m_2} & {s'}^* (R_{Y'}(p^* p_* {s'}_* A) \otimes_2 R_{Y'}(p^* p_* {s'}_* B)) \arrow{r}{\rho} \arrow{dl}{\epsilon} & {s'}^* R_{Y'}(p^* p_* {s'}_* A \otimes_1 p^* p_* {s'}_* B) \arrow{uu}{m_1} \arrow{drr}{\epsilon} && \bullet \\
    		\bullet \arrow{rrrr}{\rho} &&&& \bullet \arrow{u}{\bar{m}_1}
    	\end{tikzcd}
    \end{equation*}
    Here, the central rectangle is commutative by axiom ($m$ITS) over $\Ucal$, the upper-right five-term piece is commutative by axiom ($\Scal^{op}$-fib), and all the remaining pieces are commutative by naturality and by construction. This proves the claim.
\end{proof}

\begin{proof}[Proof of \Cref{thm_ext-rho}]
	We only check that the natural isomorphisms \eqref{tilde:rho} define an internal tensor structure on $R$; the rest is left to the interested reader.
	
	Let us check that the natural isomorphisms \eqref{tilde:rho} satisfy condition (mor-ITS): given a morphism $f: T \rightarrow S$ in $\Scal$, we have to show that the diagram of functors $\Hbb_1(S) \times \Hbb_1(S) \rightarrow \Hbb_2(T)$
	\begin{equation*}
		\begin{tikzcd}
			f^* R_S(A) \otimes_2 f^* R_S(B) \arrow{r}{m_2} \arrow{d}{\theta} & f^*(R_S(A) \otimes_2 R_S(B))  \arrow{d}{\tilde{\rho}} \\
			f^* R_S(A \otimes_1 B)  \arrow{d}{\tilde{\rho}} & R_T(f^* A) \otimes_2 R_T(f^* B) \arrow{d}{\theta} \\
			R_T(f^* A \otimes_1 f^* B) \arrow{r}{m_1} & R_T(f^* (A \otimes_1 B))
		\end{tikzcd}
	\end{equation*}
	commutes. To this end, choose an embedding $(Y;s) \in \Emb_{\Ucal}(S)$ and then a factorization $(X;t,p) \in \Fact_{\Scal}(s \circ f)$; note that the object $X$ automatically belongs to $\Ucal$, and so we obtain an embedding $(X;t) \in \Emb_{\Ucal}(T)$. Then, by construction, it suffices to show that the diagram
	\begin{equation*}
		\begin{tikzcd}[font=\small]
			f^* R_S^{(s)}(A) \otimes_2 f^* R_S^{(s)}(B)  \arrow{rr}{m_{2,f}^{(s,t,p)}} \arrow{d}{\theta_f^{(s,t,p)}} && f^*(R_S^{(s)}(A) \otimes_2 R_S^{(s)}(B)) \arrow{d}{\rho_S^{(s)}} \\
			R_T^{(t)}(f^* A) \otimes_2 R_T(f^* B) \arrow{d}{\rho_T^{(t)}} && f^* R_S^{(s)}(A \otimes_1 B) \arrow{d}{\theta_f^{(s,t,p)}} \\
			R_T^{(t)}(f^* A \otimes_1 f^* B) \arrow{rr}{m_{1,f}^{(s,t,p)}} && R_T^{(t)}(f^* (A \otimes_1 B))
		\end{tikzcd}
	\end{equation*}
	is commutative. Unwinding the various definitions, we obtain the outer part of the diagram
	\begin{equation*}
		\begin{tikzcd}[font=\small]
			f^* s^* R_Y(s_* A) \otimes_2 f^* s^* R_Y(s_* B) \arrow[equal]{d}  \arrow{r}{m_2} & f^* (s^* R_Y(s_* A) \otimes_2 s^* R_Y(s_* B)) \arrow{d}{m_2} \\
			t^* p^* R_Y(s_* A) \otimes_2 t^* p^* R_Y(s_* B) \arrow{d}{\theta} & f^* s^* (R_Y(s_* A) \otimes_2 R_Y(s_* B))  \arrow{d}{\rho} \\
			t^* R_X(p^* s_* A) \otimes_2 t^* R_X(p^* s_* B) \arrow{d}{\eta} & f^* s^* R_Y(s_* A \otimes_1 s_* B) \arrow{d}{\bar{m}_1} \\
			t^* R_X(t_* t^* p^* s_* A) \otimes_2 t^* R_X(t_* t^* p^* s_* B) \arrow[equal]{d} & f^* s^* R_Y(s_*(A \otimes_1 B)) \arrow[equal]{d} \\
			t^* R_X(t_* f^* s^* s_* A) \otimes_2 t^* R_X(t_* f^* s^* s_* B) \arrow{d}{\epsilon} & t^* p^* R_Y(s_*(A \otimes_1 B)) \arrow{d}{\theta} \\
			t^* R_X(t_* f^* A) \otimes_2 t^* R_X(t_* f^* B) \arrow{d}{m_2} & t^*  R_Y(p^* s_*(A \otimes_1 B)) \arrow{d}{\eta} \\
			t^* (R_X(t_* f^* A) \otimes_2 R_X(t_* f^* B)) \arrow{d}{\rho} & t^*  R_Y(t_* t^* p^* s_*(A \otimes_1 B)) \arrow[equal]{d} \\
			t^* R_X(t_* f^* A \otimes_1 t_* f^* B) \arrow{d}{\bar{m}_1} & t^*  R_Y(t_* f^* s^* s_*(A \otimes_1 B)) \arrow{d}{\epsilon} \\
			t^* R_X(t_* (f^* A \otimes_1 f^* B)) \arrow{r}{m_1} & t^*  R_Y(t_* f^* (A \otimes_1 B))
		\end{tikzcd}
	\end{equation*}
    that we decompose as
    \begin{equation*}
    	\begin{tikzcd}[font=\tiny]
    		\bullet \arrow[equal]{d} \arrow{rrrr}{m_2} &&&& \bullet \arrow{d}{m_2} \\
    		\bullet \arrow{d}{\theta} \arrow{r}{m_2} & t^* (p^* R_Y(s_* A) \otimes_2 p^* R_Y(s_* B)) \arrow{rr}{m_2} \arrow{d}{\theta} && t^* p^* (R_Y(s_* A) \otimes_2 R_Y(s_* B)) \arrow[equal]{r} \arrow{dd}{\rho} & \bullet  \arrow{d}{\rho} \\
    		\bullet \arrow{d}{\eta} \arrow{r}{m_2} & t^* (R_X(p^* s_* A) \otimes_2 R_X(p^* s_* B)) \arrow{dr}{\rho} \arrow{d}{\eta} &&& \bullet \arrow{d}{\bar{m}_1} \\
    		\bullet \arrow[equal]{d} \arrow{r}{m_2} & t^* (R_X(t_* t^* p^* s_* A) \otimes_2 R_X(t_* t^* p^* s_* B)) \arrow{dr}{\rho} \arrow[equal]{dd} & t^* R_X(p^* s_* A \otimes_2 p^* s_* B) \arrow{d}{\eta} \arrow{dr}{m_1} & t^* p^* R_Y(s_* A \otimes_1 s_* B) \arrow{d}{\theta} \arrow[equal]{ur} \arrow{dr}{\bar{m}_1} & \bullet \arrow[equal]{d} \\
    		\bullet \arrow{d}{\epsilon} \arrow{dr}{m_2} && t^* R_X(t_* t^* p^* s_* A \otimes_1 t_* t^* p^* s_* B) \arrow{d}{\bar{m}_1} & t^* R_X (p^* (s_* A \otimes_1 s_* B)) \arrow{d}{\eta} \arrow{ddl}{\eta} \arrow{dr}{\bar{m}_1} & \bullet \arrow{d}{\theta} \\
    		\bullet \arrow{d}{m_2} & t^* (R_X (t_* f^* s^* s_* A) \otimes_2 R_X(t_* f^* s^* s_* B)) \arrow{d}{\rho} \arrow{dl}{\epsilon} & t^* R_X(t_* (t^* p^* s_* A \otimes_1 t^* p^* s_* B)) \arrow{d}{m_1} \arrow[equal]{ddl} & t^* R_X(t_* t^* p^* (s_* A \otimes_1 s_* B)) \arrow[equal]{d} \arrow{dr}{\bar{m}_1} & \bullet \arrow{d}{\eta} \\
    		\bullet \arrow{d}{\rho} & t^* R_X(t_* f^* s^* s_* A \otimes_1 t_* f^* s^* s_* B) \arrow{dl}{\epsilon} \arrow{d}{\bar{m}_1} \arrow[equal]{uur} & t^* R_X(t_* t^* (p^* s_* A \otimes_1 p^* s_* B)) \arrow{ru}{m_1} & t^* R_X (t_* f^* s^* (s_* A \otimes_1 s_* B)) \arrow{dr}{\bar{m}_1} & \bullet \arrow[equal]{d} \\
    		\bullet \arrow{d}{\bar{m}_1} & t^* R_X (t_* (f^* s^* s_* A \otimes_1 f^* s^* s_* B)) \arrow{dl}{\epsilon} \arrow{rr}{m_1} && t^* R_X (t_* f^* (s^* s_* A \otimes_1 s^* s_* B)) \arrow{u}{m_1} \arrow{dr}{\epsilon} & \bullet \arrow{d}{\epsilon} \\
    		\bullet \arrow{rrrr}{m_1} &&&& \bullet
    	\end{tikzcd}
    \end{equation*}
	Here, the central upper six-term piece is commutative by axiom (mor-ETS) over $\Ucal$, the central lower six-term piece is commutative by \cite[Lemma~2.3]{Ter23Fib}, and all the remaining pieces are commutative by naturality and by construction. This proves the claim.
\end{proof}

\subsection*{Extending associative, commutative and unitary morphisms} 

We conclude this section by studying the compatibility of the above construction with internal associativity, commutativity, and unit constraints. As in \cite[\S~8]{Ter23Fib}, we describe these compatibilities as follows:
\begin{enumerate}
	\item Given internal associativity constraint $a_1$ and $a_2$ on $(\otimes_1,m_1)$ and $(\otimes_2,m_2)$, respectively, we say that $\rho$ is \textit{associative} (with respect to $a_1$ and $a_2$) if it satisfies the following additional condition:
	\begin{enumerate}
		\item[(mor-$a$ITS)] For every $S \in \Scal$, the diagram of functors $\Hbb_1(S) \times \Hbb_1(S) \times \Hbb_1(S) \rightarrow \Hbb_2(S)$
		\begin{equation*}
			\begin{tikzcd}
				(R_S(A) \otimes_2 R_S(B)) \otimes_2 R_S(C) \arrow{r}{\rho} \arrow{d}{a_2} & R_S(A \otimes_1 B) \otimes_2 R_S(C) \arrow{r}{\rho} & R_S((A \otimes_1 B) \otimes_1 C) \arrow{d}{a_1} \\
				R_S(A) \otimes_2 (R_S(B) \otimes_2 R_S(C)) \arrow{r}{\rho} & R_S(A) \otimes_2 R_S(B \otimes_1 C) \arrow{r}{\rho} & R_S(A \otimes_1 (B \otimes_1 C))
			\end{tikzcd}
		\end{equation*}
		is commutative.
	\end{enumerate}
    \item Given internal commutativity constraints $c_1$ and $c_2$ on $(\otimes_1,m_1)$ and $(\otimes_2,m_2)$, respectively, we say that $\rho$ is \textit{symmetric} (with respect to $c_1$ and $c_2$) if it satisfies the following additional condition:
    \begin{enumerate}
    	\item[(mor-$c$ITS)] For every $S \in \Scal$, the diagram of functors $\Hbb_1(S) \times \Hbb_1(S) \rightarrow \Hbb_2(S)$
    	\begin{equation*}
    		\begin{tikzcd}
    			R_S(A) \otimes_2 R_S(B) \arrow{r}{\rho} \arrow{d}{c_2} & R_S(A \otimes_1 B) \arrow{d}{c_1} \\
    			R_S(B) \otimes_2 R_S(A) \arrow{r}{\rho} & R_S(B \otimes_1 A)
    		\end{tikzcd}
    	\end{equation*}
    	is commutative.
    \end{enumerate}
    \item Given internal unit constraints $u_1$ and $u_2$ on $(\otimes_1,m_1)$ and $(\otimes_2,m_2)$ with unit sections $\unit^{(1)}$ and $\unit^{(2)}$, respectively, together with an isomorphism $w: R(\unit^{(1)}) \xrightarrow{\sim} \unit^{(2)}$ between section of $\Hbb_2$, we say that $\rho$ is \textit{unitary via $w$} (with respect to $u_1$ and $u_2$) if it satisfies the following additional condition:
    \begin{enumerate}
    	\item[(mor-$u$ITS)] For every $S \in \Scal$, the two diagrams of functors $\Hbb_1(S) \rightarrow \Hbb_2(S)$
    	\begin{equation*}
    		\begin{tikzcd}
    			R_S(A) \otimes_2 R_S(\unit^{(1)}_S) \arrow{r}{w_S} \arrow{d}{\rho} & R_S(A) \otimes_2 \unit^{(2)}_S \arrow{d}{u_{2,r}} \\
    			R_S(A \otimes_1 \unit^{(1)}_S) \arrow{r}{u_{1,r}} & R_S(A)
    		\end{tikzcd}
    		\qquad
    		\begin{tikzcd}
    			R_S(\unit^{(1)}_S) \otimes_2 R_S(B) \arrow{r}{w_S} \arrow{d}{\rho} & \unit^{(2)}_S \otimes_2 R_S(B) \arrow{d}{u_{2,l}} \\
    			R_S(\unit^{(1)}_S \otimes_1 B) \arrow{r}{u_{1,l}} & R_S(B)
    		\end{tikzcd}
    	\end{equation*}
    	are commutative.
    \end{enumerate}
\end{enumerate}
Clearly, each of these three properties is preserves under restriction to the underlying $\Ucal$-fibered categories. In the converse direction, we have the following result:

\begin{lem}\label{lem:ext-acu-rho}
	The following statements hold:
	\begin{enumerate}
		\item Let $a_1$ and $a_2$ be internal associativity constraints on $(\otimes_1,m_1)$ and $(\otimes_2,m_2)$, respectively, and suppose that $\rho$ is associative (with respect to $a_1$ and $a_2$). Then $\tilde{\rho}$ is associative (with respect to the extended  constraints $\tilde{a}_1$ and $\tilde{a}_2$). 
		\item Let $c_1$ and $c_2$ be internal commutativity constraints on $(\otimes_1,m_1)$ and $(\otimes_2,m_2)$, respectively, and suppose that $\rho$ is symmetric (with respect ot $c_1$ and $c_2$). Then $\tilde{\rho}$ is symmetric (with respect to the extended constraints $\tilde{c}_1$ and $\tilde{c}_2$).
		\item Let $u_1$ and $u_2$ be internal unit constraints with unit sections $\unit^{(1)}$ and $\unit^{(2)}$, respectively, and suppose that $\rho$ is unitary via $w: R(\unit^{(1)}) \xrightarrow{\sim} \unit^{(2)}$ (with respect to to $u_1$ and $u_2$). Then $\tilde{\rho}$ is unitary via the extended isomorphism $\tilde{w}: \tilde{R}(\tilde{\unit}^{(1)}) \xrightarrow{\sim} \tilde{\unit}^{(2)}$ (with respect to the extended constraints $\tilde{u}_1$ and $\tilde{u}_2$).
	\end{enumerate}
\end{lem}
\begin{proof}
	\begin{enumerate}
		\item We need to check that condition (mor-$a$ITS) is satisfied: given $S \in \Scal$, we have to show that the diagram of functors $\Hbb_1(S) \times \Hbb_1(S) \times \Hbb_1(S) \rightarrow \Hbb_2(S)$
		\begin{equation*}
			\begin{tikzcd}
				(R_S(A) \otimes_2 R_S(B)) \otimes_2 R_S(C) \arrow{r}{\tilde{\rho}} \arrow{d}{a_2} & R_S(A \otimes_1 B) \otimes_2 R_S(C) \arrow{r}{\tilde{\rho}} & R_S((A \otimes_1 B) \otimes_1 C) \arrow{d}{a_1} \\
				R_S(A) \otimes_2 (R_S(B) \otimes_2 R_S(C)) \arrow{r}{\tilde{\rho}} & R_S(A) \otimes_2 R_S(B \otimes_1 C) \arrow{r}{\tilde{\rho}} & R_S(A \otimes (B \otimes C))
			\end{tikzcd}
		\end{equation*}
		is commutative. To this end, choose an embedding $(Y;s) \in \Emb_{\Ucal}(S)$. Then, by construction, it suffices to show that the diagram
		
			\begin{equation*}
				\begin{tikzcd}[font=\small]
					(R_S^{(s)}(A) \otimes_{2,s} R_S^{(s)}(B)) \otimes_{2,s} R_{S}^{(s)}(C) \arrow{r}{\rho_S^{(s)}} \arrow{d}{a_{2,S}^{(s)}} & R_S^{(s)}(A \otimes_{1,s} B) \otimes_{2,s} R_S^{(s)}(C) \arrow{r}{\rho_S^{(s)}} & R_S^{(s)}((A \otimes_{1,s} B) \otimes_{1,s} C) \arrow{d}{a_{1,S}^{(s)}} \\
					R_S^{(s)}(A) \otimes_{2,s} (R_S^{(s)}(B) \otimes_{2,s} R_{S}^{(s)}(C)) \arrow{r}{\rho_S^{(s)}} & R_S^{(s)}(A) \otimes_{2,s} R_S^{(s)}(B \otimes_{1,s} C) \arrow{r}{\rho_S^{(s)}} & R_S^{(s)}(A \otimes_{1,s} (B \otimes_{1,s} C))
				\end{tikzcd}
			\end{equation*}
		
		is commutative. Unwinding the various definitions (and writing  for notational convenience) we obtain the image under $s^*$ of the diagram
		\begin{equation*}
			\begin{tikzcd}[font=\tiny]
				s_* s^* R_Y(A' \otimes_1 B') \otimes_2 s_* s^* R_Y(C') \arrow{r}{\eta} & s_* s^* R_Y(s_* s^* (A' \otimes_1 B')) \otimes_2 s_* s^* R_Y(C') & R_Y(s_* s^* (A' \otimes_1 B')) \otimes_2 R_Y(C') \arrow{l}{\eta} \arrow{d}{\rho} \\
				s_* s^* (R_Y(A') \otimes_2 R_Y(B')) \otimes_2 s_* s^* R_Y(C') \arrow{u}{\rho} \arrow{d}{\eta} && R_Y(s_* s^* (A'' \otimes_1 B') \otimes_1 C') \arrow{d}{\eta} \\
				s_* s^* (s_* s^* R_Y(A') \otimes_2 s_* s^* R_Y(B')) \otimes_2 s_* s^* R_Y(C') && s_* s^* R_Y(s_* s^* (A' \otimes_1 B') \otimes_1 C') \\
				(s_* s^* R_Y(A') \otimes_2 s_* s^* R_Y(B')) \otimes_2 s_* s^* R_Y(C') \arrow{u}{\eta} \arrow{d}{a} && s_* s^* R_Y((A' \otimes_1 B') \otimes_1 C') \arrow{u}{\eta} \arrow{d}{a} \\
				s_* s^* R_Y(A') \otimes_2 (s_* s^* R_Y(B') \otimes_2 s_* s^* R_Y(C')) \arrow{d}{\eta} && s_* s^* R_Y(A' \otimes_1 (B' \otimes_1 C')) \arrow{d}{\eta} \\
				s_* s^* R_Y(A') \otimes_2 s_* s^* (s_* s^* R_Y(B') \otimes_2 s_* s^* R_Y(C')) && s_* s^* R_Y(A' \otimes_1 s_* s^* (B' \otimes_1 C')) \\
				s_* s^* R_Y(A') \otimes_2 s_* s^* (R_Y(B') \otimes_2 R_Y(C')) \arrow{u}{\eta} \arrow{d}{\rho} && R_Y(A' \otimes_1 s_* s^* (B' \otimes_1 C')) \arrow{u}{\eta} \\
				s_* s^* R_Y(A') \otimes_2 s_* s^* R_Y(B' \otimes_1 C') \arrow{r}{\eta} & s_* s^* R_Y(A') \otimes_2 s_* s^* R_Y(s_* s^* (B' \otimes_1 C')) & R_Y(A') \otimes_2 R_Y(s_* s^* (B' \otimes_1 C')) \arrow{l}{\eta} \arrow{u}{\rho}
			\end{tikzcd}
		\end{equation*}
		where, for notational convenience, we have set $A' := s_* A$, $B' := s_* B$ and $C' := s_* C$. We decompose it as
		\begin{equation*}
		\begin{tikzcd}[font=\small]
			\bullet \arrow{r}{\eta} & \bullet & \bullet \arrow{l}{\eta} \arrow{d}{\rho} \\
			\bullet \arrow{u}{\rho} \arrow{d}{\eta} & R_Y(A' \otimes_1 B') \otimes_2 R_Y(C') \arrow{ul}{\eta} \arrow{ur}{\eta} \arrow{ddr}{\rho} & \bullet \arrow{d}{\eta} \\
			\bullet && \bullet \\
			\bullet \arrow{u}{\eta} \arrow{d}{a} & (R_Y(A') \otimes_2 R_Y(B') \otimes_2 R_Y(C')) \arrow{uu}{\rho} \arrow{l}{\eta} \arrow{uul}{\eta} \arrow{d}{a} & \bullet \arrow{u}{\eta} \arrow{d}{a} \\
			\bullet \arrow{d}{\eta} & R_Y(A') \otimes_2 (R_Y(B') \otimes_2 R_Y(C')) \arrow{dd}{\rho} \arrow{l}{\eta} \arrow{ddl}{\eta} & \bullet \arrow{d}{\eta} \\
			\bullet && \bullet \\
			\bullet \arrow{u}{\eta} \arrow{d}{\rho} & R_Y(A') \otimes_2 R_Y(B' \otimes_1 C') \arrow{dl}{\eta} \arrow{dr}{\eta} \arrow{uur}{\rho} & \bullet \arrow{u}{\eta} \\
			\bullet \arrow{r}{\eta} & \bullet & \bullet \arrow{l}{\eta} \arrow{u}{\rho}
		\end{tikzcd}
		\end{equation*}
	    Here, the central-right diagram is commutative by axiom (mor-$a$ITS) over $\Ucal$, and all the remaining pieces are commutative by naturality and by construction. This proves the claim. 
		\item We need to check that condition (mor-$c$ITS) is satisfied: given $S \in \Scal$, we have to show that the diagram of functors $\Hbb_1(S) \times \Hbb_1(S) \rightarrow \Hbb_2(S)$
		\begin{equation*}
			\begin{tikzcd}
				R_S(A) \otimes_2 R_S(B) \arrow{r}{\tilde{\rho}} \arrow{d}{c_2} & R_S(A \otimes_1 B) \arrow{d}{c_1} \\
				R_S(B) \otimes_2 R_S(A)\arrow{r}{\tilde{\rho}} & R_S(B \otimes_1 A)
			\end{tikzcd}
		\end{equation*}
		is commutative. To this end, choose an embedding $(Y;s) \in \Emb_{\Ucal}(S)$. Then, by construction, it suffices to show that the diagram
		\begin{equation*}
			\begin{tikzcd}[font=\small]
				R_S^{(s)}(A) \otimes_{2,s} R_S^{(s)}(B) \arrow{r}{\rho_S^{(s)}} \arrow{d}{c_{2,S}^{(s)}} & R_S^{(s)}(A \otimes_{1,s} B) \arrow{d}{c_{1,S}^{(s)}} \\
				R_S^{(s)}(B) \otimes_{2,s} R_S^{(s)}(A) \arrow{r}{\rho_S^{(s)}} & R_S^{(s)}(B \otimes_{1,s} A)
			\end{tikzcd}
		\end{equation*}
		is commutative. Unwinding the various definitions, we obtain the image under $s^*$ of the outer part of the diagram
		
			\begin{equation*}
				\begin{tikzcd}[font=\small]
					s_* s^* R_Y(s_* A) \otimes_2 s_* s^* R_Y(s_* B) \arrow{d}{c_2} & R_Y(s_* A) \otimes_2 R_Y(s_* B) \arrow{l}{\eta} \arrow{r}{\rho} \arrow{d}{c_2} & R_Y(s_* A \otimes_1 s_* B) \arrow{r}{\eta} \arrow{d}{c_1} & s_* s^* R_Y(s_* A \otimes_1 s_* B) \arrow{d}{c_1} \\
					s_* s^* R_Y(s_* B) \otimes_2 s_* s^* R_Y(s_* A) & R_Y(s_* B) \otimes_2 R_Y(s_* A) \arrow{l}{\eta} \arrow{r}{\rho} & R_Y(s_* B \otimes_1 s_* A) \arrow{r}{\eta} & s_* s^* R_Y(s_* B \otimes_1 s_* A)
				\end{tikzcd}
			\end{equation*}
		
		where the central square is commutative by axiom (mor-$c$ITS) applied to $\rho$ while the two lateral squares are commutative by naturality.
		\item We need to check that condition (mor-$u$ITS) is satisfied: given $S \in \Scal$, we have to show that the two diagrams of functors $\Hbb_1(S) \rightarrow \Hbb_2(S)$
		\begin{equation*}
			\begin{tikzcd}
				R_S(A) \otimes_2 R_S(\unit^{(1)}_S) \arrow{r}{w_S} \arrow{d}{\tilde{\rho}} & R_S(A) \otimes_2 \unit^{(2)}_S \arrow{d}{u_{2,r}} \\
				R_S(A \otimes_1 \unit^{(1)}_S) \arrow{r}{u_{1,r}} & R_S(A)
			\end{tikzcd}
			\qquad
			\begin{tikzcd}
				R_S(\unit^{(1)}_S) \otimes_2 R_S(B) \arrow{r}{w_S} \arrow{d}{\tilde{\rho}} & \unit^{(2)}_S \otimes_2 R_S(B) \arrow{d}{u_{2,l}} \\
				R_S(\unit^{(1)}_S \otimes_1 B) \arrow{r}{u_{1,l}} & R_S(B)
			\end{tikzcd}
		\end{equation*}
		are commutative. We only treat the case of the left-most diagram; the case of the right-most diagram is analogous. Choose an embedding $(Y;s) \in \Emb_{\Ucal}(S)$. Then, by construction, in order to show the commutativity of the left-most diagram above, it suffices to show that the diagram
		\begin{equation*}
			\begin{tikzcd}[font=\small]
				R_S^{(s)}(A) \otimes_2 R_S^{(s)}(\unit^{(1,s)}_S) \arrow{r}{w_S^{(s)}} \arrow{d}{\rho_S^{(s)}} & R_S^{(s)}(A) \otimes_2 \unit^{(2,s)}_S \arrow{d}{u_{2,r}^{(s)}} \\
				R_S^{(s)}(A \otimes_1 \unit^{(1,s)}_S) \arrow{r}{u_{1,r}^{(s)}} & R_S^{(s)}(A)
			\end{tikzcd}
		\end{equation*}
		is commutative. Unwinding the various definitions, we obtain the more explicit diagram
		\begin{equation*}
			\begin{tikzcd}[font=\small]
				s^* (s_* s^* R_Y(s_* A) \otimes_2 s_* s^* R_Y(s_* s^* \unit^{(1)}_Y)) & s^* (s_* s^* R_Y(s_* A) \otimes_2 s_* s^* R_Y(\unit^{(1)}_Y)) \arrow{l}{\eta} \arrow{r}{w_Y} & s^* (s_* R_Y(s_* A) \otimes_2 s_* s^* \unit^{(2)}_Y) \\
				s^* (R_Y(s_* A) \otimes_2 R_Y(s_* s^* \unit^{(1)}_Y)) \arrow{u}{\eta} \arrow{dd}{\rho} && s^* (s_* s^* R_Y(s_* A) \otimes_2 \unit^{(2)}_Y) \arrow{u}{\eta} \arrow{d}{u_{2,r}} \\
				&& s^* s_* s^* R_Y(s_* A) \arrow{d}{\epsilon} \\
				s^* R_Y (s_* A \otimes_1 s_* s^* \unit^{(1)}_Y) \arrow{d}{\eta} && s^* R_Y (s_* A) \\
				s^* R_Y (s_* s^* (s_* A \otimes_1 s_* s^* \unit^{(1)}_Y)) & s^* R_Y (s_* s^* (s_* A \otimes_1 \unit^{(1)}_Y)) \arrow{l}{\eta} \arrow{r}{u_{1,r}} & s^* R_Y (s_* s^* s_* A) \arrow{u}{\epsilon} 
			\end{tikzcd}
		\end{equation*}
		that we decompose as
		\begin{equation*}
			\begin{tikzcd}[font=\small]
				\bullet & \bullet \arrow{l}{\eta} \arrow{rr}{w_Y} && \bullet \\
				\bullet \arrow{u}{\eta} \arrow{dd}{\rho} & s^* (R_Y(s_* A) \otimes_2 R_Y(\unit^{(1)}_Y)) \arrow{l}{\eta} \arrow{u}{\eta} \arrow{dd}{\rho} \arrow{dr}{w_Y} && \bullet \arrow{u}{\eta} \arrow{d}{u_{2,r}} \\
				&& s^* (R_Y(s_* A) \otimes_2 \unit^{(2)}_Y) \arrow{uur}{\eta} \arrow{ur}{\eta} \arrow{dr}{u_{2,r}} & \bullet \arrow{d}{\epsilon} \\
				\bullet \arrow{d}{\eta} & s^* R_Y(s_* A \otimes_1 \unit^{(1)}_Y) \arrow{l}{\eta} \arrow{d}{\eta} \arrow{rr}{u_{1,r}} && \bullet \\
				\bullet & \bullet \arrow{l}{\eta} \arrow{rr}{u_{1,r}} && \bullet \arrow{u}{\epsilon} 
			\end{tikzcd}
		\end{equation*}
	    Here, the central piece is commutative by axiom (mor-$u$ITS) over $\Ucal$, while all the remaining pieces are commutative by naturality and by construction. This proves the claim.
	\end{enumerate}
\end{proof}

\end{document}